\documentclass[12pt,letterpaper]{amsart}
\usepackage[utf8]{inputenc}
\usepackage[margin=1in,footskip=0.25in]{geometry}
\usepackage{amsmath, amssymb, amsthm}
\usepackage{graphicx}
\usepackage{xcolor}
\usepackage{bm}
\usepackage{hyperref}
\usepackage{subfig}
\usepackage{cite}
\usepackage{mathtools}
\usepackage{cleveref}

\newtheorem{theorem}{Theorem}[section]

\theoremstyle{definition}
\newtheorem{remark}[theorem]{Remark}
\newtheorem{lemma}[theorem]{Lemma}
\newtheorem{definition}[theorem]{Definition}
\newtheorem{proposition}[theorem]{Proposition}
\newtheorem{example}[theorem]{Example}

\newcommand{\scom}{simplicial complex}
\newcommand{\K}{\mathcal{K}}
\newcommand{\lines}{\overline{\mathcal{L}}}
\newcommand\norm[1]{\lVert#1\rVert}

\newcommand\difference[2]{#2}

\newcommand{\pdgm}{\text{PD}}
\newcommand{\base}{B}
\newcommand{\bp}{p}
\newcommand{\homdim}{q}
\newcommand{\idx}{\text{idx}}
\newcommand{\R}{\mathbb{R}}

\title[Persistence Diagram Bundles]{Persistence Diagram Bundles: \\A Multidimensional Generalization of Vineyards}
\author{Abigail Hickok}
\date{\today}

\begin{document}

\maketitle

\begin{abstract}
I introduce the concept of a persistence diagram (PD) bundle, which is the space of PDs for a fibered filtration function (a set $\{f_\bp: \K^\bp \to \mathbb{R}\}_{\bp \in \base}$ of filtrations that is parameterized by a topological space $\base$). Special cases include vineyards, the persistent homology transform, and fibered barcodes for multiparameter persistence modules. I prove that if $\base$ is a smooth compact manifold, then for a generic fibered filtration function, $\base$ is stratified such that within each stratum $Y \subseteq \base$, there is a single PD ``template'' (a list of ``birth'' and ``death'' simplices) that can be used to obtain the PD for the filtration $f_p$ for any $\bp \in Y$. If $\base$ is compact, then there are finitely many strata, so the PD bundle for a generic fibered filtration on $\base$ is determined by the persistent homology at finitely many points in $\base$. I also show that not every local section can be extended to a global section (a continuous map $s$ from $\base$ to the total space $E$ of PDs such that $s(\bp) \in \pdgm(f_\bp)$ for all $\bp \in \base$). Consequently, a PD bundle is not necessarily the union of ``vines'' $\gamma: \base \to E$; this is unlike a vineyard. When there is a stratification as described above, I construct a cellular sheaf that stores sufficient data to construct sections and determine whether a given local section can be extended to a global section.
\end{abstract}

\section{Introduction}\label{sec:intro}
\difference{}{In topological data analysis (TDA), our aim is to understand the global shape of a data set. Often, the data set takes the form of a collection of points in $\mathbb{R}^n$, called a \textit{point cloud}, and we hope to analyze the topology of a lower-dimensional space that the points lie on.
TDA has found applications in a variety of fields, such as biology \cite{biology}, neuroscience \cite{cortical}, and chemistry \cite{cyclo}.

We use \emph{persistent homology} (PH), a tool from algebraic topology \cite{eat}. The first step of persistent homology is to construct a filtered complex from our data; a \emph{filtered complex} is a nested sequence
\begin{equation}\label{eq:filtration}
    \K_{r_0} \subseteq \K_{r_1} \subseteq \cdots \subseteq \K_{r_n} \subseteq \cdots
\end{equation}
of simplicial complexes. For example, one of the standard ways to build a filtered complex from point cloud data is to construct the Vietoris--Rips filtered complex.
At filtration-parameter value $r$, the Vietoris--Rips complex $\K_r$ includes a simplex for every subset of points within $r$ of each other. In persistent homology, one studies how the topology of $\K_r$ changes as the filtration parameter-value $r$ increases. As $r$ grows, new homology classes (which represent ``holes'' in the data) are born and old homology classes die. One way of summarizing this information is a \emph{persistence diagram}: a multiset of points in the extended plane $\overline{\mathbb{R}}^2$. If there is a homology class that is born at filtration-parameter value $b$ and dies at filtration-parameter value $d$, the persistence diagram contains the point $(b, d)$.}

Developing new methods for analyzing how the topology of a data set changes as multiple parameters vary is a very active area of research \cite{multi_review}. For example, if a point cloud evolves over time (i.e., it is a dynamic metric space), then one maybe interested in using time as a second parameter, in addition to the filtration parameter $r$. Common examples of time-evolving point clouds include swarming or flocking animals whose positions and/or velocities are represented by points (\cite{CorcoranJones, crocker, kim_memoli}). In such cases, one can obtain a filtered complex $\K_{r_0}^t \subseteq \K_{r_1}^t \subseteq \cdots \subseteq K_{r_n}^t$ at every time $t$ by constructing, e.g., the Vietoris--Rips filtered complex for the point cloud at time $t$. It is also common to use the density of the point cloud as a parameter (\cite{top_regimes, dtm_initial, kde_sublevel}). Many other parameters can also vary in the topological analysis of point clouds or other types of data sets.

One can use a \difference{vineyard}{\emph{vineyard}} \cite{vineyards} to study a 1-parameter family of filtrations $\{\K_{r_0}^t \subseteq \K_{r_1}^t \subseteq \cdots \K_{r_n}^t\}_{t \in \mathbb{R}}$ such as that obtained from a time-varying point cloud. \difference{}{At each $t \in \mathbb{R}$, one can compute the PH of the filtration $\K_{r_0}^t \subseteq \K_{r_1}^t \subseteq \cdots \subseteq \K_{r_n}^t$ and obtain a persistence diagram PD$(t)$. A vineyard is visualized as the continuously-varying ``stack of PDs'' $\{\textnormal{PD}(t)\}_{t \in \mathbb{R}}$. See Figure \ref{fig:vineyard} for an illustration. As $t \in \mathbb{R}$ varies, the points in the PDs trace out curves (``vines'') in $\mathbb{R}^3$. Each vine corresponds to a homology class (i.e., one of the holes in the data), and shows how the persistence of that homology class changes with time (or, more generally, as some other parameter varies).
\begin{figure}
	\centering
	\includegraphics[width = .4\textwidth]{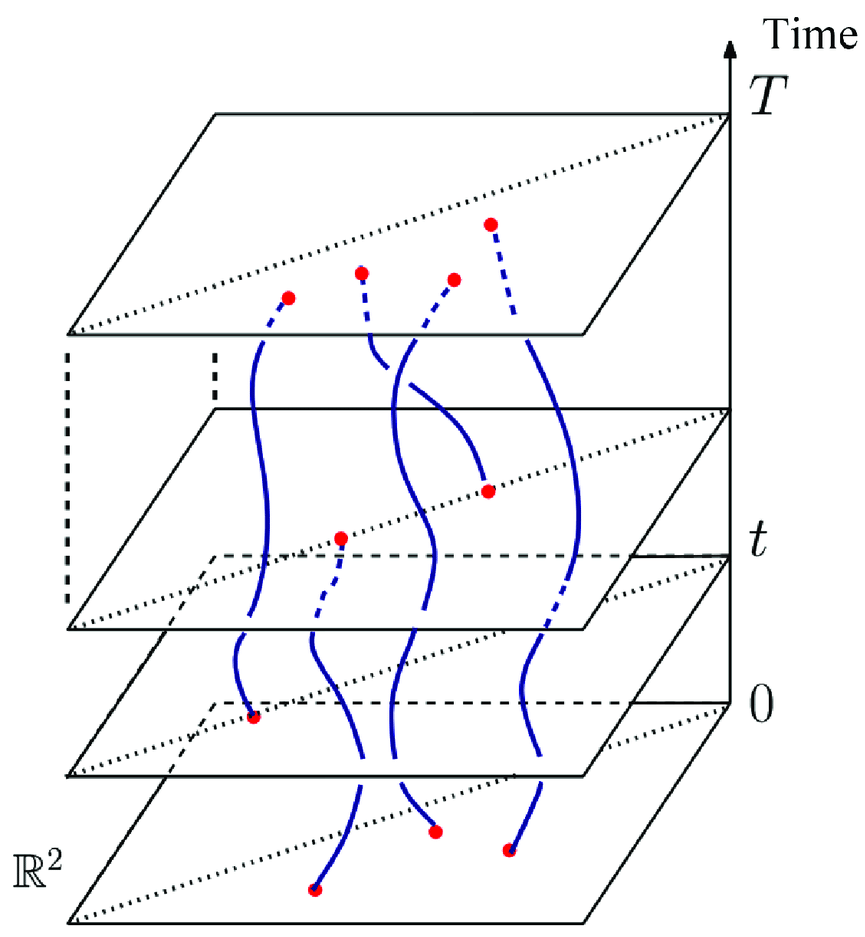}
	\caption{An example of a vineyard. There is a persistence diagram for each time $t$. Each curve is a vine in the vineyard. (This figure is a slightly modified version of a figure that appeared originally in \cite{vineyard_figure}.)}
	\label{fig:vineyard}
\end{figure}
}
However, one cannot use a vineyard for a set of filtrations that is parameterized by a space that is not a subset of $\mathbb{R}$. For example, suppose that we have a time-varying point cloud whose dynamics depend on some system-parameter values $\mu_1, \ldots, \mu_m \in \mathbb{R}$. Many such systems exist. For example, the D'Orsogna model is a multi-agent dynamical system that models attractive and repulsive interactions between particles \cite{dorsogna}.
Each particle is represented by a point in a point cloud.
In certain parameter regimes, there are interesting topological features, such as mills or double mills \cite{topaz2015}. For each time $t \in \mathbb{R}$ and for each $\mu_1, \ldots, \mu_m \in \mathbb{R}$, one can obtain a filtered complex
\begin{equation}\label{eq:param filtration}
    \K_{r_0}^{t, \mu_1, \ldots, \mu_m} \subseteq \K_{r_1}^{t, \mu_1, \ldots, \mu_m} \subseteq \cdots \subseteq \K_{r_n}^{t, \mu_1, \ldots, \mu_m}
\end{equation}
by constructing, e.g., the Vietoris--Rips filtered complex for the point cloud at time $t$ at system-parameter values $\mu_1, \ldots, \mu_m$. A parameterized set of filtered complexes like the one in \eqref{eq:param filtration} cannot be studied using a vineyard for the simple reason that there are too many parameters.

Such a parameterized set of filtered complexes cannot be studied using multiparameter PH \cite{multi}, either. A \emph{multifiltration} is a set $\{\K_{\bm{u}}\}_{\bm{u} \in \mathbb{R}^n}$ of simplicial complexes such that $\K_{\bm{u}} \subseteq \K_{\bm{v}}$ whenever $\bm{u} \leq \bm{v}$. \emph{Multiparameter PH} is the $\mathbb{F}[x_1, \ldots, x_n]$ module obtained by applying homology (over a field $\mathbb{F}$) to a multifiltration. (For more details, see reference \cite{multi}.) The parameterized set of filtered complexes in \eqref{eq:param filtration} is not typically a multifiltration because it is not necessarily the case that $\K_{r_i}^{t, \mu_1, \ldots, \mu_m} \not\subseteq \K_{r_i}^{t', \mu_1', \ldots, \mu_m'}$ for all values of $t$, $t'$, $\{\mu_i\}$, and $\{\mu_i'\}$. Not only is there not necessarily an inclusion $\K_{r_i}^{t, \mu_1, \ldots, \mu_m} \xhookrightarrow{} \K_{r_i}^{t', \mu_1', \ldots, \mu_m'}$, but there is not any given simplicial map $\K_{r_i}^{t, \mu_1, \ldots, \mu_m} \to  \K_{r_i}^{t', \mu_1', \ldots, \mu_m'}$. Therefore, we cannot use multiparameter PH.

In what follows, we will work with a slightly different notion of filtered complex than that of \eqref{eq:filtration}. A \emph{filtration function} is a function $f: \K \to \mathbb{R}$, where $\K$ is a simplicial complex, such that every sublevel set $\K_r := \{\sigma \in \K \mid f(\sigma) \leq r\}$ is a simplicial complex (i.e., $f(\tau) \leq f(\sigma)$ if $\tau$ is a face of $\sigma$). For every $r \leq s$, we have that $\K_r \subseteq \K_s$. A simplex $\sigma \in \K$ appears in the filtration at $r = f(\sigma)$. By setting $\{r_i\} = \textnormal{Im}(f)$, where $r_i < r_{i+1}$, we obtain a nested sequence as in \eqref{eq:filtration}. Conversely, given a nested sequence of simplicial complexes, the associated filtration function is $f(\sigma) = \min\{r_i \mid \sigma \in \K_{r_i}\}$, with $\K = \bigcup_i \K_{r_i}$.

\subsection{Contributions}
I introduce the concept of a \emph{persistence diagram (PD) bundle}, in which PH varies over an arbitrary ``base space'' $\base$. A PD bundle gives a way of studying a \emph{fibered filtration function}, which is a set $\{f_\bp : \K^\bp \to \mathbb{F}\}_{\bp \in \base}$ of functions such that $f_\bp$ is a filtration of a simplicial complex $\K^\bp$.
At each $\bp \in \base$, the sublevel sets of $f_\bp$ form a filtered complex. For example, in \eqref{eq:param filtration}, we have $\base = \mathbb{R}^{n+1}$ and we obtain a fibered filtration function $\{f_{t, \mu_1, \ldots, \mu_m}: \K \to \mathbb{R}\}_{(t, \mu_1, \ldots, \mu_m) \in \mathbb{R}^{m+1}}$ by defining $f_{t, \mu_1, \ldots, \mu_m}$ to be the filtration function associated with the filtered complex in \eqref{eq:param filtration}. The associated PD bundle is the space of persistence diagrams PD$(f_\bp)$ as they vary with $\bp \in \base$ (see Definition \ref{def:pdbundle}). In the special case in which $\base$ is an interval in $\mathbb{R}$, a PD bundle is equivalent to a vineyard.

I prove that for ``generic'' fibered filtration functions (see Section \ref{sec:generic}), the base space $\base$ can be stratified in a way that makes PD bundles tractable to compute and analyze. Theorem \ref{thm:stratified} says that for a ``generic'' fibered filtration function on a smooth compact manifold $\base$, the base space $\base$ is stratified such that within each stratum, there is a single PD ``template'' that can be used to obtain $PD(f_\bp)$ at any point $\bp$ in the stratum. Proposition \ref{prop:polyhedrons_constant} shows that \emph{all} ``piecewise-linear'' PD bundles (see Definition \ref{def:PL_pdbundle}) have such a stratification. The template is a list of (birth, death) simplex pairs, and the diagram $\pdgm(f_\bp)$ is obtained by evaluating $f_\bp$ on each simplex. In particular, when $\base$ is a smooth compact manifold, the number of strata is finite, so the PD bundle is determined by the PH at a finite number of points in the base space.

I show that unlike vineyards, PD bundles 
do not necessarily decompose into a union of ``vines''. More precisely, there may not exist continuous maps $\gamma_1, \ldots, \gamma_m : \base \to E$ such that
\begin{equation}\label{eq:vine_decomp}
    \pdgm(f_\bp) = \bigcup_{i=1}^m \gamma_i(\bp)
\end{equation}
for all $\bp \in \base$. This is a consequence of \Cref{prop:monodromy}, in which it is shown that nontrivial global sections are not guaranteed to exist. That is, given a point $z_0 \in PD(f_{\bp_0})$ for some $\bp_0 \in \base$, it may not be possible to extend $\bp_0 \mapsto z_0$ to a continuous map $s : \base \to E:= \{(\bp, z) \mid z \in PD(f_\bp)\}$ such that $s(\bp) \in \pdgm(f_\bp)$ for all $\bp$. This behavior is a feature that gives PD bundles a richer mathematical structure than vineyards.

For any fibered filtration with a stratification as described above (see Theorem \ref{thm:stratified} and Proposition \ref{prop:polyhedrons_constant}), I construct a ``compatible cellular sheaf'' (see Section \ref{sec:cellularsheaf}) that stores the data in the PD bundle. Rather than analyzing the entire PD bundle, which consists of continuously varying PDs over the base space $\base$, we can analyze the cellular sheaf, which is discrete. For example, in Proposition \ref{prop:discrete_secs}, I prove that an extension of $\bp_0 \mapsto z_0$ to a global section exists if a certain associated global section of the cellular sheaf exists. A compatible cellular sheaf stores sufficient data to reconstruct the associated PD bundle and analyze its sections.

Though not the focus of this paper, I also give a simple example of vineyard instability in Appendix \ref{sec:vineyard_unstable}. It is often quoted in the research literature that ``vineyards are unstable''; however, this ``well-known fact'' has been shared only in private correspondence and, to the best of my knowledge, has never been published. The example of vineyard instability is furnished from an example in Proposition \ref{prop:monodromy}.

\subsection{Related work}
PD bundles are a generalization of vineyards, which were introduced in \cite{vineyards}. Two other important special cases of PD bundles are the fibered barcode of a multiparameter persistence module \cite{fibered} and the persistent homology transform ($\base = S^n$) from shape analysis \cite{elevation, pht}. I discuss the special case of fibered barcodes in detail in Section \ref{sec:fibered}; the base space $\base$ is a subset of the space of lines in $\mathbb{R}^n$. The persistent homology transform (PHT) is defined for a constructible set $M \subseteq \mathbb{R}^{n+1}$. For any unit vector $v \in S^n$, one defines the filtration $M_r^v = \{x \in M \mid x \cdot v \leq r\}$ (i.e., the sublevel filtration of the height function with respect to the direction $v$). PHT is the map that sends $v \in S^n$ to the persistence diagram for the filtration $\{M_r^v\}_{r \in \mathbb{R}}$. The significance of PHT is that it is a sufficient statistic for shapes in $\mathbb{R}^2$ and $\mathbb{R}^3$ \cite{pht}. Applications of PHT are numerous and include protein docking \cite{protein_docking}, barley-seed shape analysis \cite{barley}, and heel-bone analysis in primates \cite{pht}.

For PHT, Curry et al. \cite{pht_finite} proved that the base space $S^n$ is stratified such that the PHT of a shape $M$ is determined by the PH of $\{M_r^v\}_{r \in \mathbb{R}}$ for finitely many directions $v \in S^n$ (one direction $v$ per stratum). This is related to the stratification given by Theorem \ref{thm:stratified}, in which I show that a ``generic'' PD bundle whose base space $\base$ is a compact smooth manifold (such as $S^n$) is similarly stratified and thus determined by finitely many points in $\base$ (one $\bp \in \base$ per stratum). The primary difference between the stratifications in \cite{pht_finite} and Theorem \ref{thm:stratified} is that in \cite{pht_finite}, each stratum is a subset in which the order of the vertices of a triangulated shape $M$ (as ordered by the height function) is constant, whereas in Theorem \ref{thm:stratified}, each stratum is a subset in which the order of the simplices (as ordered by the filtration function) is constant. The stratification result of the present paper (\Cref{thm:stratified}) applies to general PD bundles, while \cite{pht_finite} applies only to PHT.

The stratification that we study in the present \difference{chapter}{paper} is used in \cite{pd_bundle_alg} to develop an algorithm for computing ``piecewise-linear'' PD bundles (see Definition \ref{def:PL_pdbundle}). The algorithm relies on the fact that for any piecewise-linear PD bundle on a compact triangulated base space $\base$, there are a finite number of strata, so the PD bundle is determined by the PH at a finite number of points in $\base$.

The existence (or nonexistence) of nontrivial global sections in PD bundles is related to the study of ``monodromy'' in fibered barcodes of multiparameter persistence modules \cite{monodromy}. Cerri et al. \cite{monodromy} constructed an example in which there is a path through the fibered barcode that loops around a ``singularity'' (a PD in the fibered barcode for which there is a point in the PD with multiplicity greater than \difference{$1$}{one}) and finishes in a different place than where it starts.

\subsection{Organization}
\difference{This chapter}{This paper} proceeds as follows.\difference{}{ I review background on persistent homology in Section \ref{sec:background}.} In Section \ref{sec:def}, I give the definition of a PD bundle, with some examples, and I compare PD bundles to multiparameter PH. In Section \ref{sec:partition}, I show how to stratify the base space $\base$ into strata in which the (birth, death) simplex pairs are constant (see Theorem \ref{thm:stratified} and Proposition \ref{prop:polyhedrons_constant}). I discuss sections of PD bundles and the existence of monodromy in Section \ref{sec:monodromy}. I construct a compatible cellular sheaf in Section \ref{sec:cellularsheaf}.
I conclude and discuss possible directions for future research in Section \ref{sec:conclusion}. In Appendix \ref{sec:vineyard_unstable}, I use the example of monodromy from Section \ref{sec:monodromy} to construct an example of vineyard instability. In Appendix \ref{sec:bundle_details}, I provide technical details that are needed to prove Theorem \ref{thm:stratified}.


\section{Background}\label{sec:background}
We begin by reviewing persistent homology (PH) and cellular sheaves. For a more thorough treatment of PH, see \cite{edel_book, roadmap}, and for more on cellular sheaves, see \cite{spectral, curry}.

\subsection{Filtrations}
Consider a simplicial complex $\K$. A \emph{filtration function} on $\K$ is a real-valued function $f: \K \to \mathbb{R}$ such that if $\tau \in \K$ is a face of $\sigma \in K$, then $f(\tau) \leq f(\sigma)$. The \emph{filtration value} of a simplex $\sigma \in \K$ is $f(\sigma)$. The $r$-sublevel sets $\K_r := \{\sigma \in \K \mid f(\sigma) \leq r\}$
form a \emph{filtered complex}. The condition that $f(\tau) \leq f(\sigma)$ if $\tau \subseteq \sigma$ guarantees that $\K_r$ is a \scom\ for all $r$. For all $s \leq r$, we have $\K_s \subseteq \K_r$. The parameter $r$ is the \emph{filtration parameter}. For an example, see Figure \ref{fig:fsc_example}.

\begin{figure}
    \centering
    \subfloat[$\K_0$]{\includegraphics[width = .17\linewidth]{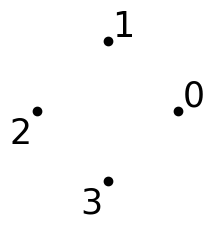}}
    \hspace{0.025\linewidth}
    \subfloat[$\K_1$]{\includegraphics[width = .17\linewidth]{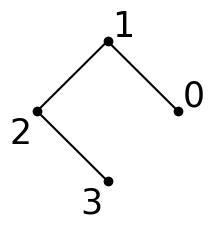}}
    \hspace{0.025\linewidth}
    \subfloat[$\K_2$]{\includegraphics[width = .17\linewidth]{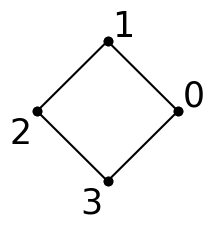}}
    \hspace{0.025\linewidth}
    \subfloat[$\K_3$]{\includegraphics[width = .17\linewidth]{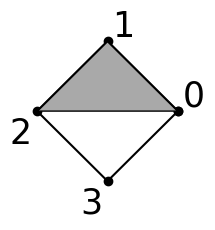}}
    \hspace{0.025\linewidth}
    \subfloat[$\K_4$]{\includegraphics[width = .17\linewidth]{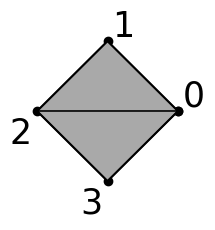}}
    \caption{An example of a filtration. The simplicial complex $\K_i$ has the associated filtration-parameter value $i$. (This figure appeared originally in \cite{covid}.)}
    \label{fig:fsc_example}
\end{figure}

For example, suppose that $X = \{x_i\}_{i=1}^M$ is a point cloud. Let $\K$ be the simplicial complex that has a simplex $\sigma$ with vertices $\{x_j\}_{j \in J}$ for all $J \subseteq \{1, \ldots, M\}$. The \emph{Vietoris--Rips filtration function} is $f(\sigma) = \frac{1}{2} \max_{j, k \in J} \{ \norm{x_j - x_k}\}$, where $\{x_j\}_{j \in J}$ are the vertices of $\sigma$.


\subsection{Persistent homology}\label{sec:ph}
Let $f: \K \to \mathbb{R}$ be a filtration function on a finite simplicial complex $\K$, and let $\{K_r\}_{r \in \mathbb{R}}$ be the associated filtered complex. Let $r_1 < \cdots < r_N$ be the filtration values of the simplices of $\K$. These are the critical points at which simplices are added to the filtration. For all $s \in [r_i, r_{i+1})$, we have $\K_s = \K_{r_i}$.

For all $i \leq j$, the inclusion $\iota^{i, j}: \K_{r_i} \hookrightarrow \K_{r_j}$ induces a map $\iota^{i, j}_*: H_*(\K_{r_i}, \mathbb{F}) \to H_*(\K_{r_j}, \mathbb{F})$ on homology, where $\mathbb{F}$ is a field that we set to $\mathbb{Z}/2\mathbb{Z}$ for the rest of this paper. The \emph{$\homdim$th-persistent homology} (PH) is the pair 
\begin{equation*}
    \Big( \{H_\homdim(\K_{r_i}, \mathbb{F})\}_{1 \leq i \leq N}\,, \{\iota_*^{i, j}\}_{1 \leq i \leq j \leq N} \Big)\,.
\end{equation*}
The Fundamental Theorem of Persistent Homology yields compatible choices of bases for the vector spaces $H_\homdim(\K_{r_i}, \mathbb{F})$, which we use below in our definition of a persistence diagram.

The \emph{$\homdim$th persistence diagram} $PD_\homdim(f)$ is a multiset of points in the extended plane $\overline{\mathbb{R}}^2$ that summarizes the $\homdim$th-persistent homology. The PD contains the diagonal as well as a point for every generator. We say that a generator $\gamma \in H_\homdim(\K_{r_i}, \mathbb{F})$ is \emph{born} at $r_i$ if it is not in the image of $\iota^{i, i-1}_*$. The homology class $\gamma$ subsequently \emph{dies} at $r_j > r_i$ if $\iota_*^{i, j}(\gamma) = 0$ and $\iota_*^{i, j-1}(\gamma) \neq 0$. If $\iota_*^{i, j}(\gamma) \neq 0$ for all $j > i$, then $\gamma$ never dies. For every generator, the PD contains the point $(r_i, r_j)$ if the generator is born at $r_i$ and dies at $r_j$, or else contains the point $(r_i, \infty)$ if the generator is born at $r_i$ and never dies.

\subsection{Birth and death simplex pairs} Computing persistent homology can be reduced to computing the set of ``birth'' and ``death'' simplices for the generating homology classes. Informally, a \emph{birth simplex} $\sigma_b$ is a $\homdim$-simplex that creates a new $\homdim$-dimensional homology class when it is added to the filtration and a \emph{death simplex} is a $(\homdim+1)$-simplex that destroys a $\homdim$-dimensional homology class when it is added to the filtration. For example, in Figure \ref{fig:fsc_example}, the 1D PH has one generator. Its birth simplex is the $1$-simplex $(0, 3)$ and its death simplex is the $2$-simplex $(0, 2, 3)$. For every pair $(\sigma_b, \sigma_d)$ of (birth, death) simplices, the persistence diagram contains the point $(f(\sigma_b), f(\sigma_d))$. For every unpaired birth simplex $\sigma_b$, the persistence diagram contains the point $(f(\sigma_b), \infty)$.

In \cite{ph_alg}, Edelsbrunner and Harer presented an algorithm for computing the (birth, death) simplex pairs of a filtration $f: \K \to \mathbb{R}$. Let $\sigma_1, \ldots, \sigma_N$ be the simplices of $\K$, indexed such that $i < j$ if $\sigma_i$ is a proper face of $\sigma_j$. 

\begin{definition}\label{def:spx_ordering}
    The \emph{simplex order} induced by $f$ is the strict partial order $\prec_f$ on $\K$ such that $\sigma_i \prec_f \sigma_j $ if and only if $f(\sigma_i) < f(\sigma_j)$.
\end{definition}
\noindent If the simplex orders $\prec_{f_1}, \prec_{f_2}$ induced by filtrations $f_1, f_2$ (respectively) are the same, then $f_1(\sigma_i) < f_1(\sigma_j)$ if and only if $f_2(\sigma_i) < f_2(\sigma_j)$ and $f_1(\sigma_i) = f_1(\sigma_j)$ if and only if $f_2(\sigma_i) = f_2(\sigma_j)$.

The algorithm of \cite{ph_alg} requires a \emph{compatible simplex indexing}.
\begin{definition}\label{def:spx_indexing}
   A \emph{compatible simplex indexing} is a function $\idx : \K \to \{1, \ldots, N\}$ such that $\idx(\sigma_i) < \idx(\sigma_j)$ if $\sigma_i \prec_f \sigma_j$ or $\sigma_i$ is a proper face of $\sigma_j$. Because a compatible simplex indexing may not be unique, we fix the \emph{simplex indexing induced by $f$} to be the unique function $\idx_f: \K \to \{1, \ldots, N \}$ such that $\idx_f(\sigma_i) < \idx_f(\sigma_j)$ if either $\sigma_i \prec_f \sigma_j$ or if $f(\sigma_i) = f(\sigma_j)$ and $i < j$.
\end{definition}
\noindent The function $\idx_f$ is a compatible simplex indexing because if $\sigma_i$ is a proper face of $\sigma_j$, then $i < j$ and $f(\sigma_i) \leq f(\sigma_j)$. The sequence $\idx_f^{-1}(1), \ldots, \idx_f^{-1}(N)$ of simplices is ordered by the value of $f$ on each simplex, with ties broken by the order of the simplices in the sequence $\sigma_1, \ldots, \sigma_N$. The indexing $\idx_f$ is defined such that if we define $\K'_j := \{\sigma \in \K \mid \idx_f(\sigma) \leq j\}$, then
\begin{equation*}
    \K_1' \subseteq \K_2' \subseteq \cdots \subseteq \K'_N
\end{equation*}
is a nested sequence of simplicial complexes and if $r_i = f(\sigma_{j_1}) = \cdots = f(\sigma_{j_k})$, where $j_1 < \cdots < j_k$ and $\{r_i\} = \text{Im}(f)$, with $r_i < r_{i+1}$, then
\begin{equation*}
    \K_{r_i} = \K'_{j_1} \subset \K'_{j_2} \subset \cdots \subset \K_{j_k}' \subset \K_{r_{i+1}}\,.
\end{equation*}
In other words, $\{\K'_j\}$ is a refinement of $\{\K_{r_i}\}$.

The following lemma is a straightforward corollary of the work in \cite{edel_book}, and we will rely on it repeatedly in the present paper.

\begin{lemma}[\cite{edel_book}]\label{lem:only_order}
If $f_0, f_1: \K \to \mathbb{R}$ are two filtration functions such that $\prec_{f_0}$ is the same as $\prec_{f_1}$, then $\idx_{f_0} = \idx_{f_1}$ and $f_1$ and $f_2$ both induce the same set of (birth, death) simplex pairs.
\end{lemma}

\subsection{Updating persistent homology when the simplex indexing is updated}\label{sec:updates}
One of the main contributions of \cite{vineyards}, in which vineyards were introduced, is an algorithm for updating the (birth, death) simplex pairs when the simplex indexing changes. We review the relevant details in this subsection.

Suppose that $\idx_{f_0}, \idx_{f_1} : \K \to \{1, \ldots, N\}$ are the simplex indexings that are induced by filtrations $f_0, f_1 : \K \to \mathbb{R}$ (respectively), and suppose that $\idx_{f_0}$ and $\idx_{f_1}$ differ only by a transposition of a pair $(\sigma, \tau)$ of consecutive simplices. That is, $\idx_{f_0}(\tau) = \idx_{f_0}(\sigma) + 1$, $\idx_{f_1}(\tau) = \idx_{f_0}(\sigma)$, and $\idx_{f_1}(\sigma) = \idx_{f_0}(\tau)$. Let $S_{\idx_{f_0}}$ and $S_{\idx_{f_1}}$ be the sets of (birth, death) simplex pairs for $f_0$ and $f_1$, respectively.\footnote{Recall that, by Lemma~\ref{lem:only_order}, the pairs depend only on the simplex orders $\idx_{f_0}$, $\idx_{f_1}$, which is why we label the sets by their associated simplex indexings.} The update rule of \cite{vineyards} gives us bijection $\phi^{\idx_{f_0},\, \idx_{f_1}}: S_{\idx_{f_0}} \to S_{\idx_{f_1}}$.

We review the key properties of the bijection $\phi^{\idx_{f_0}, \,\idx_{f_1}}$. We write
\begin{equation*}
    \phi^{\idx_{f_0},\, \idx_{f_1}} = (\phi_b^{\idx_{f_0},\, \idx_{f_1}}, \phi_d^{\idx_{f_0},\, \idx_{f_1}})\,,
\end{equation*}
where $\phi^{\idx_{f_0},\, \idx_{f_1}}_b: S_{\idx_{f_0}} \to \K$ maps a simplex pair $(\sigma_b, \sigma_d) \in S_{\idx_{f_0}}$ to the birth simplex of $\phi^{\idx_{f_0},\, \idx_{f_1}}((\sigma_b, \sigma_d))$ and $\phi^{\idx_{f_0},\, \idx_{f_1}}_d: S_{\idx_{f_0}} \to \K$ maps $(\sigma_b, \sigma_d) \in S_{\idx_{f_0}}$ to the death simplex of $\phi^{\idx_{f_0},\, \idx_{f_1}}((\sigma_b, \sigma_d))$. If $(\sigma_b, \sigma_d) \in S_{\idx_{f_0}}$ is a pair such that $\sigma_b, \sigma_d \not\in \{\sigma, \tau\}$, then $\phi^{\idx_{f_0},\, \idx_{f_1}}((\sigma_b, \sigma_d)) = (\sigma_b, \sigma_d)$. If $(\sigma_b^1, \sigma_d^1) \in S_{\idx_{f_0}}$ is the pair that contains $\sigma$, then let $\lambda \in \{b, d\}$ be the index such that $\sigma^1_{\lambda} = \sigma$. Similarly, if $(\sigma_b^2, \sigma_d^2) \in S_{\idx_{f_0}}$ is the pair that contains $\tau$, then let $\mu\in \{b, d\}$ be the index such that $\sigma^1_{\mu} = \tau$. The key fact about the update rule of \cite{vineyards} is that $\phi^{\idx_{f_0} ,\, \idx_{f_1}}$ is defined such that either 
\begin{align*}
    \phi^{\idx_{f_0},\, \idx_{f_1}}((\sigma^1_b, \sigma^2_d)) &= (\sigma^1_b, \sigma^1_d)\,, \\
    \phi^{\idx_{f_0},\, \idx_{f_1}}((\sigma^2_b, \sigma^2_d)) &= (\sigma^2_b, \sigma^2_d)\,,
\end{align*}
or
\begin{align*}
    \phi^{\idx_{f_0},\, \idx_{f_1}}_{\lambda}((\sigma^1_b, \sigma^1_d)) &= \tau\,, \qquad \phi^{\idx_{f_0}, \idx_{f_1}}_{\lambda^c}((\sigma^1_b, \sigma^1_d)) = \sigma^1_{\lambda^c}\,, \\
    \phi^{\idx_{f_0},\, \idx_{f_1}}_{\mu}((\sigma^2_b, \sigma^2_d)) &= \sigma\,, \qquad \phi^{\idx_{f_0},\, \idx_{f_1}}_{\mu^c}((\sigma^2_b, \sigma^2_d)) = \sigma^1_{\mu^c}\,,
\end{align*}
where 
\begin{equation*}
    \lambda^c := \begin{cases} b\,, & \lambda = d\\ d \,, & \lambda = b\end{cases} \qquad \mu^c := \begin{cases} b\,, & \mu = d\\ d \,, & \mu = b\,.\end{cases}
\end{equation*}
In other words, either $\phi^{\idx_{f_0},\, \idx_{f_1}}$ is the identity map or $\phi^{\idx_{f_0},\, \idx_{f_1}}$ swaps $\sigma$ and $\tau$ in the pairs that contain them. The particular case depends on the order of $f_0(\sigma_b^1), f_0(\sigma_d^1), f_0(\sigma_b^2), f_0(\sigma_d^2)$ (see \cite{vineyards} for details; they are not relevant to the present \difference{thesis}{paper}).

More generally, suppose that $\idx_{f_0}$, $\idx_{f_1}$ are the simplex indexings induced by any two filtrations $f_0, f_1$, where $\idx_{f_0}$ and $\idx_{f_1}$ are no longer required to differ only by the transposition of two consecutive simplices. Let $S_{\idx_{f_0}}$ and $S_{\idx_{f_1}}$ be the sets of (birth, death) simplex pairs for $f_0$ and $f_1$, respectively. The update rule of Cohen-Steiner et al. \cite{vineyards} defines a bijection $\phi^{\idx_{f_0},\, \idx_{f_1}} : S_{\idx_{f_0}} \to S_{\idx_{f_1}}$ as follows. Every permutation can be decomposed into a sequence of transpositions that transpose consecutive elements, so there is a sequence $\zeta_0, \ldots, \zeta_m$ of simplex indexings such that $\zeta_0 = \idx_{f_0}$, $\zeta_m = \idx_{f_1}$, and $\zeta_i$, $\zeta_{i+1}$ differ only by the transposition of two consecutive simplices. Cohen-Steiner et al. \cite{vineyards} defined
\begin{equation}\label{eq:bijection_def}
    \phi^{\idx_{f_0},\, \idx_{f_1}} := \phi^{\zeta_{m-1},\, \zeta_m} \circ \cdots \circ \phi^{\zeta_0,\, \zeta_1}\,.
\end{equation}

\begin{remark}\label{rmk:nonunique}
    If $\idx_{f_0}$, $\idx_{f_1}$ do not differ by only the transposition of two consecutive simplices, then the sequence $\zeta_0, \ldots, \zeta_m$ is not unique. Unfortunately, the definition of $\phi^{\idx_{f_0},\, \idx_{f_1}}$ does depend on the sequence $\zeta_0, \ldots, \zeta_m$ in its definition. This is implicitly shown in \Cref{prop:monodromy}.
    \end{remark}

\subsection{Cellular Sheaves}

\difference{References on cellular sheaves include \cite{spectral, curry}.}{} A \emph{cell complex} is a topological space $Y$ with a partition into a set $\{Y_{\alpha}\}_{\alpha \in P_Y}$ of subspaces (the \emph{cells} of the cell complex) that satisfy the following conditions:
\begin{enumerate}
    \item Every cell $Y_{\alpha}$ is homeomorphic to $\mathbb{R}^{k_{\alpha}}$ for some $k_{\alpha} \geq 0$. The cell $Y_{\alpha}$ is a \emph{$k_{\alpha}$-cell}.
    \item For every cell $Y_{\alpha}$, there is a homeomorphism $\phi_{\alpha}: B^{k_{\alpha}} \to \overline{X_{\alpha}}$, where $B^{k_{\alpha}}$ is the closed $k_{\alpha}$-dimensional ball, such that $\phi_{\alpha}(\text{int}(B^{k_{\alpha}})) = X_{\alpha}$.
    \item {\bf Axiom of the Frontier:} If the intersection $\overline{Y_{\beta}} \cap Y_{\alpha}$ is nonempty, then $Y_{\alpha} \subseteq \overline{Y_{\beta}}$. We say that $Y_{\alpha}$ is a \emph{face} of $Y_{\beta}$.
    \item {\bf Locally finite:} Every $x \in X$ has an open neighborhood $U$ such that $U$ intersects finitely many cells.
\end{enumerate}
For example, a polyhedron is a cell complex whose $k$-cells are the $k$-dimensional faces of the polyhedron. A graph is a another example of a cell complex; the $0$-cells are the vertices and the $1$-cells are the edges.

We will first review the most general definition of cellular sheaves, which uses category theory, and then we will specialize to the case of interest for \difference{this dissertation}{the present paper}, which does not require category theory. Let $Y$ be a cell complex with cells $\{Y_{\alpha}\}_{\alpha \in P_Y}$, and let $\mathcal{D}$ be a category. The set $P_Y$ is a poset with the relation $\alpha \leq \beta$ if $Y_{\alpha} \subseteq \overline{Y_{\beta}}$.

A \emph{$\mathcal{D}$-valued cellular sheaf on $Y$} consists of 
\begin{enumerate}
    \item An assignment of an object $\mathcal{F}(Y_{\alpha}) \in \mathcal{D}$ (the \emph{stalk} of $\mathcal{F}$ at $Y_{\alpha}$) for every cell $Y_{\alpha}$ in $Y$, and
    \item A morphism $\mathcal{F}_{\alpha \leq \beta}: F(Y_{\alpha}) \to F(Y_{\beta})$ (a \emph{restriction map}) whenever $Y_{\alpha}$ is a face of $Y_{\beta}$. The morphisms must satisfy the \emph{composition condition}:
    \begin{equation}\label{eq:composition}
        \mathcal{F}_{\beta \leq \gamma} \circ \mathcal{F}_{\alpha \leq \beta} = \mathcal{F}_{\alpha \leq \gamma}
    \end{equation}
    whenever $\alpha \leq \beta \leq \gamma$.
\end{enumerate}
Equivalently, a $\mathcal{D}$-valued cellular sheaf on $Y$ is a functor $\mathcal{F}: P_Y \to \mathcal{D}$, where $P_Y$ is considered as a category.

A \emph{global section} of a cellular sheaf $\mathcal{F}$ is a function
\begin{equation*}
s: \{Y_{\alpha}\}_{\alpha \in P_Y} \to \bigcup_{\alpha} \mathcal{F}(Y_{\alpha})
\end{equation*}
such that
\begin{enumerate}
    \item $s(Y_{\alpha}) \in \mathcal{F}(Y_{\alpha})$ (i.e., $s(Y_{\alpha})$ is a choice of element in the stalk at $Y_{\alpha}$), and
    \item if $\alpha \leq \beta$, then $s(Y_{\beta}) = \mathcal{F}_{\alpha \leq \beta}( s(Y_{\alpha}))$\,.
\end{enumerate}

\difference{In Chapter \ref{ch:pdbundle}}{In what follows}, we will consider a {\bf Set}-valued cellular sheaf. The objects of the category {\bf Set} are sets, and the morphisms between sets $A$ and $B$ are the functions from $A$ to $B$. A {\bf Set}-valued cellular sheaf on a cell complex $Y$ consists of
\begin{enumerate}
\item A set $\mathcal{F}(Y_{\alpha})$ for every cell $Y_{\alpha}$ in $Y$, and
\item A function $\mathcal{F}_{\alpha \leq \beta}: \mathcal{F}(Y_{\alpha}) \to \mathcal{F}(Y_{\beta})$ whenever $Y_{\alpha}$ is a face of $Y_{\beta}$. The functions must satisfy the condition:
\begin{equation*}
        \mathcal{F}_{\beta \leq \gamma} \circ \mathcal{F}_{\alpha \leq \beta} = \mathcal{F}_{\alpha \leq \gamma}
\end{equation*}
whenever $\alpha \leq \beta \leq \gamma$.
\end{enumerate}


\section{Definition of a Persistence Diagram Bundle}\label{sec:def}

A vineyard is a $1$-parameter set of persistence diagrams that is computed from a $1$-parameter set of filtration functions on a simplicial complex $\K$. We generalize a vineyard to a ``persistence diagram bundle'' as follows.

\begin{definition}[Fibered filtration function]\label{def:fff}
    A \emph{fibered filtration function} is a set $\{f_\bp : \K^\bp \to \mathbb{R}\}_{\bp \in \base}$, where $\base$ is a topological space, $\{\K^\bp\}_{\bp \in \base}$ is a set of simplicial complexes parameterized by $\base$, and $f_\bp$ is a filtration function on $\K^\bp$.
\end{definition}
\noindent When $\K^\bp \equiv \K$ for all $\bp \in \base$, we define $f: \K \times \base \to \mathbb{R}$ to be the function $f(\sigma, \bp) := f_\bp(\sigma)$. In a slight abuse of notation, we refer to $f : \K \times \base \to \mathbb{R}$, rather than to $\{f_{\bp}: \K \to \base\}_{\bp \in \base}$, as the fibered filtration function. For all $\bp \in \base$, the function $f(\cdot, \bp):\K \to \mathbb{R}$ is a filtration of $\K$. For several examples with $\K^\bp \equiv \K$, see Section \ref{sec:examples}.

\begin{definition}[Persistence diagram bundle]\label{def:pdbundle}
Let $\{f_\bp: \K^\bp \to \mathbb{R}\}_{\bp \in \base}$ be a fibered filtration function. The \emph{base} of the bundle is $\base$. The \emph{$\homdim$th total space} of the bundle is $E := \{(\bp, z)\} \mid \bp \in \base, z \in PD_\homdim(f_\bp)\}$, with the subspace topology inherited from the inclusion $E \hookrightarrow \base \times \overline{\mathbb{R}}^2$.\footnote{Technically, $E$ is a multiset because persistence diagrams are multisets. However, when considering $E$ as a topological space (which we do in Section \ref{sec:monodromy} to study continuous paths in $E$ and sections of the PD bundle), we consider $E$ as a set.}
The \emph{$\homdim$th persistence diagram bundle} is the triple $(E, \base, \pi)$, where $\pi: E \to \base$ is the projection $(\bp, z) \mapsto \bp$.
\end{definition}

For example, when $\base$ is an interval in $\mathbb{R}$ and $\K^\bp \equiv \K$, Definition \ref{def:pdbundle} reduces to that of a vineyard: a $1$-parameter set of PDs for a $1$-parameter set of filtrations of $\K$. As discussed in Section \ref{sec:intro}, PHT is a special case with $\base = S^d$. The fibered barcode of a multiparameter persistence module is another special case; we will discuss it in Section \ref{sec:fibered}.

\begin{remark}
    In Definition \ref{def:pdbundle}, we are suggestively using the language of fiber bundles. However, it is important to note that a PD bundle is not guaranteed to be a true fiber bundle. The fibers need not be homeomorphic to each other for all $\bp \in \base$. At ``singularities'' (points $\bp_* \in \base$ at which $\pdgm(f_{\bp_*})$ has an off-diagonal point with multiplicity), points in $\pdgm(f_{\bp})$ for nearby $\bp$ may merge into each other, changing the homotopy type of the fiber. However, if $f: \K \times \base \to \mathbb{R}$ is continuous and $\bp \in \base$ is not a singularity, then there is a neighborhood $U \subseteq \base$ and a homeomorphism $\phi : \pi^{-1}(U) \to U \times \pdgm(f_{\bp_*})$ that preserves fibers (i.e., a local trivialization).
\end{remark}

As a special case of fibered filtration functions, we define \emph{piecewise-linear fibered filtration functions}, which are simpler to analyze.
\begin{definition}[Piecewise-linear fibered filtration function]\label{def:PL_pdbundle}
    Let $\{f_\bp : \K^\bp \to \mathbb{R}\}_{\bp \in \base}$ be a fibered filtration function such that $\K^\bp \equiv \K$. We define $f(\sigma, \bp) := f_\bp(\sigma)$ for all $\sigma \in \K$ and $\bp \in \base$. If $\base$ is a simplicial complex and $f(\sigma, \cdot)$ is linear on each simplex of $\base$ for all simplices $\sigma \in \K$, then $f$ is a \emph{piecewise-linear fibered filtration function.} The resulting PD bundle is a \emph{piecewise-linear PD bundle}.
\end{definition}
\noindent For instance, the fibered filtration function in Example \ref{ex:image}, below, is piecewise linear.

\subsection{Examples}\label{sec:examples}

The following are concrete examples of PD bundles. We begin with the example that motivated PD bundles in Section \ref{sec:intro}.

\begin{example}\label{ex:main}
Suppose that $X(t, \bm{\mu}) = \{x_1(t, \bm{\mu}), \ldots, x_k(t, \bm{\mu})\}$ is a point cloud that varies continuously with time $t \in \mathbb{R}$ and system-parameter values $\mu_1, \ldots, \mu_m \in \mathbb{R}$. We obtain a fibered filtration function $f: \K \times \mathbb{R}^{m+1} \to \mathbb{R}$ by defining $f(\cdot, (t, \bm{\mu})): \K \to \mathbb{R}$ to be the Vietoris--Rips filtration function for the point cloud $X(t, \bm{\mu})$ at all $(t, \bm{\mu}) \in \mathbb{R}^{m+1}$ (or any other filtration for the point cloud at each $(t, \bm{\mu})$). The \scom\ $\K$ is the simplicial complex that has a simplex for every subset of points in the point cloud.

\end{example}

\begin{example}\label{ex:image}

Consider a color image. Enumerate the pixels and let $r(i)$, $g(i)$, and $b(i)$ denote the red, green, and blue values of the $i$th pixel. Triangulate each pixel to obtain a simplicial complex $\K$. (Every pixel is split into two triangles.) Let $\base = \{(w_1, w_2) \in [0, 1]^2 \mid 0 \leq w_1 + w_2 \leq 1\}$. For all $(w_1, w_2) \in \base$, define $p(i, (w_1, w_2)) = w_1 r(i) + w_2 g(i) + (1 - w_1 - w_2)b(i)$. The function $p(i, (w_1, w_2))$ is a weighted average of the red, green, and blue values of the $i$th pixel. Define a piecewise-linear fibered filtration function $f: \K \times \base \to \mathbb{R}$ as follows. For a $2$-simplex $\sigma$, define $f(\sigma, {\bf w}) = p(i(\sigma), {\bf w})$, where $i(\sigma)$ is the pixel containing $\sigma$. For any other simplex $\sigma \in \K$, define $f(\sigma, {\bf w}) = \min \{f(\tau, {\bf w}) \mid \sigma \subseteq \tau, \, \dim(\tau) = 2\}$. At ${\bf w} = (1, 0)$, ${\bf w} = (0, 1)$, and ${\bf w} = (1, 1)$, the filtration function $f(\cdot, {\bf w})$ is the sublevel filtration by red, green, and blue pixel values, respectively. At all other ${\bf w} \in \base$, the filtration function $f(\cdot, {\bf w})$ is the sublevel filtration by a weighted average of the red, green, and blue pixel values.
\end{example}

\begin{example}\label{ex:orbits}

Let $\mu_1, \ldots, \mu_m \in \mathbb{R}$ denote the system-parameter values of some discrete dynamical system. For given system-parameter values $\bm{\mu} \in \mathbb{R}^m$, let $x_i^{\bm{\mu}} \in \mathbb{R}^n$ be the solution at the $i$th time step and let $X(\bm{\mu}) = \{x_0^{\bm{\mu}}, \ldots, x_k^{\bm{\mu}}\}$ be the set of points obtained after the first $k$ time steps. For example, persistent homology has been used to study orbits of the linked twist map (a discrete dynamical system) \cite{pers_images}. We obtain a fibered filtration function $f: \K \times \mathbb{R}^m \to \mathbb{R}$ by defining $f(\cdot, \bm{\mu}): \K \to \mathbb{R}$ to be the Vietoris--Rips filtration function for the point cloud $X(\bm{\mu})$ (or any other filtration for the point cloud at each $\bm{\mu}$). The \scom\ $\K$ has a simplex for every subset of points in the point cloud.
\end{example}

\begin{example}\label{ex:sublevel}
Suppose that $X(t) = \{x_1(t), \ldots, x_k(t)\}$ is a time-varying point cloud in a compact triangulable subset $S \subseteq \mathbb{R}^n$. Let $\rho_h(\cdot, t)$ be a kernel density estimator at time $t$, with bandwidth parameter $h > 0$. For fixed $h$ and $t$, we define a filtered complex by considering sublevel sets of $\rho_h$ as follows. Let $\mathcal{K}$ be a triangulation of $S \subseteq \mathbb{R}^n$. A vertex $v$ of $\mathcal{K}$ is included in the \scom\ $\K_r^{(t, h)}$ if $\rho_h(v, t) \leq r$, and a simplex of $\mathcal{K}$ is included in $\K_r^{t, h}$ if all of its vertices are in $\K_r^{(t, h)}$. For each $t$ and $h$, the set $\{\K^{(t, h)}_r\}_{r \in \mathbb{R}}$ is a filtered complex. We obtain a fibered filtration function $f:\K \times \mathbb{R}_+^2 \to \mathbb{R}$ by defining $f(\cdot, (t, h))$ to be the filtration function associated with the filtered complex $\{\K_r^{(t, h)}\}_{r \in \mathbb{R}}$.

Density sublevels of time-varying point clouds were also considered by Corcoran et al. \cite{CorcoranJones}, who studied a school of fish swimming in a shallow pool that was modeled as a subset of $\mathbb{R}^2$. However, Corcoran et al. \cite{CorcoranJones} fixed a bandwidth parameter $h$ and a sublevel $r$, and only studied how the PH changed with time (by using zigzag PH).
\end{example}

\subsection{Comparison to multiparameter PH}\label{sec:multi}
Multiparameter PH was introduced in \cite{multi}; see \cite{multi_review} for a review. Typically, a fibered filtration function does not induce a multifiltration, but the fibered barcode of a multiparameter peristence module is an example of a PD bundle.

\subsubsection{Multifiltrations} We review the definition of a multifiltration and compare it to the definition of a fibered filtration.
\begin{definition}\label{def:multifilt}
A \emph{multifiltration} is a set $\{\K_{\bm{u}}\}_{\bm{u} \in \mathbb{Z}^n}$ of simplicial complexes such that if $\bm{u} \leq \bm{v}$, then $\K_{\bm{u}} \subseteq \K_{\bm{v}}$.
\end{definition}
\noindent The inclusion $\iota^{\bm{u}, \bm{v}}: \K_{\bm{u}} \hookrightarrow \K_{\bm{v}}$ induces a map $\iota_*^{\bm{u}, \bm{v}}: H_\homdim(\K_{\bm{u}}, \mathbb{F}) \to H_\homdim(\K_{\bm{v}}, \mathbb{F})$ from the $\homdim$th homology of $\K_{\bm{u}}$ to the $\homdim$th homology of $\K_{\bm{v}}$ over a field $\mathbb{F}$. Given a multifiltration $\{\K_{\bm{u}}\}_{\bm{u} \in \mathbb{Z}^n}$, the \emph{multiparameter persistence module} is the graded $\mathbb{F}[x_1, \ldots, x_n]$-module $\bigoplus_{\bm{u} \in \mathbb{Z}^n} H_\homdim(\K_{\bm{u}}, \mathbb{F})$. The action of $x_i$ on a homogeneous element $\gamma \in H_\homdim(\K_{\bm{u}}, \mathbb{F})$ is given by $x_i \gamma = \iota_{\bm{u}, \bm{v}}^*(\gamma)$, where $v_j = u_j + \delta_{ij}$.

\begin{remark}\label{rmk:general_multifilt}
Some researchers define multifiltrations more generally as functors $\mathcal{F}: \mathcal{P}\to \bf{Simp}$, where $\mathcal{P}$ is any poset and $\bf{Simp}$ is the category of simplicial complexes, with simplicial maps as morphisms. Definition \ref{def:multifilt} 
is the specific case in which $\mathcal{P} = \mathbb{Z}^n$ and $\mathcal{F}_{\bm{u} \leq \bm{v}}: \K_{\bm{u}} \to \K_{\bm{v}}$ is an inclusion map.
\end{remark}

To see why a fibered filtration function does not typically induce a multifiltration, consider a fibered filtration function $\{f_\bp : \K^\bp \to \mathbb{R}\}_{\bp \in \base}$ with $\base = \mathbb{R}^n$. Let $\K_r^{\bp} : = \{\sigma \in \K^\bp \mid f_\bp(\sigma) \leq r\}$ denote the $r$-sublevel set of $f_\bp$. It is not necessarily the case that $\K_s^{\bp_1} \subseteq \K_r^{\bp_2}$ whenever $r \leq s$ and $\bp_1 \leq \bp_2$. Moreover, there are no canonical simplicial maps $\K_s^{\bp_1} \to \K_r^{\bp_2}$, so it is not guaranteed that $\{\K_r^{\bp}\}_{(\bp, r) \in \mathbb{R}^n \times \mathbb{R}}$ is a multifiltration even in the general sense of Remark \ref{rmk:general_multifilt}. Therefore, such a set of filtered complexes cannot be analyzed using multiparameter persistent homology.

\subsubsection{Fibered barcodes}\label{sec:fibered}
Consider a multifiltration $\{\K_{\bm{u}}\}_{\bm{u} \in \mathbb{R}^n}$. Let $\lines$ denote the space of lines in $\mathbb{R}^n$ with a parameterization of the form
\begin{align*}
    &L: \mathbb{R} \to \mathbb{R}^n\,, \\
    &L(r) = r\bm{v} + \bm{b}\,, \qquad \bm{v} \in [0, \infty)^n,\, \norm{\bm{v}} = 1,\, \bm{b} \in \mathbb{R}^n.
\end{align*}
For example, when $n = 2$, the space $\lines$ is the space of lines in $\mathbb{R}^2$ with non-negative slope, including vertical lines. For each line $L \in \lines$, we define $\K^L_r := \K_{L(r)}$. That is, $\{\K^L_r\}_{r \in \mathbb{R}}$ is the filtered complex obtained by restricting the multifiltration $\{\K_{\bm{u}}\}_{\bm{u} \in \mathbb{R}^n}$ to the line $L \subseteq \mathbb{R}^n$. The set $\{\K^L_r\}_{r \in \mathbb{R}}$ is a filtered complex because $L(r)_i \leq L(s)_i$ for all $r \leq s$ and $i \in \{1, \ldots, n\}$. The \emph{fibered barcode} \cite{fibered} is the map that sends $L \in \lines$ to the barcode for the persistent homology of $\{\K^L_r\}_{r \in \mathbb{R}}$.

A fibered barcode is a PD bundle whose base space is $\base = \lines$. For $L \in \lines$, the filtration function is
\begin{align*}
    &f_L : \K^L \to \mathbb{R}\,, \\
    &f_L(\sigma) = \min\{r \mid \sigma \in \K_{L(r)}\}\,,
\end{align*}
where $\K^L := \bigcup_{r \in \mathbb{R}} \mathcal{K}_{L(r)}$. Unlike the other examples in Section \ref{sec:examples}, the simplicial complex $\K^L$ is not independent of $L \in \lines$.


\section{A Stratification of the Base Space}\label{sec:partition}

There are many different notions of a stratified space \cite{stratified}. In the present \difference{thesis}{paper}, what we mean by a stratification is the following definition.
\begin{definition}\label{def:stratification}
A \emph{stratification} of a topological space $B$ is a nested sequence
\begin{equation*}
  \base^0 \subseteq \base^1 \subseteq \base^2 \subseteq \cdots \subseteq \base^n = \base
\end{equation*}
of closed subsets $\base^m$ such that the following hold:
\begin{enumerate}
    \item For all $m$, the space $\base^m \setminus \base^{m-1}$ is either empty or a smooth $m$-dimensional submanifold of $\base$ (where we set $\base^{-1} := \emptyset$). The \emph{$m$-dimensional strata} are the connected components of $\base^m \setminus \base^{m-1}$. We denote the set of strata by $\mathcal{Y}$.
    \item The set $\mathcal{Y}$ of strata is \emph{locally finite}: every $\bp \in \base$ has an open neighborhood $U$ such that $U$ intersects finitely many elements of $\mathcal{Y}$.
    \item The set $\mathcal{Y}$ of strata satisfy the \emph{Axiom of the Frontier}: If $Y_{\alpha}$, $Y_{\beta} \in \mathcal{Y}$ are strata such that $Y_{\beta} \cap \overline{Y_{\alpha}}$, then $Y_{\beta} \subseteq \overline{Y_{\alpha}}$. We write that $Y_{\beta}$ is a \emph{face} of $Y_{\alpha}$.
\end{enumerate}
\end{definition}
\noindent In the present \difference{thesis}{paper}, $\base$ is the base space of a fibered filtration function.

Theorem \ref{thm:stratified} says that for any ``generic'' smooth fibered filtration function (see Section \ref{sec:generic}), the base space $\base$ can be stratified so that in each stratum $Y \subseteq \base$, the set of (birth, death) simplex pairs is constant and can be used to obtain $\pdgm(f_\bp)$ for any $\bp \in Y$.

\subsection{Piecewise-linear fibered filtrations} As a warm-up, we first consider piecewise-linear fibered filtration functions, which will provide intuition for the general case. However, note that Proposition \ref{prop:polyhedrons_constant} below is not simply a special case of \Cref{thm:stratified}, in which we consider generic smooth fibered filtrations on smooth compact manifolds (see \Cref{sec:generic}). Here, we consider \emph{all} piecewise-linear fibered filtrations, rather than only generic piecewise-linear fibered filtrations.

First, we establish some notation and definitions.
\begin{definition}
An \emph{open half-space} of an affine space $A$ is one of the two connected components of $A \setminus H$ for some hyperplane $H$.
\end{definition}
\noindent For example, an open half-space of $\mathbb{R}^n$ is a set of the form $\{\bm{x} \in \mathbb{R}^n \mid A \bm{x} > \bm{b}\}$ for some $n \times n$ matrix $A$ and some vector $\bm{b} \in \mathbb{R}^n$.
\begin{definition}
An \emph{open polyhedron} is the intersection of open half-spaces.
\end{definition}
\noindent For example, an open polygon $P$ (a polygon without its faces) in $\mathbb{R}^2$ is an open 2D polyhedron because it is the intersection of half-spaces of $\mathbb{R}^2$. The 1D faces (i.e., edges) of $P$ are 1D polyhedra because an edge is a subset of a line $L \subseteq \mathbb{R}^2$ and the edge is the intersection of two half-spaces of $L$. The 0D faces (i.e., vertices) of $P$ are 0D polyhedra.

We fix a simplicial complex $\K$ for the remainder of this section. For each pair $(\sigma, \tau)$ of simplices in $\K$, we define
\begin{equation}\label{eq:I_def}
    I(\sigma, \tau) := \{\bp \in \base \mid f(\sigma, \bp) = f(\tau, \bp)\}\,.
\end{equation}
\begin{lemma}\label{lem:constantorder}
    Suppose that $f: \K \times \base \to \mathbb{R}$ is a continuous fibered filtration function (i.e., $f(\sigma, \cdot): \base \to \mathbb{R}$ is continuous for all simplices $\sigma \in \K$) and that $Y$ is a path-connected subset of $\base$. If all pairs $(\sigma, \tau)$ of simplices satisfy either $I(\sigma, \tau) \cap Y = \emptyset$ or $I(\sigma, \tau) \cap Y = Y$, then the simplex order is constant in $Y$. That is, there is a strict partial order $\prec_Y$ on $\K$ such that for all $y \in Y$, we have that $\prec_{f(\cdot, y)}$ is the same as $\prec_Y$.
\end{lemma}
\begin{proof}
Let $(\sigma, \tau)$ be a pair of simplices. If $I(\sigma, \tau) \cap Y = Y$, then $f(\sigma, \bp) = f(\tau, \bp)$ for all $\bp \in Y$, so $\sigma \not\prec_{f(\cdot, \bp)} \tau$ and $\tau \not\prec_{f(\cdot, \bp)} \sigma$ for all $\bp \in Y$. If $I(\sigma, \tau) \cap Y = \emptyset$, then $f(\sigma, \bp) \neq f(\tau, \bp)$ for all $\bp \in Y$. Let $\bp_0$ be a point in $Y$. Without loss of generality, $\tau \prec_{f(\cdot, \bp_0)} \sigma$. Therefore, $f(\sigma, \bp_0) > f(\tau, \bp_0)$. To obtain a contradiction, suppose that $f(\sigma, \bp_1) < f(\tau, \bp_1)$ for some $\bp_1 \in Y$.  Let $\gamma:[0, 1] \to Y$ be a continuous path from $\bp_0$ to $\bp_1$, and let $g(s) = f(\sigma, \gamma(s)) - f(\tau, \gamma(s))$ for $s \in [0,1]$. By the Intermediate Value Theorem, there is an $s_* \in [0,1]$ such that $f(\sigma, \gamma(s_*)) = f(\tau, \gamma(s_*))$, but this is a contradiction. Therefore, $f(\sigma, \bp) > f(\tau, \bp)$ for all $\bp \in Y$, which implies that $\tau \prec_{f(\cdot, \bp)} \sigma$ for all $\bp \in Y$.
\end{proof}

\begin{proposition}\label{prop:polyhedrons_constant}
Let $\base$ be a simplicial complex. If $f: \K \times \base \to \mathbb{R}$ is a piecewise-linear fibered filtration function, then $\base$ can be partitioned into disjoint polyhedra $P$ on which the simplex order induced by $f$ is constant. That is, there is a strict partial order $\prec_P$ on $\K$ such that $\prec_{f(\cdot, p)}$ is the same as $\prec_P$ for all $p \in P$. Consequently, the set $\{(\sigma_b, \sigma_d)\}$ of (birth, death) simplex pairs for $f$ is constant in each $P$ and for any $\bp \in P$, the persistence diagram $\pdgm(f_{\bp})$ consists of the diagonal and the multiset $\{(f(\sigma_b), f(\sigma_d)\}$.
\end{proposition}
\begin{proof}
	Let $\Delta$ be an $n$-dimensional simplex of the simplicial complex $\base$ and let $\sigma$ and $\tau$ be distinct simplices of $\K$.
Because $f(\sigma, \cdot)\vert_\Delta$ and $f(\tau, \cdot)\vert_\Delta$ are linear,
	the set $I(\sigma, \tau) \cap \Delta$ is one of the following:
	\begin{enumerate}
	    \item the intersection of an $(n-1)$-dimensional hyperplane with $\Delta$\,;
		\item $\emptyset$\,;
		\item $\Delta$\,;
		\item a vertex of $\Delta\,$.
	\end{enumerate}
	Therefore, the set $\partial \Delta \bigcup \{I(\sigma, \tau) \cap \Delta \mid \emptyset \subset (I(\sigma, \tau)\cap \Delta) \subset \Delta \}_{\sigma, \tau \in \K}$ partitions $\Delta$ into polyhedra. By Lemma \ref{lem:constantorder}, the simplex order induced by $f$ is constant on each polyhedron. The last statement of Proposition \ref{prop:polyhedrons_constant} follows from Lemma \ref{lem:only_order}.
\end{proof}

For example, if $\base$ is a triangulated surface, then the set
\begin{equation}\label{eq:lines}
    L := \bigcup_{\Delta \in \base} \partial \Delta \cup \{I(\sigma, \tau)\cap \Delta \mid \emptyset \subset (I(\sigma, \tau)\cap \Delta) \subset \Delta \}_{\sigma, \tau \in \K}
\end{equation}
partitions $\Delta$ into polyhedra such that the simplex order is constant on each polyhedron, including the 1D polyhedra (i.e., edges) and the 0D polyhedra (i.e., vertices). The polygonal subdivision induced by $L$ is called a \emph{line arrangement} $\mathcal{A}(L)$. For an example of such a line arrangement, see Figure \ref{fig:polygon partition}.
\begin{figure}
    \centering
    \includegraphics[width = .5\textwidth]{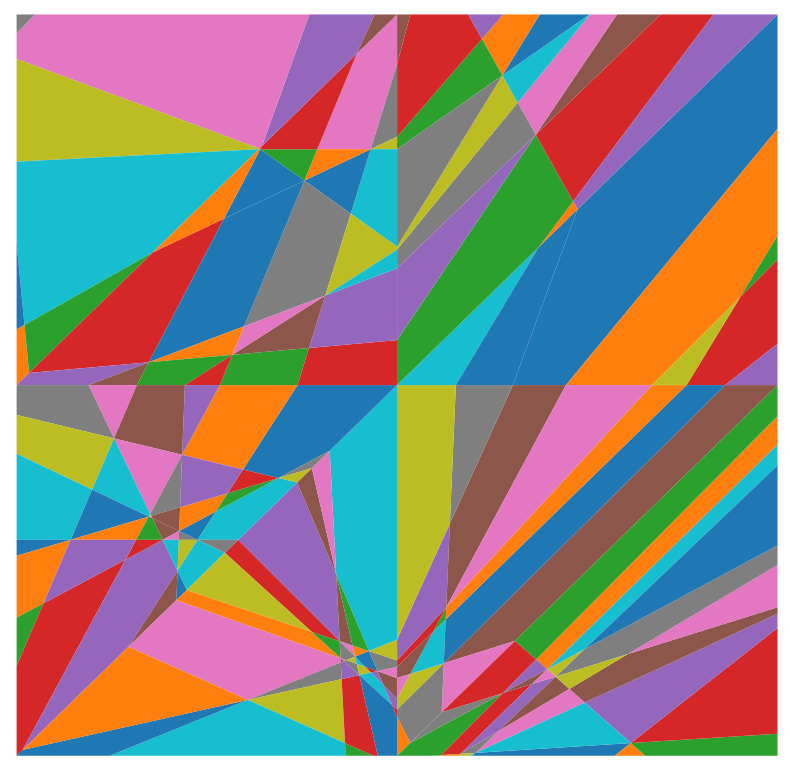}
    \caption{A line arrangement that represents the partition of a triangulated surface $\base$ (the base space) into polyhedra on which the simplex order is constant.}
    \label{fig:polygon partition}
\end{figure}

\subsection{Generic smooth fibered filtrations}\label{sec:generic}
We now consider generic smooth fibered filtration functions. Throughout Section \ref{sec:generic}, we consider a smooth fibered filtration function of the form $f: \K \times \base \to \mathbb{R}$ for some $n$-dimensional smooth compact manifold $\base$ and some simplicial complex $\K$. (A fibered filtration $f$ is \emph{smooth} if $f(\sigma, \cdot): \base \to \mathbb{R}$ is smooth for all $\sigma \in \K$.) To make precise the notion of a ``generic'' fibered filtration function, we consider perturbations of $f$ of a certain form. Because the filtration value of a simplex $\sigma$ must be at least as large as the filtration value of any face $\tau$ at all $\bp \in \base$, we consider only perturbations $f_{\bm{a}}: \K \times \base \to \mathbb{R}$ of the form
\begin{equation*}
    f_{\bm{a}}(\sigma_i, \bp) := f(\sigma_i, \bp) + a_i\,,
\end{equation*}
where $\bm{a}$ is an element of the set
\begin{equation}\label{eq:Adef}
A := \{\bm{a} \in \mathbb{R}^N \mid a_i \leq a_j \text{ for all } i \leq j \}\
\end{equation}
and $\sigma_1, \ldots, \sigma_N$ are the simplices of $\K$, indexed such that $i < j$ if $\sigma_i$ is a proper face of $\sigma_j$. By construction, $f_{\bm{a}}$ is a fibered filtration function for all $\bm{a} \in A$. 

For each simplex $\sigma_k$ in $\K$, we define the manifold
\begin{equation*}
    M_k := \{(\bp, f(\sigma_k, \bp)) \mid \bp \in B\} \subseteq B \times \mathbb{R}\
\end{equation*}
and for each $\bm{a} \in A$, we define the manifold
\begin{equation}\label{eq:Ma_def}
    M_{\bm{a}, k} := \{(\bp, f_{\bm{a}}(\sigma_k, \bp)) \mid \bp \in \base\} \subseteq B \times \mathbb{R}\,.
\end{equation}
For each pair $(\sigma_i, \sigma_j)$ of simplices in $\K$, we define $I(\sigma_i, \sigma_j)$ as in \eqref{eq:I_def}. The set $I(\sigma_i, \sigma_j)$ is the projection of $M_i \cap M_j \subseteq \base \times \mathbb{R}$ to a subset of $\base$. For each $\bm{a} \in A$, we define the set
\begin{equation*}
I_{\bm{a}}(\sigma_i, \sigma_j) := \{\bp \in \base \mid f_{\bm{a}}(\sigma_i, \bp) = f_{\bm{a}}(\sigma_j, \bp)\}\,.
\end{equation*}
We also define
\begin{align*}
    E^m &:= \{I(\sigma_{i_1}, \sigma_{j_1}) \cap \cdots \cap I(\sigma_{i_m}, \sigma_{j_m})\}\,,
\end{align*}
which is the set of all $m$-way intersections of sets $I(\sigma_i, \sigma_j)$. For all $\bm{a} \in A$, we define 
\begin{equation}\label{eq:Eadef}
    E^m_{\bm{a}} := \{I_{\bm{a}}(\sigma_{i_1}, \sigma_{j_1}) \cap \cdots \cap I_{\bm{a}}(\sigma_{i_m}, \sigma_{j_m})\}\,.
\end{equation}
Lastly, we define
\begin{equation}\label{eq:Ekdef}
    E^{m, k}_{\bm{a}} := \{ I_{\bm{a}}(\sigma_{i_1}, \sigma_{j_1}) \cap \cdots \cap I_{\bm{a}}(\sigma_{i_m}, \sigma_{j_m}) \mid i_r, j_r \leq k \text{ for all } r \}\,,
\end{equation}
which is the set of $m$-way intersections that only involve the simplices $\sigma_1, \ldots, \sigma_k$.
\begin{remark}
    There are several facts to keep in mind. First, it is not guaranteed that $I_{\bm{a}}(\sigma_i, \sigma_j)$ is homeomorphic to $I(\sigma_i, \sigma_j)$ even for arbitrarily small $\bm{a}$. Additionally, the sets $I_{\bm{a}}(\sigma_i, \sigma_j)$ are not ``independent'' of each other; a perturbation of $f(\sigma, \cdot)$ for a single simplex $\sigma$ causes a perturbation of $I(\sigma, \tau)$ for all $\tau \in \K$. Furthermore, not every element of $E^m$ is an $(n-m)$-dimensional submanifold, even generically. For example, if $I(\sigma_{i_2}, \sigma_{i_1})$ and $I(\sigma_{i_3}, \sigma_{i_2})$ are $(n-1)$-dimensional submanifolds that intersect transversely, then $I(\sigma_{i_3}, \sigma_{i_1}) \cap I(\sigma_{i_3}, \sigma_{i_2}) \cap I(\sigma_{i_2}, \sigma_{i_1}) = I(\sigma_{i_3}, \sigma_{i_2}) \cap I(\sigma_{i_2}, \sigma_{i_1})$ is an $(n-2)$-dimensional submanifold, rather than an $(n-3)$-dimensional submanifold. Finally, $\bigcap_{r = 1}^m I(\sigma_{i_r}, \sigma_{j_r})$ is not necessarily equal to the projection of $\bigcap_{r = 1}^m (M_{i_r} \cap M_{j_r})$ to $\base$. In other words, not every intersection in $B$ lifts to an intersection of the manifolds $\{M_k\}_{k=1}^N \subseteq B \times \mathbb{R}$. These are the main subtleties in the proof of Theorem \ref{thm:stratified}.
\end{remark}
\begin{definition}
Let $S$ be an element of $E^m_{\bm{a}}$, where $E^m_{\bm{a}}$ is defined as in \eqref{eq:Eadef}. The set $S$ is \emph{$m$-reduced} if it equals a set of the form $I_{\bm{a}}(\sigma_{i_1}, \sigma_{j_1}) \cap \cdots \cap I_{\bm{a}}(\sigma_{i_m}, \sigma_{j_m})$, where $i_1 > i_2 > \cdots > i_m$ and $i_r > j_r$ for all $r$.
\end{definition}
\noindent For example, if $\sigma_{i_1}$, $\sigma_{i_2}$, and $\sigma_{i_3}$ are distinct simplices, then $I_{\bm{a}}(\sigma_{i_3}, \sigma_{i_2}) \cap I_{\bm{a}}(\sigma_{i_2}, \sigma_{i_1})$ is $2$-reduced, but $I_{\bm{a}}(\sigma_{i_3}, \sigma_{i_1}) \cap I_{\bm{a}}(\sigma_{i_3}, \sigma_{i_2}) \cap I_{\bm{a}}(\sigma_{i_2}, \sigma_{i_1})$ is not $3$-reduced. We define
\begin{align}
    \overline{E^m_{\bm{a}}} &:= \{S \in E^m_{\bm{a}} \mid S \text{ is } m\text{-reduced}\}\,, \label{eq:reduced_sets}\\
    \overline{E^{m, k}_{\bm{a}}} &:= \{S \in E^{m, k}_{\bm{a}} \mid S \text{ is } m\text{-reduced}\} \label{eq:Ek_reduced}\,,
\end{align}
where $E^m_{\bm{a}}$ is defined as in \eqref{eq:Eadef} and $E^{m, k}_{\bm{a}}$ is defined as in \eqref{eq:Ekdef}.

\begin{lemma}\label{lem:reduced_form} For all $m \geq 1$, all $k$, and all $\bm{a} \in A$, where $A$ is defined as in \eqref{eq:Adef}, every $S \in E^{m, k}_{\bm{a}}$ belongs to $\overline{E^{m', k}_{\bm{a}}}$ for some $m' \leq m$, where $E^{m, k}_{\bm{a}}$ and $\overline{E^{m', k}_{\bm{a}}}$ are defined as in \eqref{eq:Ekdef} and \eqref{eq:Ek_reduced}, respectively.
\end{lemma}
\begin{proof}
    We prove the lemma by induction on $m$. 
    For all $k$, every $S \in E^{1, k}_{\bm{a}}$ is $1$-reduced by definition. Assume that Lemma \ref{lem:reduced_form} is true for $m-1 \geq 1$, and let $S$ be an element of $E^{m, k}_{\bm{a}}$. The set $S$ is equal to a set of the form
    \begin{equation*}
        I_{\bm{a}}(\sigma_{i_1}, \sigma_{j_1}) \cap \cdots \cap I_{\bm{a}}(\sigma_{i_m}, \sigma_{j_m})\,,
    \end{equation*}
    where $i_r > j_r$ for all $r$ and $k \geq i_1 \geq i_2 \geq \cdots \geq i_m$ without loss of generality. By the induction hypothesis,
    \begin{equation*}
        I_{\bm{a}}(\sigma_{i_2}, \sigma_{j_2}) \cap \cdots \cap I_{\bm{a}}(\sigma_{i_m}, \sigma_{j_m}) = I_{\bm{a}}(\sigma_{i_2'}, \sigma_{j_2'}) \cap \cdots \cap I_{\bm{a}}(\sigma_{i_{\ell}'}, \sigma_{j_{\ell}'})
    \end{equation*}
    for some $\ell \leq m$, where $i'_r > j'_r$ for all $r$ and $k \geq i_1 \geq i_2 \geq i'_2 > i'_3 > \cdots > i'_{\ell}$. If $i_1 > i'_2$, then $S$ is an element of $E^{\ell, k}$ and we are done. Otherwise,
    \begin{equation*}
        I_{\bm{a}}(\sigma_{i_1}, \sigma_{j_1}) \cap I_{\bm{a}}(\sigma_{i_2'}, \sigma_{j_2'}) = I_{\bm{a}}(\sigma_{i_1}, \sigma_{j_1}) \cap I_{\bm{a}}(\sigma_{j_1}, \sigma_{j_2'})
    \end{equation*}
    because $i_1 = i_2'$. If $j_1 = j_2'$, then 
    \begin{equation*}
    S = I_{\bm{a}}(\sigma_{i_1}, \sigma_{j_1}) \cap I_{\bm{a}}(\sigma_{i_3'}, \sigma_{j_3'}) \cap \cdots \cap I_{\bm{a}}(\sigma_{i_{\ell}'}, \sigma_{j_{\ell}'})\,,
    \end{equation*}
    so $S$ is $(\ell - 1)$-reduced. Otherwise, 
    \begin{equation*}
        S = I_{\bm{a}}(\sigma_{i_1}, \sigma_{j_1}) \cap I_{\bm{a}}(\sigma_{j_1}, \sigma_{j_2'}) \cap I_{\bm{a}}(\sigma_{i_3'}, \sigma_{j_3'}) \cap \cdots \cap I_{\bm{a}}(\sigma_{i_{\ell}'}, \sigma_{j_{\ell}'})\,,
    \end{equation*}
    where $k \geq i_1 > j_1, j_2'$ and $i_1 > i_r'$ for all $r \geq 3$. By the induction hypothesis, the set $I_{\bm{a}}(\sigma_{j_1}, \sigma_{j_2'}) \cap I_{\bm{a}}(\sigma_{i_3'}, \sigma_{j_3'}) \cap \cdots \cap I_{\bm{a}}(\sigma_{i_{\ell}'}, \sigma_{j_{\ell}'})$ belongs to $\overline{E^{\ell', k-1}_{\bm{a}}}$ for some $\ell' \leq \ell - 1$, so $S$ belongs to $\overline{E^{\ell' + 1, k}_{\bm{a}}}$, where $\ell' + 1 \leq \ell \leq m$.
\end{proof}

\begin{lemma}\label{lem:manifolds}
For almost every $\bm{a} \in A$ (where $A$ is defined as in \eqref{eq:Adef}), we have that $M_{\bm{a}, i}$ intersects $M_{\bm{a}, j}$ transversely\footnote{When manifolds $M_1$ and $M_2$ intersect transversely, we will use the notation $M_1 \pitchfork M_2$.} for $i \neq j$ and every $S \in \overline{E^m_{\bm{a}}}$ is either $\emptyset$ or an $(n-m)$-dimensional submanifold of $\base$ for all $m \in \{1, \ldots, n\}$, where $M_{\bm{a}, i}$ is defined as in \eqref{eq:Ma_def} and $\overline{E^m_{\bm{a}}}$ is defined as in \eqref{eq:reduced_sets}.
\end{lemma}
\begin{proof}
Define $g^{ij}(\bp) := f(\sigma_i, \bp) - f(\sigma_j, \bp)$ for all $i \neq j$. For almost every $\bm{a} \in A$, the quantity $a_j - a_i$ is a regular value of $g^{ij}$ by Sard's Theorem. The set of regular values is open for all $i \neq j$ because $g^{ij}$ is smooth and $\base$ is compact. Therefore, there is an $\epsilon^*$ such that for all $i \neq j$, every $y \in (a_j - a_i - 2\epsilon^*, a_j - a_i + 2\epsilon^*)$ is a regular value of $g^{ij}$.

Given an $\bm{a}$ and $\epsilon^*$ as above, it suffices to show that for almost every $\bm{\epsilon} \in \mathbb{R}^N$ with $\vert \epsilon_i\vert \leq \epsilon^*$, we have that every $S \in \overline{E^m_{\bm{a + \epsilon}}}$ is an $(n-m)$-dimensional submanifold of $\base$ for all $m$. For $m = 1$, every element of $\overline{E^1_{\bm{a + \epsilon}}}$ is of the form $I_{\bm{a + \epsilon}}(\sigma_i, \sigma_j)$ for some $i \neq j$. The set $I_{\bm{a + \epsilon}}(\sigma_i, \sigma_j)$ is the $(a_j - a_i + \epsilon_j - \epsilon_i)$-level set of $g^{ij}$. Because $(a_j - a_i + \epsilon_j - \epsilon_i)$ is a regular value of $g^{ij}$, the set $I_{\bm{a + \epsilon}}(\sigma_i, \sigma_j)$ is an $(n-1)$-dimensional submanifold of $\base$ and we must have $M_{\bm{a}, i} \pitchfork M_{\bm{a}, j}$.

For $m \geq 2$, observe that 
\begin{equation*}
    \overline{E^m_{\bm{a}}} = \overline{E^{m, 2}_{\bm{a}}} \cup \Big(\bigcup_{k=3}^N \overline{E^{m, k}_{\bm{a}}} \setminus \overline{E^{m, k-1}_{\bm{a}}}\Big)\,,
\end{equation*}
where $\overline{E^{m, k}_{\bm{a}}}$ is defined as in \eqref{eq:Ek_reduced}. We induct on $k \in \{2, \ldots, N \}$, where $N$ is the number of simplices in $\K$. When $k = 2$, we have
\begin{equation*}
    \overline{E^{m, 2}_{\bm{a + \epsilon}} }= \begin{cases}
        \{I_{\bm{a + \epsilon}}(\sigma_1, \sigma_2)\} \,, & m = 1 \\
        \emptyset \,, & m \geq 2\,,
    \end{cases}
\end{equation*}
so every $S \in \overline{E^{m, 2}_{\bm{a + \epsilon}}}$ is either $\emptyset$ or an $(n-m)$-dimensional submanifold of $\base$. Now suppose that $k > 2$ and that every element of $S \in \overline{E^{m, k-1}_{\bm{a + \epsilon}}}$ is either $\emptyset$ or an $(n-m)$-dimensional submanifold for all $m$. Every element $S$ in $\overline{E^{m, k}_{\bm{a + \epsilon}}} \setminus \overline{E^{m, k-1}_{\bm{a + \epsilon}}}$ is equal to a set of the form
\begin{equation*}
    S = I_{\bm{a + \epsilon}}(\sigma_k, \sigma_\ell) \cap S'\,,
\end{equation*}
where $\ell \leq k-1$ and $S' \in \overline{E^{m-1, k-1}_{\bm{a + \epsilon}}}$.
We define the vectors $\bm{\epsilon}^j := (0, \ldots, 0, \epsilon_j, 0, \ldots, 0)$ and $\bm{b}^j := \bm{a + \epsilon} - \bm{\epsilon}^j$.
Note that $I_{\bm{a} + \bm{\epsilon}}(\sigma_k, \sigma_{\ell}) = I_{\bm{b} + \bm{\epsilon}^k}(\sigma_k, \sigma_{\ell})$ because $b_i^k + \epsilon_i^k = a_i + \epsilon_i$ for all $i$, and $\overline{E^{m-1, k-1}_{\bm{a + \epsilon}}} = \overline{E^{m-1, k-1}_{\bm{b}^k}}$ because $a_i + \epsilon_i = b_i$ for all $i \leq k-1$. Therefore, every $S \in \overline{E^{m, k}_{\bm{a + \epsilon}}} \setminus \overline{E^{m, k-1}_{\bm{a + \epsilon}}}$ is equal to a set of the form
\begin{equation*}
    S = I_{\bm{b}^k + \bm{\epsilon}^k}(\sigma_k, \sigma_{\ell}) \cap S'
\end{equation*}
for some $S' \in \overline{E^{m-1, k-1}_{\bm{b}^k}}$ and $\ell \leq k-1$. Because $g^{\ell k}$ has no critical values between $b_{\ell} - b_k - \epsilon_k$ and $b_{\ell} - b_k$, we have that $I_{\bm{b}^k + \bm{\epsilon}^k}(\sigma_k, \sigma_\ell)$ is diffeomorphic to $I_{\bm{b}^k}(\sigma_k, \sigma_\ell)$ for all $\epsilon_k \in (-\epsilon^*,  \epsilon^*)$. In other words, $I_{\bm{b}^k + \bm{\epsilon}^k}(\sigma_k, \sigma_\ell)$ is a perturbation of $I_{\bm{b}^k}(\sigma_k, \sigma_\ell)$ for all $\epsilon_k \in (-\epsilon^*,  \epsilon^*)$. By Thom's Transversality Theorem, $I_{\bm{b}^k + \bm{\epsilon}^k}(\sigma_k, \sigma_\ell)$ intersects every $S' \in E_{\bm{b}^k}^{m-1, k-1}$ transversely for almost every $\epsilon_k \in (-\epsilon^*, \epsilon^*)$. This shows that $S$ is either $\emptyset$ or an $(n-m)$-dimensional submanifold of $\base$ for almost every $\epsilon_k \in (\epsilon^*, \epsilon^*)$. Because there are finitely many elements in $\overline{E^{m, k}_{\bm{a}}}$, we must have that every $S \in \overline{E^{m, k}_{\bm{a}}}$ is either $\emptyset$ or an $(n-m)$-dimensional submanifold of $\base$ for almost every $\epsilon_k \in (\epsilon^*, \epsilon^*)$. Induction on $k$ concludes the proof.
\end{proof}

\begin{lemma}\label{lem:stratum}
    For all $\bm{a} \in A$, with $A$ defined as in \eqref{eq:Adef}, define
    \begin{equation}\label{eq:strat_def}
        B^n_{\bm{a}} := B\, \qquad B^m_{\bm{a}} := \bigcup_{\ell \leq m} \bigcup_{S \in \overline{E_{\bm{a}}^{n- \ell}}} S \qquad \text{ for } m < n\,.
    \end{equation}
    If $\bm{a} \in A$ is such that every $S \in \overline{E_{\bm{a}}^{n - \ell}}$ is either $\emptyset$ or an $\ell$-dimensional smooth submanifold for every $\ell \in \{1, \ldots, n\}$, where $\overline{E_{\bm{a}}^{n - \ell}}$ is defined as in \eqref{eq:reduced_sets}, then $B^m_{\bm{a}} \setminus B^{m-1}_{\bm{a}}$ is the disjoint union of smooth $m$-dimensional manifolds.
\end{lemma}
\begin{proof}
    We have that
\begin{equation*}
    \base_{\bm{a}}^m \setminus \base_{\bm{a}}^{m-1} = \bigcup_{S \in \overline{E^{n-m}_{\bm{a}}}} \Big(S \setminus \bigcup_{\substack{S' \in \overline{E^{n - \ell}_{\bm{a}}} \\ \ell \leq m-1}} S'\Big)\,.
\end{equation*}
If $S' \in \overline{E^{n - \ell}_{\bm{a}}}$ is a subset of $S \in \overline{E^{n-m}_{\bm{a}}}$, then $S'$ is a closed subset of $S$. Therefore, the set $S \setminus \Big(\bigcup_{\substack{S' \in \overline{E^{n - \ell}_{\bm{a}}} \\ \ell \leq m-1}} S'\Big)$ is an open subset of the smooth manifold $S$, which implies that $S \setminus \Big(\bigcup_{\substack{S' \in \overline{E^{n - \ell}_{\bm{a}}} \\ \ell \leq m-1}} S'\Big)$ is a smooth manifold. If $S_1$ and $S_2$ are distinct elements of $\overline{E^{n-m}_{\bm{a}}}$, then
\begin{equation*}
    \Big(S_1 \setminus \bigcup_{\substack{S' \in \overline{E^{n - \ell}_{\bm{a}}} \\ \ell \leq m-1}} S'\Big) \cap \Big(S_2 \setminus \bigcup_{\substack{S' \in \overline{E^{n - \ell}_{\bm{a}}} \\ \ell \leq m-1}} S'\Big) = \emptyset \,,
\end{equation*}
which completes the proof.
\end{proof}

For the remainder of Section \ref{sec:generic}, let $\{\base^m_{\bm{a}}\}_{m=0}^{n}$ be defined as in \eqref{eq:strat_def}, and define
\begin{equation}\label{eq:Ya_def}
    \mathcal{Y}_{\bm{a}}= \bigcup_{m=0}^n \mathcal{Y}^m_{\bm{a}}\,,
\end{equation}
where $\mathcal{Y}^m_{\bm{a}}$ is the set of connected components of $\base^m_{\bm{a}} \setminus \base^{m-1}_{\bm{a}}$ (with $\base^{-1}_{\bm{a}} := \emptyset$).

\begin{lemma}\label{lem:constantorder_strata}
Let $A$ be defined as in \eqref{eq:Adef}. If $\bm{a} \in A$ is such that each $Y \in \mathcal{Y}_{\bm{a}}$ is a manifold, then the simplex order induced by $f$ is constant in each $Y$. That is, there is a strict partial order $\prec_Y$ on $\K$ such that $\prec_{f_{\bm{a}}(\cdot, y)}$ is the same as $\prec_Y$ for all $y \in Y$.
\end{lemma}

\begin{proof}
Let $Y \in \mathcal{Y}_{\bm{a}}$. The set $Y$ is connected by definition. Because $Y$ is a manifold, it is also path-connected. For each pair $(\sigma, \tau)$ of simplices, we have by construction that $Y \cap I_{\bm{a}}(\sigma, \tau)$ equals either $\emptyset$ or $Y$. (In fact, this statement holds for all $\bm{a} \in A$ and does not require $Y$ to be a manifold.) By Lemma \ref{lem:constantorder}, the simplex order is constant in $Y$.
\end{proof}

\begin{lemma}\label{lem:tangent_spaces}
For almost every $\bm{a} \in A$ (where $A$ is defined as in \eqref{eq:Adef}), we have that $\bigcap_{r = 1}^m I_{\bm{a}}(\sigma_{i_r}, \sigma_{j_r})$ is a submanifold of $B$ and
\begin{equation}\label{eq:tangent_spaces_intersection}
    T_\bp \Big(\bigcap_{r = 1}^m I_{\bm{a}}(\sigma_{i_r}, \sigma_{j_r}) \Big) = \bigcap_{r = 1}^m T_\bp \Big( I_{\bm{a}}(\sigma_{i_r}, \sigma_{j_r})\Big)
\end{equation}
for all points $\bp \in \bigcap_{r = 1}^m I_{\bm{a}}(\sigma_{i_r}, \sigma_{j_r})$ and all sets $\{(i_r, j_r)\}_{r = 1}^m$ of index pairs such that $\{i_r, j_r\} \neq \{i_s, j_s\}$ if $r \neq s$.
\end{lemma}
\begin{proof}
    Because there are finitely many sets of index pairs, it suffices to fix a set $\{(i_r, j_r)\}_{r = 1}^m$ of index pairs and show that \eqref{eq:tangent_spaces_intersection} holds for all $y \in \bigcap_{r = 1}^m I_{\bm{a}}(\sigma_{i_r}, \sigma_{j_r})$ for almost every $\bm{a} \in A$. 
    By Lemmas \ref{lem:reduced_form} and \ref{lem:manifolds}, the set $\bigcap_{r = 1}^m I_{\bm{a}}(\sigma_{i_r}, \sigma_{j_r})$ is a manifold for almost every $\bm{a} \in A$.
    By Lemma \ref{lem:open_cover}, there is a finite open cover $\{U_k\}_{k=1}^K$ of $B$ such that for each $k$, there is a disjoint partition $\bigcup_{\ell} J_{\ell, k} = \{1, \ldots, m\}$ such that $\{i_r, j_r \mid r \in J_{\ell_1, k} \} \cap \{i_r, j_r \mid r \in J_{\ell_2, k} \} = \emptyset$ if $\ell_1 \neq \ell_2$ and
    \begin{equation*}
        \pi\Big(\bigcap_{r \in J_{\ell, k}} (M_{\bm{a}, i_r} \cap M_{\bm{a}, j_r})\Big) \cap U_k = \bigcap_{r \in J_{\ell, k}} I_{\bm{a}}(\sigma_{i_r}, \sigma_{j_r}) \cap U_k
    \end{equation*}
    for all $\ell$, where $\pi$ is the projection $\pi: \base \times \mathbb{R} \to \base$.\footnote{Recall that an arbitrary intersection $\bigcap_r I_{\bm{a}}(\sigma_{i_r}, \sigma_{j_r}) \subseteq \base$ does not necessarily lift to an intersection $\bigcap_{r=1}^m (M_{\bm{a}, i} \cap M_{\bm{a}, j}) \subseteq \base \times \mathbb{R}$. \Cref{lem:open_cover} says that in local neighborhoods $U \subseteq \base$, we can partition the set $\{1, \ldots, m\}$ into subsets $J_\ell$ such that the intersection $\bigcap_{r \in J_\ell} I_{\bm{a}}(\sigma_{i_r}, \sigma_{j_r})$ does lift to a subset of the intersection $\bigcap_{r \in J_\ell} (M_{\bm{a}, i} \cap M_{\bm{a}, j})$.} Because the number $K$ of open sets is finite, it suffices to fix $U_k$ and show that \eqref{eq:tangent_spaces_intersection} holds for all $\bp \in \bigcap_{r = 1}^m I_{\bm{a}}(\sigma_{i_r}, \sigma_{j_r}) \cap U_k$ for almost every $\bm{a} \in A$.

    By Lemma \ref{lem:inter_component}, we have
    \begin{equation*}
         T_\bp\Big(\bigcap_{r = 1}^m I_{\bm{a}}(\sigma_{i_r}, \sigma_{j_r}) \Big) 
         = T_\bp\Big( \bigcap_{\ell} \bigcap_{r \in J_{\ell, k}}  I_{\bm{a}}(\sigma_{i_r}, \sigma_{j_r}) \Big)
         = \bigcap_{\ell} T_\bp \Big(\bigcap_{r \in J_{\ell, k}} \Big( I_{\bm{a}}(\sigma_{i_r}, \sigma_{j_r})\Big)\Big)
    \end{equation*}
    for all $\bp \in \bigcap_{r = 1}^m I_{\bm{a}}(\sigma_{i_r}, \sigma_{j_r}) \cap U_k$ for almost every $\bm{a} \in A$. By Lemma \ref{lem:intra_component}, we have
    \begin{equation*}
        \bigcap_{\ell} T_\bp \Big(\bigcap_{r \in J_{\ell, k}} \Big( I_{\bm{a}}(\sigma_{i_r}, \sigma_{j_r})\Big)\Big) = \bigcap_{\ell}\bigcap_{r \in J_{\ell, k}} T_\bp (I_{\bm{a}}(\sigma_{i_r}, \sigma_{j_r})
    \end{equation*}
    for all $\bp \in \bigcap_{r = 1}^m I_{\bm{a}}(\sigma_{i_r}, \sigma_{j_r}) \cap U_k$ for almost every $\bm{a} \in A$.
\end{proof}

For any strict partial order $\prec$ on $\K$, we define
\begin{align}\label{eq:Zdef}
    Z^{\prec}_{\bm{a}} := &\{\bp \in B \mid f_{\bm{a}}(\sigma, \bp) < f_{\bm{a}}(\tau, \bp) \text{ if } \sigma \prec \tau \notag \\ 
    &\qquad \text{ and } f_{\bm{a}}(\sigma, \bp) = f_{\bm{a}}(\tau, \bp) \text{ if } \sigma \not\prec \tau \text{ and } \tau \not\prec \sigma\}\,.
\end{align}
That is, $Z_{\bm{a}}^{\prec}$ is the subset of $\base$ such that for all $z$ in $Z_{\bm{a}}^{\prec}$, the strict partial order $\prec_{f_{\bm{a}}(\cdot, z)}$ is the same as $\prec$. 

\begin{lemma}\label{lem:locallyfinite}
Let $A$ be defined as in \eqref{eq:Adef}. If $\bm{a} \in A$ is such that
\begin{enumerate}
\item every $Y \in \mathcal{Y}_{\bm{a}}$ is a manifold, where $\mathcal{Y}_{\bm{a}}$ is defined as in \eqref{eq:Ya_def}, 
\item $M_{\bm{a}, i} \pitchfork M_{\bm{a}, j}$ for all $i \neq j$, where $M_{\bm{a}, i}$ is defined as in \eqref{eq:Ma_def},
\item $\bigcap_{r= 1}^m I_{\bm{a}}(\sigma_{i_r}, \sigma_{j_r})$ is a manifold for all sets $\{(i_r, j_r\}_{r =1}^m$ of index pairs, and
\item $T_{\bp}\Big(\bigcap_{r=1}^m I_{\bm{a}}(\sigma_{i_r}, \sigma_{j_r})\Big) = \bigcap_{r= 1}^m T_{\bp} \Big(I_{\bm{a}}(\sigma_{i_r}, \sigma_{j_r})\Big)$ for all sets $\{(i_r, j_r\}_{r =1}^m$ of index pairs and all $\bp \in \bigcap_{r= 1}^m I_{\bm{a}}(\sigma_{i_r}, \sigma_{j_r})$,
\end{enumerate}
then $\mathcal{Y}_{\bm{a}}$ is locally finite.
\end{lemma}
\begin{proof}
    Let $\bp$ be a point in $\base$. There are finitely many strict partial orders $\prec_1, \ldots, \prec_i$ on $\K$. By Lemma \ref{lem:Z_decomp}, we have that for each $j \in \{1, \ldots, i\}$, there is a subset $\mathcal{Y}^{\prec_j}_{\bm{a}} \subseteq \mathcal{Y}_{\bm{a}}$ such that $Z^{\prec_j}_{\bm{a}} = \bigcup_{Y \in \mathcal{Y}^{\prec_j}_{\bm{a}}} Y$. For each $j$, the point $\bp$ has a neighborhood $U_j$ that intersects at most one $Y \in \mathcal{Y}^{\prec_j}_{\bm{a}}$ by Lemma \ref{lem:uniqueY}. Therefore $\bigcap_{j=1}^i U_j$ is a neighborhood of $\bp$ that intersects at most $i$ elements of $\mathcal{Y}_{\bm{a}}$.
\end{proof}

\begin{lemma}\label{lem:frontier}
Let $A$ be defined as in \eqref{eq:Adef}. If $\bm{a} \in A$ is such that 
\begin{enumerate}
\item every $S \in \overline{E_{\bm{a}}^{n - \ell}}$ is an $\ell$-dimensional smooth submanifold for every $\ell \in \{1, \ldots, n\}$, where $\overline{E_{\bm{a}}^{n - \ell}}$ is defined as in \eqref{eq:reduced_sets},
\item $M_{\bm{a}, i} \pitchfork M_{\bm{a}, j}$ for all $i \neq j$, where $M_{\bm{a}, i}$ is defined as in \eqref{eq:Ma_def}, 
\item $\bigcap_{i=r}^m I_{\bm{a}}(\sigma_{i_r}, \sigma_{j_r})$ is a manifold for all sets $\{(i_r, j_r\}_{r =1}^m$ of index pairs, and
\item $T_{\bp}\Big(\bigcap_{i=r}^m I_{\bm{a}}(\sigma_{i_r}, \sigma_{j_r})\Big) = \bigcap_{i=1}^m T_{\bp} \Big(I_{\bm{a}}(\sigma_{i_r}, \sigma_{j_r})\Big)$ for all sets $\{(i_r, j_r\}_{r =1}^m$ of index pairs and all $\bp \in \bigcap_{i=r}^m I_{\bm{a}}(\sigma_{i_r}, \sigma_{j_r})$,
\end{enumerate}
then $\mathcal{Y}_{\bm{a}}$ satisfies the Axiom of the Frontier in Definition \ref{def:stratification}, where $\mathcal{Y}_{\bm{a}}$ is defined as in \eqref{eq:Ya_def}
\end{lemma}
\begin{proof}
    By Lemma \ref{lem:stratum}, each $Y \in \mathcal{Y}_{\bm{a}}$ is a manifold. Let $Y_{\alpha}$ be an element of $\mathcal{Y}_{\bm{a}}$. It suffices to show that if $Y_{\beta} \neq Y_{\alpha}$ is another element of $\mathcal{Y}_{\bm{a}}$ and $Y_{\beta} \cap \partial Y_{\alpha} \neq \emptyset$, where $\partial Y_{\alpha}$ denotes the boundary of the manifold $Y_{\alpha}$, then $Y_{\beta} \subseteq \partial Y_{\alpha}$.

    By Lemma \ref{lem:constantorder_strata}, the simplex order induced by $f$ is constant on each $Y$, so there is a strict partial order $\prec_{\alpha}$ on $\K$ such that $\prec_{f_{\bm{a}}(\cdot, y)}$ is the same as $\prec_{\alpha}$ for all $y \in Y_{\alpha}$. Let $Z^{\prec_{\alpha}}_{\bm{a}}$ be defined as in \eqref{eq:Zdef}. By Lemma \ref{lem:Z_decomp}, there is a subset $\mathcal{Y}^{\prec_{\alpha}}_{\bm{a}} \subseteq \mathcal{Y}_{\bm{a}}$ such that $Y_{\alpha} \in \mathcal{Y}^{\prec_{\alpha}}_{\bm{a}}$ and $Z^{\prec_{\alpha}}_{\bm{a}}= \bigcup_{Y \in \mathcal{Y}^{\prec_{\alpha}}_{\bm{a}}} Y$. We have $\partial Z^{\prec_{\alpha}}_{\bm{a}} = \bigcup_{Y \in \mathcal{Y}^{\prec_{\alpha}}_{\bm{a}}} \partial Y$ because $\mathcal{Y}_{\bm{a}}$ is locally finite by Lemma \ref{lem:locallyfinite}. Therefore,
    \begin{equation}\label{eq:frontier_eq1}
        Y_{\beta} \cap \partial Z^{\prec_{\alpha}}_{\bm{a}} = \bigcup_{Y \in \mathcal{Y}^{\prec_{\alpha}}_{\bm{a}}} (Y_{\beta} \cap \partial Y)\,.
    \end{equation}
    By Lemmas \ref{lem:Z_decomp} and \ref{lem:Zboundary}, we have that if $Y \in \mathcal{Y}_{\bm{a}}$ intersects $\partial Z^{\prec_{\alpha}}_{\bm{a}}$, then $Y \subseteq \partial Z^{\prec_{\alpha}}_{\bm{a}}$. Therefore $Y_{\beta} \subseteq \partial Z^{\prec_{\alpha}}_{\bm{a}}$ because $Y_{\beta} \cap \partial Z^{\prec_{\alpha}}_{\bm{a}}$ contains $Y_{\beta} \cap \partial Y_{\alpha} \neq \emptyset$. Together with \eqref{eq:frontier_eq1}, this shows that 
    \begin{equation}\label{eq:frontier_eq2}
        Y_{\beta} = \bigcup_{Y \in \mathcal{Y}^{\prec_{\alpha}}_{\bm{a}}} (Y_{\beta} \cap \partial Y)\,.
    \end{equation}
    By Lemma \ref{lem:uniqueY}, every point in $\base$ has a neighborhood that intersects at most one $Y \in \mathcal{Y}^{\prec_{\alpha}}_{\bm{a}}$, so $\partial Y' \cap \partial Y = \emptyset$ for all $Y, Y' \in \mathcal{Y}^{\prec_{\alpha}}_{\bm{a}}$ such that $Y \neq Y'$. Because $Y_{\beta}$ is connected (by definition) and $Y_{\beta} \cap \partial Y_{\alpha} \neq \emptyset$, we must have that $Y_{\beta} \cap \partial Y = \emptyset$ for all $Y \in \mathcal{Y}^{\prec_{\alpha}}_{\bm{a}}$ such that $Y \neq Y_{\alpha}$. By \eqref{eq:frontier_eq2},
    \begin{equation*}
        Y_{\beta} = Y_{\beta} \cap \partial Y_{\alpha} \subseteq \partial Y_{\alpha}\,.
    \end{equation*}
\end{proof}

\begin{theorem}\label{thm:stratified}
Let $\base$ be a smooth compact $n$-dimensional manifold. For every $\bm{a} \in A$, define $\{\base^m_{\bm{a}}\}_{m=0}^n$ as in \eqref{eq:strat_def}, with $A$ defined as in \eqref{eq:Adef}. For almost every $\bm{a} \in A$, we have that $\{\base^m_{\bm{a}}\}_{m=0}^n$ is a stratification of $\base$. In each stratum $Y$, the simplex order induced by $f_{\bm{a}}$ is constant. (In other words, there is a strict partial order $\prec_Y$ on $\K$ such that $\prec_{f_{\bm{a}}(\cdot, y)}$ is the same as $\prec_Y$ for all $y \in Y$.) Consequently, the set $\{(\sigma_b, \sigma_d)\}$ of (birth, death) simplex pairs is constant in each stratum $Y$ and for any $\bp \in Y$, the persistence diagram $PD(f_\bp)$ consists of the diagonal (with infinite multiplicity) and the multiset $\{(f(\sigma_b, \bp), f(\sigma_d, \bp))\}$.
\end{theorem}
\begin{proof}
By Lemmas \ref{lem:reduced_form}, \ref{lem:manifolds}, \ref{lem:stratum}, \ref{lem:tangent_spaces}, \ref{lem:locallyfinite}, and \ref{lem:frontier}, $\{\base_{\bm{a}}^m\}_{m = 0}^n$ is a stratification of $\base$ for almost every $\bm{a} \in A$. By Lemma \ref{lem:constantorder_strata}, the simplex order induced by $f_{\bm{a}}$ is constant in each stratum $Y \in \mathcal{Y}_{\bm{a}}$ whenever $\{\base_{\bm{a}}^m\}_{m = 0}^n$ is a stratification of $\base$. The last statement of Theorem \ref{thm:stratified} follows from Lemma \ref{lem:only_order}.
\end{proof}


\section{Monodromy in PD Bundles}\label{sec:monodromy}

\begin{definition}[Local section]
Let $(E, \base, \pi)$ be a PD bundle. A \emph{local section} is a continuous map $s: U \to E$, where $U$ is an open set in $\base$ and $\pi \circ s(\bp) = \bp$ for all $\bp \in U$.
\end{definition}
\noindent For example, consider a vineyard, in which $\base$ is an interval $I$ in $\mathbb{R}$. Let $(t_0, T)$ be an open interval in $I$. A local section in the vineyard is a map $s:(t_0, T) \to E$ that parameterizes an open subset of one of the vines (a curve in $\mathbb{R}^3$).
\begin{definition}[Global section]
Let $(E, \base, \pi)$ be a PD bundle. A \emph{global section} is a continuous map $s: \base \to E$ with $\pi \circ s(\bp) = \bp$ for all $\bp \in \base$. In particular, a \emph{nontrivial global section} is a global section $s: \base \to E$ such that there exists a $\bp_* \in \base$ for which $s(\bp_*)$ is not on the diagonal of $PD(f_{\bp_*})$.
\end{definition}
\noindent In a vineyard, every local section can be extended to a global section. In other words, we can trace out how the persistence of a single homology class changes over $\base = [t_0, t_1] \subseteq \mathbb{R}$, so there are individual ``vines'' in the vineyard. We will show that local sections of a PD bundle cannot necessarily be extended to global sections. Consequently, a PD bundle does not necessarily have a decomposition of the form \eqref{eq:vine_decomp}; if it does, then each $\gamma$ is a global section.

\begin{proposition}\label{prop:monodromy}
There is a PD bundle $(E, \base, \pi)$ for which no 
 nontrivial global sections exist.
\end{proposition}
\begin{proof}
The proof is constructive. Let $\K$ be the simplicial complex in Figure \ref{fig:inconsistent_K}, which has vertices $0$, $1$, $2$, and $3$. Let $a$ be the edge with vertices $(0, 1)$, let $b$ be the edge with vertices $(0, 2)$, let $c$ be the triangle with vertices $(0, 1, 2)$, and let $d$ be the triangle with vertices $(0, 2, 3)$. 

Let $f: \K \times \mathbb{R}^2 \to \mathbb{R}$ be a continuous fibered filtration function that satisfies the following conditions:
\begin{align*}
    f(c, (x, y)), f(d, (x, y)) &>  f(a, (x, y)), f(b, (x, y)) > 0 \qquad \text{for all } (x, y) \in \mathbb{R}^2\,, \\
    f(a, (x, y)) &> f(b, (x, y))\,, \qquad y > 0\,, \\
    f(a, (x, y)) &< f(b, (x, y))\,, \qquad y < 0\,, \\
    f(c, (x, y)) &> f(d, (x, y))\,, \qquad x > 0\,, \\
    f(c, (x, y)) &< f(d, (x, y))\,, \qquad x < 0\,, \\
    f(\sigma, (x, y)) &= 0\,, \qquad \text{for all other } \sigma, \text{ for all } x, y\,. 
\end{align*}
The conditions on the fibered filtration function $f$ are illustrated in Figure \ref{fig:filtfunction}.
\begin{figure}
    \centering
    \subfloat[]{\includegraphics[width=.2\textwidth]{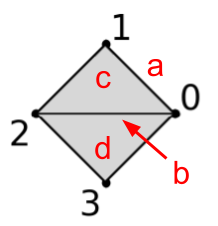}    \label{fig:inconsistent_K}}
    \hspace{5mm}
    \subfloat[]{\includegraphics[width = .35\textwidth]{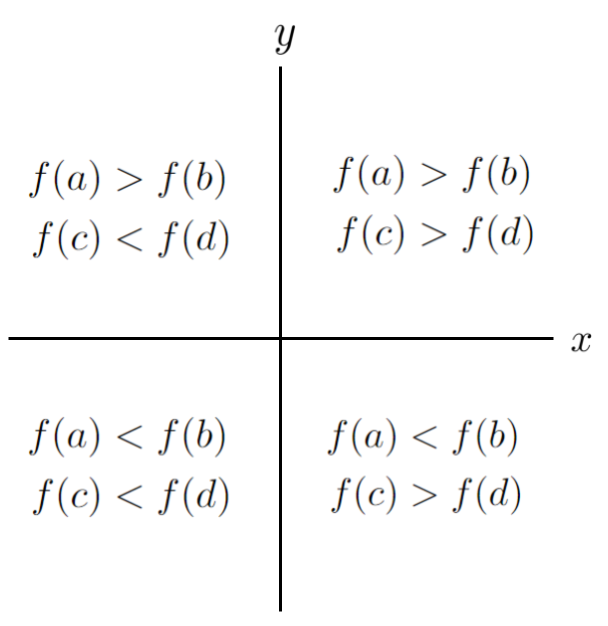}\label{fig:filtfunction}}
    \hspace{5mm}
    \subfloat[]{\includegraphics[width=.35\textwidth]{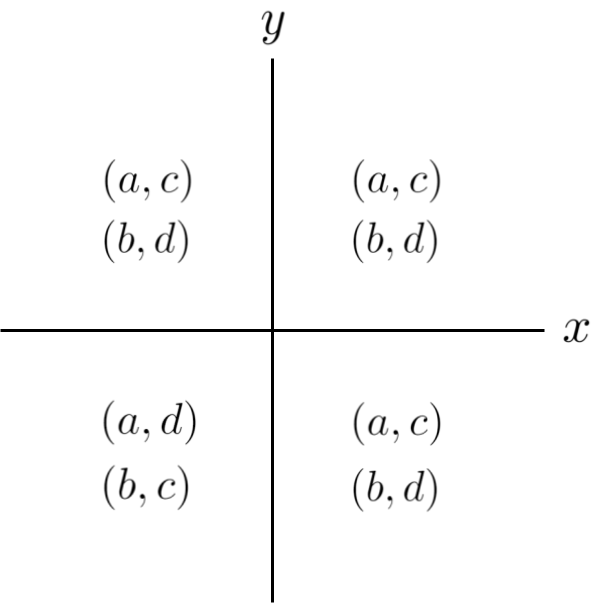}\label{fig:inconsistent_pairings}}
    \caption{(A) The \scom\ $\K$ that is defined in the proof of Proposition \ref{prop:monodromy}. (B) The conditions on the fibered filtration $f: \K \times \mathbb{R}^2 \to \mathbb{R}$ that is defined in the proof of Proposition \ref{prop:monodromy}. (C) The (birth, death) simplex pairs in each quadrant for the 1D PH.}
    \label{fig:example_def}
\end{figure}
These conditions imply that simplices $a$ and $b$ swap their order along the $x$-axis and the simplices $c$ and $d$ swap their order along the $y$-axis.

In Figure \ref{fig:inconsistent_pairings}, we list the (birth, death) simplex pairs for the 1D PH in each quadrant. In quadrants $1$, $2$, and $4$, the simplex pairs are $(a, c)$ and $(b, d)$. In quadrant $3$, the simplex pairs are $(a, d)$ and $(b, c)$.

Let $(E, \mathbb{R}^2, \pi)$ be the corresponding PD bundle, where $E = \{ ((x, y), z) \mid (x, y) \in \mathbb{R}^2, z \in PD_1(f(\cdot, (x, y)))\}$ is the total space and $\pi$ is the projection to $\mathbb{R}^2$. We will show that $(E, \mathbb{R}^2, \pi)$ has no 
 nontrivial global sections.

If $s: \base \to E$ is a global section and $s(\bp_*)$ is on the diagonal of $\pdgm(f(\cdot, \bp_*))$ for some $\bp_* \in \base$, then $s$ is a trivial section because $f(\sigma_b, \bp) \neq f(\sigma_d, \bp)$ for all $\bp$ for any (birth, death) simplex pairs $(\sigma_b, \sigma_d)$ at $\bp$. Therefore, if $s : \base \to E$ is a nontrivial global section, $s(\bp)$ is not on the diagonal of $\pdgm(f(\cdot, \bp))$ for any $\bp \in \base$.

Suppose that $\gamma: [0, 1] \to E$ is a continuous path such that $\gamma(u)$ is not on the diagonal of $\pdgm(f(\cdot, (x, y)))$ for any $(x, y) \in \mathbb{R}^2$ and such that
\begin{equation}\label{eq:base_proj}
    \pi \circ \gamma(u) = \theta(u) := (\cos (2\pi u + \pi/4), \sin (2 \pi u + \pi/4)) \in S^1\,.
\end{equation}
That is, $\pi \circ \gamma$ is a parameterization of $S^1$ that starts in the first quadrant of $\mathbb{R}^2$ at $\bp_0 = (\sqrt{2}/2, \sqrt{2}/2)$. The path $\gamma$ is determined uniquely by its initial condition $\gamma(0)$. The simplex pairs in the first quadrant are $(a, c)$ and $(b, d)$, so $\gamma(0)$ equals either $(\bp_0, (f(a, \bp_0), f(c, \bp_0)))$ or $(\bp_0, (f(b, \bp_0), f(d, \bp_0)))$. In Figure \ref{fig:monodromy}, we illustrate the two possibilities for the path $\gamma$. If $\gamma(0) = (\bp_0, (f(a, \bp_0), f(c, \bp_0)))$, then $\gamma(1) = (\bp_0, (f(b,\bp_0), f(d, \bp_0)))$; if $\gamma(0) = (\bp_0, (f(b, \bp_0), f(d, \bp_0)))$, then $\gamma(1) = (\bp_0, (f(a, \bp_0), f(c, \bp_0)))$. In either case, $\gamma(0) \neq \gamma(1)$.

\begin{figure}
    \centering
    \includegraphics[width = \textwidth]{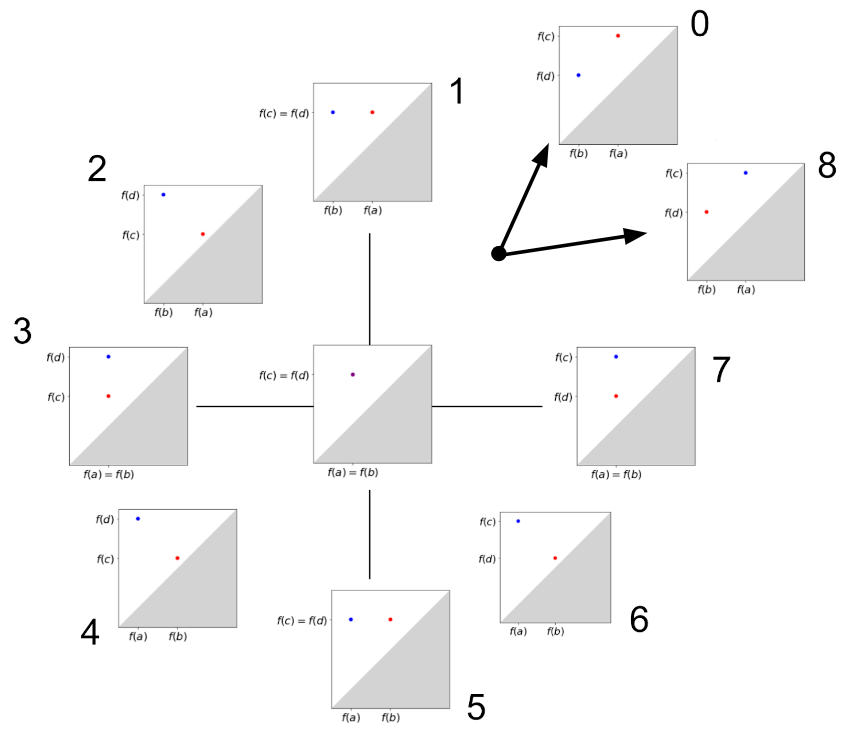}
    \caption{A visualization of the two choices for the path $\gamma:[0, 1] \to E$ in the proof of Proposition \ref{prop:monodromy}, where $E$ is the total space of the PD bundle. We show $10$ fibers of the PD bundle for various points $\bp \in \mathbb{R}^2$. The first nine PDs (labeled $0$ through $8$) are PDs for points $\bp \in S^1$; the $k$th PD is the PD at $t_k = \theta(u_k)$, where $u_k = k/8$ and $\theta(u)$ is the parameterization of $S^1$ given by \eqref{eq:base_proj}. Note that $\theta(0) = (\sqrt{2}/2, \sqrt{2}/2) \in S^1$. The two choices for the path $\gamma(u)$, which depend only on the choice of $\gamma(u_0)$, are shown in red and blue, respectively. For each $k$, the red (respectively, blue) dot in the $k$th PD is equal to $\gamma(u_k)$ when $\gamma(u_0)$ is the red (respectively, blue) point in the $0$th PD. Observe that $\gamma(u_0) \neq \gamma(u_8)$ even though $\bp_0 = \bp_8$. The unlabeled PD at the origin is the PD for the origin in $\mathbb{R}^2$, at which there is a ``singularity.''}
    \label{fig:monodromy}
\end{figure}

To obtain a contradiction, suppose that there were a nontrivial global section $s: \mathbb{R}^2 \to E$. Let $\gamma: [0, 1] \to E$ be the path $\gamma = s \circ \theta$, where $\theta$ is the parameterization of $S^1$ defined in \eqref{eq:base_proj}. Then $\gamma(0) \neq \gamma(1)$ because $\gamma$ is a path satisfying \eqref{eq:base_proj}, but $\gamma(0) = s(\bp_0) = \gamma(1)$.
\end{proof}
Note that we will use the fibered filtration $f: \K \times \mathbb{R}^2 \to \mathbb{R}$ that was constructed in Proposition \ref{prop:monodromy} as a running example throughout Section \ref{sec:cellularsheaf}.
\begin{remark}\label{rmk:S1}
Even when $\dim(\base) = 1$, it is not guaranteed that a nontrivial global section exists. To see this, consider the 1D PH of the fibered filtration function above restricted to $S^1 \subseteq \mathbb{R}^2$. In this example, $\dim(\base) = 1$ and a 
nontrivial global section does not exist.
\end{remark}
\begin{remark}\label{rmk:singularity}
In the example of \Cref{prop:monodromy}, it was the ``singularity'' (the point $(0, 0) \in \mathbb{R}^2$ at which the PD had a point of multiplicity greater than one) that prevented the existence of a nontrivial global section. Restricting the PD bundle to $\base' := \R^2 \setminus \{(0, 0)\}$ yields a true fiber bundle; each fiber is homeomorphic to the disjoint union of a line (the diagonal) and two points (the off-diagonal points). It is well known that fiber bundles over contractible spaces are trivial (i.e., the total space is homeomorphic to the product of the base and a fiber.) However, $\base'$ is not contractible, so our PD bundle restricted to $\base'$ is not guaranteed to be trivial. Indeed, what we showed in \Cref{prop:monodromy} is that it is not. By comparison to a vineyard,
\begin{enumerate}
    \item Singularities do not occur for generic fibered filtrations $f: \K \times \R \to \R$. A singularity occurs at $\bp_* \in \base$ when there are two (birth, death) simplex pairs $(\sigma_b^1, \sigma_d^1)$, $(\sigma_b^2, \sigma_d^2)$ at $\bp_*$ such that $\bp_* \in I(\sigma_b^1, \sigma_b^2) \cap I(\sigma_d^1, \sigma_d^2)$. When $\dim(\base) = 1$, the intersection $I(\sigma_b^1, \sigma_b^2) \cap I(\sigma_d^1, \sigma_d^2)$ is empty in the generic case, so singularities do not typically exist when $\dim(\base) = 1$.
    \item Even when singularities do occur in a vineyard, there should not be monodromy in the vineyard. As in the example above, we can remove the singularities from $\R$ to obtain a disjoint union of intervals $\base_1, \ldots, \base_m$ such that when we restrict the vineyard to a $\base_i$, we have a fiber bundle. Intervals in $\R$ are contractible, so these fiber bundles must be trivial. By continuity, we can glue together the fiber bundles over each $\base_i$ to see that our PD bundle cannot have monodromy.
\end{enumerate}
\end{remark}

\section{A Compatible Cellular Sheaf}
For a given fibered filtration function that induces a stratification of $\base$ as in Theorem \ref{thm:stratified}, we construct a compatible cellular sheaf. We discuss a motivating example in Section \ref{sec:comb}, and give the definition in Section \ref{sec:cellularsheaf}.

\subsection{A motivating example}\label{sec:comb}
Again consider the example in the proof of Proposition \ref{prop:monodromy}, and also again consider the path $\gamma:[0, 1] \to E$ that is determined uniquely by the choice of
\begin{equation*}
\gamma(0) \in \{(\bp_0, (f(a, \bp_0), f(c, \bp_0))), (\bp_0, (f(b, \bp_0), f(d, \bp_0)))\}\,,
\end{equation*}
where $\bp_0 = (\sqrt{2}/2, \sqrt{2}/2)$. The two possibilities for the path $\gamma$ are illustrated in Figure \ref{fig:monodromy}. For example, if $\gamma(0) = (\bp_0, (f(a, \bp_0), f(c, \bp_0)))$, then
\begin{equation*}
    \gamma(u) = \begin{cases}
        \Big(\theta(u), (f(a, \theta(u)), f(c, \theta(u)))\Big)\,, & u \in [0, 1/8] \\
        \Big(\theta(u), (f(a, \theta(u)), f(c, \theta(u)))\Big)\,, & u \in [1/8, 3/8] \\
        \Big(\theta(u), (f(b, \theta(u)), f(c, \theta(u)))\Big)\,, & u \in [3/8, 5/8] \\
        \Big(\theta(u), (f(b, \theta(u)), f(d, \theta(u)))\Big)\,, & u \in [5/8, 7/8] \\
        \Big(\theta(u), (f(b, \theta(u)), f(d, \theta(u)))\Big)\,, & u \in [7/8, 1]\,, \\
    \end{cases}
\end{equation*}
where $\theta(u)$ is the parameterization of $S^1$ given by \eqref{eq:base_proj}. As we move through the quadrants of $\mathbb{R}^2$, the point in the PD that represents the pair $(a, c)$ in the first quadrant becomes the point that represents the pair $(a, c)$ in the second quadrant, which becomes the point that represents the pair $(b, c)$ in the third quadrant, which becomes the point that represents the pair $(b, d)$ in the fourth quadrant, which becomes the point that represents the pair $(b, d)$ in the first quadrant. One can do a similar analysis for the case in which $\gamma(0) = (\bp_0, (f(b, \bp_0), f(d, \bp_0)))$.

This analysis yields a bijection of the (birth, death) simplex pairs for any pair of adjacent quadrants. We illustrate the bijections in Figure \ref{fig:pair_matchings}. The bijection between the simplex pairs in a given quadrant and one of its adjacent quadrants is the same as the bijection defined by the update rule of Cohen-Steiner et al. \cite{vineyards} for updating the simplex pairs in a vineyard. A combinatorial perspective on Proposition \ref{prop:monodromy} is that there is no consistent way of choosing a simplex pair in each quadrant such that if $(\sigma_b, \sigma_d)$ is the (birth, death) simplex pair chosen for a given quadrant and $(\tau_b, \tau_d)$ is the (birth, death) simplex pair chosen for an adjacent quadrant, then $(\sigma_b, \sigma_d)$ and $(\tau_b, \tau_d)$ are matched in the bijection between the two quadrants. This is because if we choose an initial simplex pair in one of the quadrants and then walk in a circle through the other quadrants, then the simplex pair at which we finish is different from the initial simplex pair. For example, if we start at $(a, c)$ in the first quadrant, then we finish at $(b, d)$ when we return to the first quadrant, and vice versa. This is a discrete way of illustrating the non-existence of a nontrivial global section.
\begin{figure}
    \centering
    \includegraphics[width=.5\textwidth]{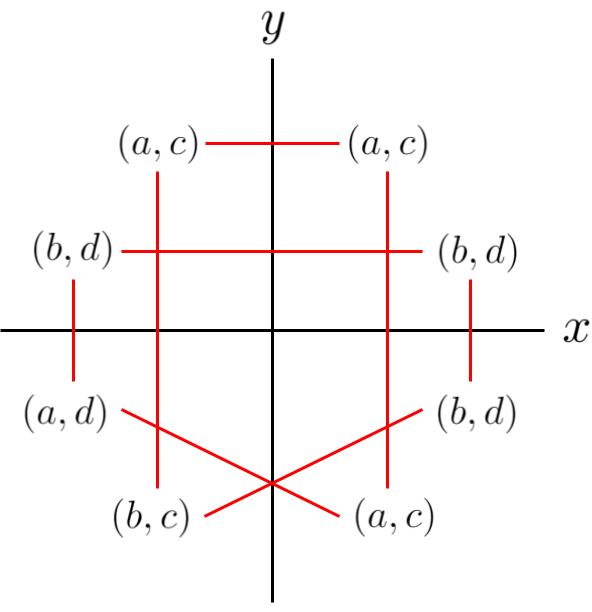}
    \caption{The (birth, death) simplex pairs in each quadrant for the 1D PH of the fibered filtration function in the proof of Proposition \ref{prop:monodromy} (see also Figure \ref{fig:example_def}). For each pair of adjacent quadrants, there is a bijection between their sets of simplex pairs; this bijection is equal to the bijection given by the update rule of Cohen-Steiner et al. \cite{vineyards}. The red lines connect simplex pairs that are in bijection with each other.}
    \label{fig:pair_matchings}
\end{figure}

\subsection{Definition of a compatible cellular sheaf}\label{sec:cellularsheaf}
I generalize the discussion in Section \ref{sec:comb} to fibered filtration functions of the form $f:\K \times B \to \mathbb{R}$ that have a stratification (see Definition \ref{def:stratification}) of $B$ such that in each stratum $Y$, the simplex order that is induced by $f$ is constant. (In other words, there is a strict partial order $\prec_Y$ on $\K$ such that $\prec_{f(\cdot, y)}$ is the same as $\prec_Y$ for all $y \in Y$.) Theorem \ref{thm:stratified} guarantees that such a stratification exists for generic fibered filtration functions, and Proposition \ref{prop:polyhedrons_constant} guarantees that such a stratification exists for all piecewise-linear fibered filtration functions. We denote the set of strata by $\mathcal{Y} = \{Y_{\alpha}\}_{\alpha \in J}$ for some index set $J$.

\begin{definition}\label{def:sheaf}
    Suppose that $\mathcal{F}$ is a {\bf Set}-valued cellular sheaf whose cell complex, stalks, and morphisms are of the following form:

    \begin{enumerate}
        \item {\bf The cell complex:} The cell complex on which $\mathcal{F}$ is constructed is the graph $G$ such that there is a vertex $v_{\alpha}$ for each stratum $Y_{\alpha} \in \mathcal{Y}$ and an edge $e_{\beta, \alpha} = (v_{\beta}, v_{\alpha})$ if $Y_{\beta} \in \mathcal{Y}$ is a face of $Y_{\alpha}$. The $0$-cells of the cell complex are the vertices of $G$ and the $1$-cells are the edges of $G$.

    \vspace{3mm}
    
        \item {\bf The stalks:} Let $S_{\alpha}$ denote the set of (birth, death) simplex pairs for a stratum $Y_{\alpha}$. The stalk at a $0$-cell $v_{\alpha} \in G$ is $\mathcal{F}(v_{\alpha}) := S_{\alpha}$. For a $1$-cell $e_{\beta, \alpha} \in G$, where $Y_{\beta}$ is a face of $Y_{\alpha}$, the stalk at $e_{\beta, \alpha}$ is $\mathcal{F}(e_{\beta, \alpha}) := S_{\alpha}$.

\vspace{3mm}

        \item {\bf The morphisms:} If $Y_{\beta} \in \mathcal{Y}$ is a face of $Y_{\alpha} \in \mathcal{Y}$, then the morphism $\mathcal{F}_{v_{\beta} \leq e_{\beta, \alpha}}: \mathcal{F}(v_{\alpha}) \to \mathcal{F}(e_{\beta, \alpha})$ is the identity map and the morphism $\mathcal{F}_{v_{\beta} \leq e_{\beta, \alpha}}: \mathcal{F}(v_{\beta}) \to \mathcal{F}(e_{\beta, \alpha})$ is
        \begin{equation*}
            \mathcal{F}_{v_{\beta} \leq e_{\beta, \alpha}} := \phi^{\idx_{\beta},\, \idx_{\alpha}}\,,
        \end{equation*}
where $\phi^{\idx_{\beta},\, \idx_{\alpha}}$ is of the form in \eqref{eq:bijection_def} and $\idx_{\alpha}: \K \to \{1, \ldots, N\}$ and $\idx_{\beta}: \K \to \{1, \ldots, N\}$ are the simplex indexings (recall Definition \ref{def:spx_indexing}) on $Y_{\alpha}$ and $Y_{\beta}$, respectively. (Recall that by \Cref{lem:constantorder_strata}, the simplex order induced by $f$ is constant within $Y_{\alpha}$ and within $Y_{\beta}$.)
    \end{enumerate}
Then the cellular sheaf $\mathcal{F}$ is a \emph{compatible cellular sheaf} for the fibered filtration function $f: \K \times \base \to \mathbb{R}$.
\end{definition}

It is not guaranteed that there is a \emph{unique} compatible cellular sheaf for a given fibered filtration function $f$. Although the cell complex (the graph $G$) is determined uniquely by $f$, the stalks and morphisms are not. Recall from Definition \ref{def:spx_indexing} that the simplex indexing that is induced by $f$ may depend on an intrinsic indexing $\sigma_1, \ldots, \sigma_N$ of the simplices in $\K$. (The intrinsic indexing breaks ties when two simplices have the same filtration value.) For a stratum $Y_{\alpha}$ such that $f(\sigma, y) = f(\tau, y)$ for all $y \in Y_{\alpha}$ for some pair $(\sigma, \tau)$ of simplices, the simplex indexing $\idx_{f(\cdot, Y_{\alpha})}$ depends on the intrinsic indexing, so $S_{\alpha}$ may not be determined uniquely by $f$. If $S_{\alpha}$ is not determined uniquely by $f$, then for any face $Y_{\beta}$ of $Y_{\alpha}$, the stalks $\mathcal{F}(v_{\alpha})$ and $\mathcal{F}(e_{\beta, \alpha})$ are not determined uniquely by $f$. As discussed in Remark \ref{rmk:nonunique}, a bijection $\phi^{\idx_{\beta}, \idx_{\alpha}}$ of the form in \eqref{eq:bijection_def} is not determined uniquely by $f$ if $\idx_{\beta}$ and $\idx_{\alpha}$ differ by more than the transposition of two consecutive simplices. Therefore, the morphism $\mathcal{F}_{v_\beta \leq e_{\beta, \alpha}}$ is not necessarily determined uniquely by $f$.

However, many aspects of the stalks and morphisms \emph{are} determined uniquely by $f$. Suppose that $Y_{\beta} \in \mathcal{Y}$ is a face of $Y_{\alpha} \in \mathcal{Y}$. If $f(\sigma, y) \neq f(\tau, y)$ for all $y$ in $Y_{\alpha}$ and all simplices $\sigma \neq \tau$, then the simplex indexing $\idx_{f(\cdot, Y_{\alpha})}$ is determined uniquely by $f$, so the stalks $\mathcal{F}(v_{\alpha})$ and $\mathcal{F}(e_{\beta, \alpha})$ are determined uniquely by $f$. Theorem \ref{thm:stratified} guarantees that this is the generic case when $Y_{\alpha}$ is an $n$-dimensional stratum (where $n = \dim(\base)$). There are also conditions under which a morphism is determined uniquely by $f$. The morphism $\mathcal{F}_{v_{\alpha} \leq e_{\beta, \alpha}}: S_\alpha \to S_\alpha$ must be the identity map. The morphism $\mathcal{F}_{v_{\alpha} \leq e_{\beta, \alpha}} := \phi^{\idx_{\beta}, \idx_{\alpha}}$ 
is determined uniquely by $f$ when $\idx_{\beta}$ and $\idx_{\alpha}$ differ by the transposition of two consecutive simplices. Theorem \ref{thm:stratified} guarantees that this is the generic case when $Y_{\beta}$ is a ``top-dimensional'' face of $Y_{\alpha}$ (i.e., when $\dim(Y_{\beta}) = \dim(Y_{\alpha}) - 1$). 

\begin{example}\label{ex:sheaf}
    Again consider a fibered filtration function $f: \K \times \mathbb{R}^2 \to \mathbb{R}$ of the form defined in \Cref{prop:monodromy}, with $\K$ defined as in Figure \ref{fig:inconsistent_K} with $N = 11$ simplices. We construct a compatible cellular sheaf $\mathcal{F}$ as follows.
    \begin{enumerate}
        \item {\bf The cell complex:} The strata are the open quadrants $Q_1, \ldots, Q_4$, the open half-axes $A_{12}, A_{23}, A_{34}, A_{14}$ with $A_{ij} = (\partial Q_i \cap \partial Q_j ) \setminus \{\bm{0}\}$, and the point $\bm{0} \in \mathbb{R}^2$. The associated graph $G$ (the cell complex for $\mathcal{F}$) has a vertex $v_{Q_i}$ for the $i$th quadrant, a vertex $v_{A_{ij}}$ for the $(i, j)$th half-axis, and a vertex $v_{\bm{0}}$ for the point $\bm{0}$. The graph $G$ has edges $(v_{A_{ij}}, v_{Q_i})$ and  $(v_{A_{ij}}, v_{Q_j})$ for each half-axis $A_{ij}$, and it has an edge $(v_{\bm{0}}, v)$ for every vertex $v \in G$ such that $v \neq v_{\bm{0}}$.

\vspace{3mm}

        \item {\bf The stalks: } We index the simplices of $\K$ such that $\sigma_8 = a$, $\sigma_9 = b$, $\sigma_{10} = c$, and $\sigma_{11} = d$, where $a$, $b$, $c$, $d$ are the simplices defined in Figure \ref{fig:inconsistent_K}. The stalk at $v_{Q_1}$ is $S_{Q_1} = \{(a, c), (b, d)\}$. The vertices $v_{Q_2}$ and $v_{A_{12}}$ have the same stalk $\{(a, c), (b, d)\}$; the vertices $v_{Q_3}$, $v_{A_{23}}$, and $v_{34}$ have the same stalk $\{(a, d), (b, c)\}$; and the vertices $v_{Q_4}$ and $v_{A_{14}}$ have the same stalk $\{(b, d), (a, c)\}$. The stalks at the edges of $G$ are determined by the stalks at the vertices. In this example, the stalks at the vertices or edges that correspond to 2D strata are determined uniquely by $f$, but the stalks at the vertices and edges that correspond to 0D or 1D strata depend on our choice of intrinsic indexing.

\vspace{3mm}

        \item {\bf The morphisms:}  There are only three distinct nonidentity morphisms. The first two are
\begin{align*}
    \mathcal{F}_{v_{A_{23}} \leq e_{(A_{23}, Q_2 )}},\, \mathcal{F}_{v_{\bm{0}} \leq e_{(\bm{0}, Q_2)}}: \{(a, c), (b, d)\} &\to \{(a, d), (b, c)\} \\
    (a, c) &\mapsto (b, c) \\
    (b, d) &\mapsto (a, d)\,, \\
    \mathcal{F}_{v_{A_{34}} \leq e_{(A_{34}, Q_4)}},\, \mathcal{F}_{v_{\bm{0}} \leq e_{(\bm{0}, Q_4)}}: \{(a, d), (b, c)\} &\to \{(a, c), (b, d)\} \\
    (a, d) &\mapsto (a, c) \\
    (b, c) &\mapsto (b, d)\,.
\end{align*}
The third distinct nonidentity morphism is a map
\begin{equation*}
    \mathcal{F}_{v_{\bm{0}} \leq e_{(\bm{0}, Q_1)}} :\{(a, d), (b, c)\} \to \{(a, c), (b, d)\}\,.
\end{equation*}
As we move from $\bm{0}$ to $Q_1$, we swap the simplex indices of $a$ and $b$ and we also swap the simplex indices of $c$ and $d$ (in the simplex indexing induced by $f$). The morphism is not canonical because the bijection $\phi^{\idx_{\bm{0}}, \idx_{Q_1}}$ depends on whether one first swaps $a$ and $b$ or one first swaps $c$ and $d$. Therefore, we may define either
\begin{align*}
    \mathcal{F}_{v_{\bm{0}} \leq e_{(\bm{0}, Q_1)}}: (a, d) &\mapsto (a, c)\,,\\
    (b, c) &\mapsto (b, d)
\end{align*}
or
\begin{align*}
    \mathcal{F}_{v_{\bm{0}} \leq e_{(\bm{0}, Q_1)}}: (a, d) &\mapsto (b, d)\,,\\
    (b, c) &\mapsto (a, c)\,.
\end{align*}
Both choices results in a compatible cellular sheaf.
    \end{enumerate}
\end{example}

\subsection{Sections of the cellular sheaf}
Let $\mathcal{F}$ be any compatible cellular sheaf for a fibered filtration $f: \K \times \base \to \mathbb{R}$. We write
\begin{equation}\label{eq:Fbd}
    \mathcal{F}_{v_{\beta} \leq e_{\beta, \alpha}} = \Big(\mathcal{F}^b_{v_{\beta} \leq e_{\beta, \alpha}},\, \mathcal{F}^d_{v_{\beta} \leq e_{\beta, \alpha}}\Big)\,,
\end{equation}
where $\mathcal{F}^b_{v_{\beta} \leq e_{\beta, \alpha}}: S_{\beta} \to S_{\alpha}$ maps a pair $(\sigma_b, \sigma_d) \in S_{\beta}$ to the birth simplex of $\mathcal{F}_{v_{\beta} \leq e_{\beta, \alpha}}((\sigma_b, \sigma_d))$ and  $\mathcal{F}^d_{v_{\beta} \leq e_{\beta, \alpha}}: S_{\beta} \to S_{\alpha}$ maps $(\sigma_b, \sigma_d) \in S_{\beta}$ to the death simplex of $\mathcal{F}_{v_{\beta} \leq e_{\beta, \alpha}}((\sigma_b, \sigma_d))$. (Recall that $S_{\alpha}$, $S_{\beta}$ are the stalks at the vertices $v_{\alpha}$, $v_{\beta}$ that are associated with the strata $Y_{\alpha}$, $Y_{\beta}$.)

In this subsection, we show that one can view sections of $\mathcal{F}$ as sections of the associated PD bundle. 

\begin{lemma}\label{lem:f_edge}
Let $Y_{\beta}$ be a face of $Y_{\alpha}$. Assume that $f: \K \times \base \to \R$ is continuous (i.e., $f(\sigma, \cdot)$ is continuous for all simplices $\sigma$ in $\K$). Then for any point $\bp \in Y_{\beta}$ and any pair $(\sigma_b, \sigma_d)$ in $\mathcal{F}(v_{\beta})$, we have
\begin{equation}\label{eq:f_edge}
    \big(f(\sigma_b, \bp), f(\sigma_d, \bp)\Big) = \Big(f(\mathcal{F}^b_{v_{\beta} \leq e_{\beta, \alpha}}((\sigma_b, \sigma_d)), \bp)\,, f(\mathcal{F}^d_{v_{\beta} \leq e_{\beta, \alpha}}((\sigma_b, \sigma_d)), \bp)\Big)\,,
\end{equation}
where $\mathcal{F}^b$ and $\mathcal{F}^d$ are defined as in \eqref{eq:Fbd}.
\end{lemma}
\begin{proof}
If the simplex orders in $Y_{\alpha}$ and $Y_{\beta}$ differ only by a transposition of simplices $(\sigma, \tau)$ with consecutive indices in the orderings, then we must have $f(\sigma, \bp) = f(\tau, \bp)$ for all $\bp \in Y_{\beta}$ because $f$ is continuous and $Y_{\beta} \subseteq \overline{Y_{\alpha}}$. By definition, $\mathcal{F}_{v_{\beta} \leq e_{\beta, \alpha}}$ is either the identity map or the map that swaps $\sigma$ and $\tau$ in the pairs that contain them. In either case, \eqref{eq:f_edge} holds because $f(\sigma, \bp) = f(\tau, \bp)$ for all $\bp \in Y_{\beta}$. Equation \eqref{eq:f_edge} holds in general because $\mathcal{F}_{v_{\beta} \leq e_{\beta, \alpha}}$ is defined as the composition of such maps.
\end{proof}

The following proposition says that a global section of a compatible cellular sheaf $\mathcal{F}$ corresponds to a global section of the PD bundle.
\begin{proposition}\label{prop:discrete_secs}
Let $z_0$ be a non-diagonal point in $PD_\homdim(f_{\bp_0})$ for some $\bp_0 \in \base$, let $(\sigma_b, \sigma_d)$ be the (birth, death) simplex pair such that $z_0 = (f(\sigma_b, \bp_0), f(\sigma_d, \bp_0))$, and let $Y_0$ be the stratum that contains $\bp_0$. Suppose that $\mathcal{F}$ is a compatible cellular sheaf, and let $v_0$ be the vertex in the graph $G$ (see Definition \ref{def:sheaf}) that is associated with $Y_0$. If there is a global section $\overline{s}$ of the cellular sheaf $\mathcal{F}$ such that $\overline{s}(v_0) = (\sigma_b, \sigma_d)$, then there is a global section $s$ of the PD bundle such that $s(\bp_0) = z_0$.
\end{proposition}
\begin{proof}
Let $\overline{s}$ be a global section of the cellular sheaf $\mathcal{F}$ such that $\overline{s}(v_0) = (\sigma_b, \sigma_d)$. For every stratum $Y_{\alpha}$, we write 
\begin{equation*}
    \overline{s}(v_{\alpha}) = (\overline{s}_b(v_{\alpha}), \overline{s}_d(v_{\alpha}))\,,
\end{equation*}
where $\overline{s}_b(v_{\alpha})$ is the birth simplex of $\overline{s}(v_{\alpha})$ and $\overline{s}_d(v_{\alpha})$ is the death simplex of $\overline{s}(v_{\alpha})$. Let $Y:\base \to \{Y_{\alpha}\}$ be the function that maps $\bp \in \base$ to the unique stratum $Y_{\alpha}$ that contains it.

We define $s: \base \to E$ to be the function
\begin{equation*}
    s(\bp) := \Big(\bp, f(\overline{s}_b(v_{Y(\bp)}), \bp), f(\overline{s}_d(v_{Y(\bp)}), \bp)\Big)\,.
\end{equation*}
To show that $s: B \to E$ is a global section of the PD bundle, it remains to show that it is continuous. The function $s\vert_{Y_{\alpha}}$ is continuous for all strata $Y_{\alpha}$ because $f(\sigma, \cdot)$ is continuous for all simplices $\sigma \in \K$. Therefore, it suffices to show that $s\vert_{\overline{Y_{\alpha}}}$ is continuous on each face $Y_{\beta}$ of $Y_{\alpha}$. Because $\overline{s}$ is a section of the cellular sheaf, 
\begin{equation*}
    \overline{s}(v_{\alpha}) = \mathcal{F}_{v_{\beta} \leq e_{\beta, \alpha}}(\overline{s}(v_{\beta}))\,.
\end{equation*}
By Lemma \ref{lem:f_edge},
\begin{equation*}
    \Big(f(\overline{s}_b(v_{\beta}), \bp), f(\overline{s}_d(v_{\beta}), \bp)\Big) = \Big( f(\mathcal{F}^b_{v_{\beta} \leq e_{\beta, \alpha}}(\overline{s}(v_{\beta})), \bp), f(\mathcal{F}^d_{v_{\beta} \leq e_{\beta, \alpha}}(\overline{s}(v_{\beta})), \bp)\Big)
\end{equation*}
for all points $\bp \in Y_{\beta}$. Therefore,
\begin{equation*}
     \Big(f(\overline{s}_b(v_{\beta}), \bp), f(\overline{s}_d(v_{\beta}), \bp)\Big) =  \Big(f(\overline{s}_b(v_{\alpha}), \bp), f(\overline{s}_d(v_{\alpha}), \bp)\Big)
\end{equation*}
for all $\bp \in Y_{\beta}$, which completes the proof.
\end{proof}


\section{Conclusions}\label{sec:conclusion}

\subsection{Summary} In this \difference{chapter}{paper}, I introduced the concept of a persistence diagram (PD) bundle, a framework that can be used to study the persistent homology of a fibered filtration function (i.e.,  set of filtrations parameterized by an arbitrary ``base space'' $\base$). Special cases of PD bundles include vineyards \cite{vineyards}, the persistent homology transform (PHT) \cite{pht}, the fibered barcode of a multiparameter persistence module \cite{fibered}, and the barcode-decorated merge tree \cite{dmt}.

In Theorem \ref{thm:stratified}, I proved that if $\base$ is a smooth compact manifold, then for generic fibered filtrations, $\base$ is stratified so that the simplex order is constant within each stratum. When such a stratification exists, the PD bundle is determined by the PDs at a locally finite (or finite, if $\base$ is compact) subset of points in $\base$. In Proposition \ref{prop:polyhedrons_constant}, I showed that every piecewise-linear PD bundle has such a stratification into polyhedra. This polyhedral stratification is utilized in \difference{Chapter \ref{ch:pd_bundle_alg} (\cite{pd_bundle_alg})}{\cite{pd_bundle_alg}} in an algorithm for computing piecewise-linear PD bundles.

I showed that, unlike vineyards, which PD bundles generalize, not every local section of a PD bundle can be extended to a global section (see Proposition \ref{prop:monodromy}). The implication is that PD bundles do not necessarily decompose into ``vines'' in the way that vineyards do (see \eqref{eq:vine_decomp}).

Lastly, I introduced a cellular sheaf that is compatible with a given PD bundle. In Proposition \ref{prop:discrete_secs}, I proved that one can determine whether a local section can be extended to a global section by determining whether or not there is an associated global section of a compatible cellular sheaf. A compatible cellular sheaf is a discrete mathematical data structure for summarizing the data in a PD bundle.

\subsection{Discussion}
For a given fibered filtration function $f$ with a stratification as in Theorem \ref{thm:stratified}, I defined a compatible cellular sheaf $\mathcal{F}$ over a graph $G$. It is tempting to instead define an associated cellular sheaf directly on the stratification of $\base$. In particular, when $f$ is piecewise linear, the strata are polyhedra, so the stratification is guaranteed to be a cellular decomposition. One could certainly define stalks $\mathcal{F}(Y_{\alpha})$ and functions $\mathcal{F}(Y_{\beta}) \to \mathcal{F}(Y_{\beta})$ in the same way as in Definition \ref{def:sheaf}. The problem is that $\mathcal{F}$ would not necessarily satisfy the composition condition (see \eqref{eq:composition}). For instance, this issue occurs in Example \ref{ex:sheaf} for the same reason that the morphism $\mathcal{F}_{v_{\bm{0}} \leq  e_{(\bm{0}, Q_1)}}$ in the example is not canonical (see the discussion in Example \ref{ex:sheaf}).

Additionally, I note that one could have defined a compatible cellular cosheaf rather than a sheaf.

\subsection{Future research} I conclude with some questions and proposals for future work:
\begin{itemize}
    \item What are the conditions under which a PD bundle must have a decomposition of the form \eqref{eq:vine_decomp}?
    \vspace{3mm}
    
    \noindent I conjecture that
    if $\base \setminus \base^*$ is contractible, where $\base^*$ is the set of singularities (i.e., points $\bp_* \in \base$ at which there is a point in \pdgm$(f_{\bp_*})$ with multiplicity greater than \difference{$1$}{one}), then there is a decomposition of the form \eqref{eq:vine_decomp}. See the discussion in Remark \ref{rmk:singularity}.
    
    \vspace{3mm}
    
    \item What algebraic or computational methods can we use to analyze global sections and to compute obstructions to the existence of global sections?
    
    \vspace{3mm}
    
    \noindent It may help to consider the cellular-sheaf perspective from Section \ref{sec:cellularsheaf}, which turns the question into a discrete problem that one can study computationally. One can also generalize Turner's vineyard-module perspective \cite{turner2023}.

    \vspace{3mm}

    \item PHT is a PD bundle over the base space $\base = S^n$. Are there constructible sets $M \subseteq \mathbb{R}^{n+1}$ for which the associated PHT exhibits monodromy? What is the geometric interpretation (in terms of $M$)?

    \vspace{3mm}

    \item Arya et al. \cite{pht_sheaf} showed that the PHT of a constructible set $M$ can be calculated by ``gluing together'' the PHT of smaller, simpler subsets of $M$. Can one generalize these results to all PD bundles?

    \vspace{3mm}
    
    \item When are PD bundles ``stable''?
    \vspace{3mm}
    
    \noindent PD bundles are ``fiberwise stable'' in the sense that if $f_1, f_2: \K \times \base \to \R$ are two fibered filtrations, then the bottleneck distance between $\pdgm(f_1(\cdot, \bp))$ and $\pdgm(f_2(\cdot, \bp))$ is bounded above by $\norm{f_1 - f_2}_{\infty}$ for all $\bp \in \base$ \cite{PD_stability}. However, this does not guarantee that the global structure of a PD bundle is stable. For example, it is well known that the structure of a vineyard is not stable (see Appendix \ref{sec:vineyard_unstable} for an example). However, vineyards are stable for generic $1$-parameter filtrations; if none of the vines intersect (which is the generic case), then sufficiently small perturbations of the filtration result in small perturbations of each vine in the vine decomposition (see \eqref{eq:vine_decomp}). I expect that an analogous result holds for generic fibered filtration functions over any base space $\base$.

    \vspace{3mm}
    
    \item It will also be interesting to study real-world applications of PD bundles, such as the examples that were mentioned in Section \ref{sec:examples}.
\end{itemize}


\section*{Acknowledgements}
I am very grateful for discussions with Andrew Blumberg, which led to the investigation of monodromy in PD bundles.
I also thank Ryan Grady and Karthik Viswanathan for helpful discussions.

\appendix

\section*{Appendix}
\renewcommand{\thesection}{A}
\subsection{Vineyard instability}\label{sec:vineyard_unstable}

For any $\epsilon > 0$, we construct two $1$-parameter filtration functions $f^+_{\epsilon}$, $f^-_{\epsilon}$ that are $\epsilon$-perturbations of each other (that is, $\vert f^+_{\epsilon}(\sigma, \bp) - f^-_{\epsilon}(\sigma, \bp) \vert < \epsilon$ for all simplices $\sigma \in \K$ and all points $\bp \in \base$) but such that for any bijection between the vines in the respective vineyards, not all of the matched vines are close to each other. In fact, we can define $f^+_{\epsilon}$ and $f^-_{\epsilon}$ so that their vines are arbitrarily far apart.

We construct our example by restricting the filtration function from Section \ref{sec:monodromy} to certain paths through $\mathbb{R}^2$. Let $\K$ and $f: \K \times \mathbb{R}^2 \to \mathbb{R}$ be defined as in the proof of \Cref{prop:monodromy}. (See Figure \ref{fig:filtfunction}.) Because $f$ is continuous, we have that for any $\epsilon > 0$, there is a $\delta > 0$ such that $\vert f(\sigma, \bp) - f(\sigma, \bm{0})\vert < \epsilon/2$ when $\norm{\bp} < \sqrt{2}\delta$ and $\sigma$ is any simplex in $\K$. We define the paths
\begin{align*}
    \gamma^+_{\epsilon}(t) &:= \begin{cases}
        (t, t) \,, & |t| \geq \delta \\
        (-\delta, \delta + 2t)\,, & -\delta < t < 0\\\
        (-\delta + 2t, \delta)\,, & 0 \leq t < \delta\,,
    \end{cases}\\
    \gamma_{\epsilon}^-(t) &:= \begin{cases}
        (t, t)\,, & |t| \geq \delta \\
        (\delta + 2t, - \delta)\,,  & -\delta < t < 0\\
        (\delta, -\delta + 2t)\,, & 0 \leq t < \delta\,.
    \end{cases}
\end{align*}
See \Cref{fig:instability_paths} for a plot of the paths $\gamma_{\epsilon}^{\pm}(t)$.
\begin{figure}
    \centering
    \includegraphics[width = .45\textwidth]{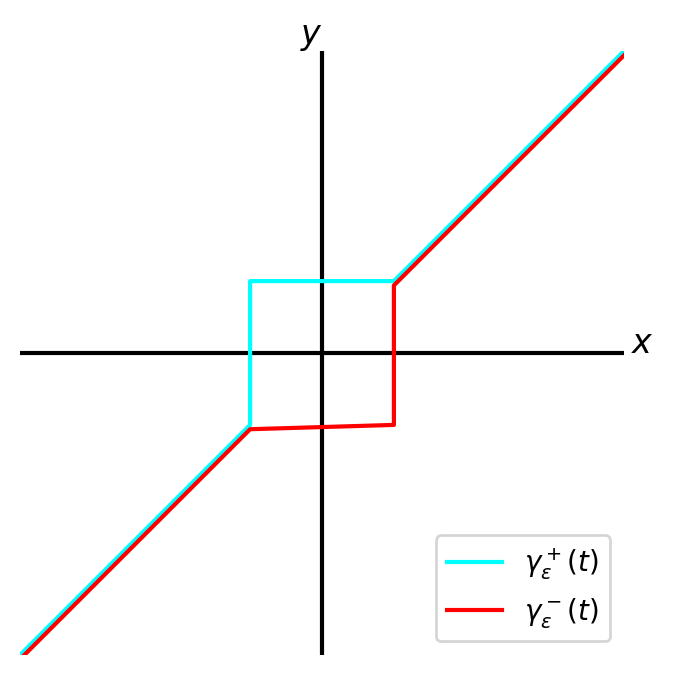}
    \caption{The paths $\gamma_{\epsilon}^+(t)$ and $\gamma_{\epsilon}^-(t)$.}
    \label{fig:instability_paths}
\end{figure}

Let $f^{\pm}_{\epsilon} : \K \times \mathbb{R} \to \mathbb{R}$ be the $1$-parameter filtration functions defined by $f^{\pm}(\sigma, t) := f(\sigma, \gamma_{\epsilon}^{\pm}(t))$. By construction, the filtrations $f^+_{\epsilon}$ and $f^-_{\epsilon}$ are $\epsilon$-perturbations of each other. 

Let $V^+$ and $V^-$ be the vineyards for $f^+_{\epsilon}$ and $f^-_{\epsilon}$, respectively, for the $1$st degree PH. The vineyards $V^{\pm}$ each have two vines $v_1^{\pm}$, $v_2^{\pm}$, which are paths $v_i^{\pm}: \mathbb{R} \to \mathbb{R}^3$. The vines are
\begin{align*}
    v^+_1(t) &= \begin{cases}
        (f(a, \gamma^+(t)), f(d, \gamma^+(t)))\,, & t \leq - \delta/2 \\
        f(b, \gamma^+(t)), f(d, \gamma^+(t)))\,, & t > -\delta/2\,,
    \end{cases}\\
   v^+_2(t) &= \begin{cases}
        (f(b, \gamma^+(t)), f(c, \gamma^+(t)))\,, & t \leq - \delta/2 \\
        f(a, \gamma^+(t)), f(c, \gamma^+(t)))\,, & t > -\delta/2\,,
    \end{cases}\\
    v^-_1(t) &= \begin{cases}
        (f(a, \gamma^-(t)), f(d, \gamma^-(t)))\,, & t \leq - \delta/2 \\
        f(a, \gamma^-(t)), f(c, \gamma^-(t)))\,, & t > -\delta/2\,,
        \end{cases}\\
    v^-_2(t) &= \begin{cases}
        (f(b, \gamma^-(t)), f(c, \gamma^-(t)))\,, & t \leq - \delta/2 \\
        f(b, \gamma^-(t)), f(d, \gamma^-(t)))\,, & t > -\delta/2\,.
        \end{cases}
\end{align*}

There is no bijection $\phi : \{1, 2\} \to \{1, 2\}$ such that $v_1^+$ and $v_2^+$ are close to $v_{\phi(1)}^-$ and $v_{\phi(2)}^-$, respectively. This is because
\begin{align*}
    \norm{v_1^+(t) - v_1^-(t)}^2 &= |f(b, (t, t)) - f(a, (t, t))|^2 + |f(d, (t, t))) - f(c, (t, t))|^2 \,, \qquad t > -\delta/2\,,\\
    \norm{v_1^+(t) - v_2^-(t)}^2 &= |f(b, (t, t)) - f(a, (t, t))|^2 + |f(d, (t, t))) - f(c, (t, t))|^2 \,, \qquad t \leq - \delta/2\,,\\
    \norm{v_2^+(t) - v_2^-(t)}^2 &= |f(b, (t, t)) - f(a, (t, t))|^2 + |f(d, (t, t))) - f(c, (t, t))|^2 \,, \qquad t > -\delta/2\,,\\
    \norm{v_2^+(t) - v_1^-(t)}^2 &= |f(b, (t, t)) - f(a, (t, t))|^2 + |f(d, (t, t))) - f(c, (t, t))|^2 \,, \qquad t \leq - \delta/2
\end{align*}
and we can define $f$ so that $|f(b, (t, t)) - f(a, (t, t))|$ and $|f(d, (t, t)) - f(c, (t, t))|$ are arbitrarily large for $t \neq 0$.

\subsection{Technical Details of Section \ref{sec:partition}}\label{sec:bundle_details}
All notation is defined as in Section \ref{sec:partition}. 

The first series of lemmas is used to prove Lemma \ref{lem:tangent_spaces}, which shows that for almost every $\bm{a} \in A$, the tangent space of the intersection of sets $I_{\bm{a}}(\sigma_{i_r}, \sigma_{j_r})$ is equal to the intersection of their tangent spaces.

\begin{lemma}\label{lem:manifold_transverse}
    For almost every $\bm{a} \in A$, we have
    \begin{equation*}
        T_p\Big(\bigcap_{j \in J} M_{\bm{a}, j}\Big) = \bigcap_{j \in J} T_p(M_{\bm{a}, j})
    \end{equation*}
    for all $J \subseteq \{1, \ldots, N\}$ and all $p \in \bigcap_{j \in J} M_{\bm{a}, j}$.
\end{lemma}
\begin{proof}
    Because there are finitely many subsets of $\{1, \ldots, N\}$, it suffices to show that for a given $J \subseteq \{1, \ldots, N\}$, we have $T_p\Big(\bigcap_{j \in J} M_{\bm{a}, j}\Big) = \bigcap_{j \in J} T_p(M_{\bm{a}, j})$ for all $p \in \bigcap_{j \in J} M_{\bm{a}, j}$ for almost every $\bm{a} \in A$. Let $\{j_i\}_{i=1}^k$ be the elements of $J$, where $j_i < j_{i+1}$ for all $i$. Because transverse intersections are generic, we have $M_{\bm{a}, j_i} \pitchfork ( M_{\bm{a}, j_1} \cap \cdots \cap M_{\bm{a}, j_{i-1}})$ for every $i$ for almost every $\bm{a} \in A$. For such an $\bm{a} \in A$, we have 
    \begin{equation*}
        T_p\Big(\bigcap_{j \in J} M_{\bm{a}, j}\Big) = T_p(M_{\bm{a}, j_k}) \cap T_p\Big( \bigcap_{i=1}^{k-1} M_{\bm{a}, j_i} \Big)
    \end{equation*}
    because $M_{\bm{a}, j_k} \pitchfork ( M_{\bm{a}, j_1} \cap \cdots \cap M_{\bm{a}, j_{k-1}})$. Therefore,
    \begin{equation*}
        T_p\Big(\bigcap_{j \in J} M_{\bm{a}, j}\Big) = \bigcap_{i=1}^k T_p(M_{\bm{a}, j_i})
    \end{equation*}
    by induction on $i$.
\end{proof}

\begin{lemma}\label{lem:open_cover}
    Let $\bm{a} \in A$, and let $\{ (i_r, j_r) \}_{r = 1}^m$ be a set of index pairs such that $\{i_r, j_r\} \neq \{i_s, j_s\}$ if $r \neq s$. If $\base$ is a compact manifold, then there is a finite open cover $\{U_k\}_{k=1}^K$ and a disjoint partition $\bigcup_\ell J_{\ell, k} = \{1, \ldots, m\}$ for each $k$ such that
    \begin{equation*}
        \{i_r, j_r \mid r \in J_{\ell_1, k}\} \cap \{i_r, j_r \mid r \in J_{\ell_2, k}\} = \emptyset
    \end{equation*}
    if $\ell_1 \neq \ell_2$ and 
    \begin{equation*}
        \pi\Big(\bigcap_{r \in J_{\ell, k}} (M_{\bm{a}, i_r} \cap M_{\bm{a}, j_r})\Big) \cap U_k = \bigcap_{r \in J_{\ell, k}} I_{\bm{a}}(\sigma_{i_r}, \sigma_{j_r}) \cap U_k
    \end{equation*}
    for all $\ell$, where $\pi$ is the projection $\pi: B \times \mathbb{R} \to B$.
\end{lemma}
\begin{proof}
    Suppose that $y \in \bigcap_{r = 1}^m I_{\bm{a}}(\sigma_{i_r}, \sigma_{j_r})$. Let $J^0_{\ell} := \{\ell\}$ be an initial disjoint partition of $\{1, \ldots, m\}$. By definition, $\pi(M_{\bm{a}, i_\ell} \cap M_{\bm{a}, j_\ell}) = I_{\bm{a}}(\sigma_{i_\ell}, \sigma_{j_\ell})$. If $i_{\ell_1} = i_{\ell_2}$ for some $\ell_1 \neq \ell_2$, then $f(\sigma_{j_{\ell_1}}, y) = f(\sigma_{i_{\ell_1}}, y) = f(\sigma_{j_{\ell_2}}, y)$, so $y \in \pi(M_{{\bm{a}}, i_{\ell_1}} \cap M_{\bm{a}, j_{\ell_1}} \cap M_{\bm{a}, i_{\ell_2}} \cap M_{\bm{a}, j_{\ell_2}})$. We combine $J^0_{\ell_1}$ and $J^0_{\ell_2}$ into a single subset of the partition, and we iterate until we obtain a disjoint partition $\{J_{\ell, y}\}_{\ell}$ of $\{1, \ldots, m\}$ such that
    \begin{equation*}
        \{i_r, j_r \mid r \in J_{\ell_1, y}\} \cap \{i_r, j_r \mid r \in J_{\ell_2, y}\} = \emptyset
    \end{equation*}
    if $\ell_1 \neq \ell_2$ and 
    \begin{equation*}
        y \in \pi \Big( \bigcap_{r \in J_{\ell, y}} (M_{\bm{a}, i_r} \cap M_{\bm{a}, j_r})\Big)
    \end{equation*}
    for all $\ell$. Therefore, for each $\ell$, there is a neighborhood $U_{\ell, y}$ such that
    \begin{equation*}
         \pi \Big( \bigcap_{r \in J_{\ell, y}} (M_{\bm{a}, i_r} \cap M_{\bm{a}, j_r})\Big) \cap U_{\ell, y} = \bigcap_{r \in J_{\ell, y}} I_{\bm{a}}(\sigma_{i_r}, \sigma_{j_r}) \cap U_{\ell, k}\,.
    \end{equation*}
    Set $U_y := \bigcap_{\ell} U_{\ell, y}$. Because $B$ is compact, there is a finite open cover $\{U_k\}_{k=1}^K$, which has the desired properties by construction.
\end{proof}
The following lemma will be repeatedly used in Lemma \ref{lem:inter_component}.
\begin{lemma}\label{lem:RV_restriction}
Suppose that $g: \base \to \mathbb{R}$ is a smooth map and $y \in \R$ is a regular value with preimage $Z_y$. If $Z \subseteq \base$ is a submanifold such that $Z \pitchfork Z_y$, then $y$ is a regular value of $g \vert_Z: Z \to \R$.
\end{lemma}
\begin{proof}
    At any $z \in Z_y$, we have $\ker(dg_z) = T_z Z_y$. Therefore, if $z \in Z \cap Z_y$, then $T_z Z \subseteq \ker(dg_z)$ only if $T_z Z \subseteq T_z Z_y$. Because $y$ is a regular value of $g:\base \to \R$, we have
    \begin{equation*}
        \dim(T_z(Z_y)) = \dim \base - 1\,,
    \end{equation*}
    so $T_z(Z_y)$ is a strict subset of $T_z \base$. Because $Z \pitchfork Z_y$, we have
    \begin{equation*}
        T_z Z + T_z Z_y = T_z \base\,,
    \end{equation*}
    so $T_z Z$ cannot be a subset of $T_z Z_y$. Therefore, $T_zZ$ is not a subset of $\ker(dg_z)$. This implies that $dg_z\vert_{T_z Z}$ is a surjection because $\dim \R = 1$. Therefore, $y$ is a regular value of $g\vert_Z$.
\end{proof}
\begin{lemma}\label{lem:inter_component}
 Let $\{ (i_r, j_r) \}_{r = 1}^m$ be a set of index pairs such that $\{i_r, j_r\} \neq \{i_s, j_s\}$ if $r \neq s$. For almost every $\bm{a} \in A$, we have that if 
 \begin{enumerate}
     \item $\bigcup_{\ell} J_\ell = \{1, \ldots, m\}$
is a disjoint partition with 
\begin{equation*}
    \{i_r, j_r \mid r \in J_{\ell_1}\} \cap \{i_r, j_r \mid r \in J_{\ell_2}\} = \emptyset
\end{equation*}
for $\ell_1 \neq \ell_2$ and
\item $U$ is an open set in $B$ such that
    \begin{equation}\label{eq:lift2}
        \pi \Big(\bigcap_{r \in J_{\ell}} (M_{\bm{a}, i_r} \cap M_{\bm{a}, j_r})\Big) \cap U = \bigcap_{r \in J_{\ell}} I_{\bm{a}}(\sigma_{i_r}, \sigma_{j_r}) \cap U\,,
    \end{equation}
    for all $\ell$, where $\pi$ is the projection $\pi: \base \times \mathbb{R} \to \base$,
 \end{enumerate}
then
\begin{enumerate}
    \item the set $\bigcap_{r \in J'} I_{\bm{a}}(\sigma_{i_r}, \sigma_{j_r})$ is a manifold for every $J' \subseteq \{1, \ldots, m\}$ and
    \item we have
    \begin{equation}\label{eq:inter_component}
    T_y\Big( \bigcap_{\ell} \bigcap_{r \in J_{\ell}} I_{\bm{a}}(\sigma_{i_r}, \sigma_{j_r})\Big) = \bigcap_{\ell} T_y\Big( \bigcap_{r \in J_{\ell}} I_{\bm{a}}(\sigma_{i_r}, \sigma_{j_r})\Big)
\end{equation}
for every $y \in \bigcap_{r = 1}^m I_{\bm{a}}(\sigma_{i_r}, \sigma_{j_r}) \cap U$.
\end{enumerate}

\end{lemma}
\begin{proof}
It suffices to show that
\begin{equation}\label{eq:inter_component2}
    \Big(\bigcap_{r \in J_{\ell}} I_{\bm{a}}(\sigma_{i_r}, \sigma_{j_r}) \cap U \Big) \pitchfork \Big( \bigcap_{\ell' < \ell} \bigcap_{r \in J_{\ell}}I_{\bm{a}}(\sigma_{i_r}, \sigma_{j_r}) \cap U \Big)
\end{equation}
for all $\ell$ and almost every $\bm{a} \in A$. Informally, what we show first is that at almost every $\bm{a} \in A$, perturbations of $\bm{a}$ produce perturbations of $\bigcap_{r \in J_{\ell}} I_{\bm{a}}(\sigma_{i_r}, \sigma_{j_r}) \cap U$ for each $\ell$. Then at the end of the proof, we apply the fact that transverse intersections are generic.

By Lemmas \ref{lem:reduced_form} and \ref{lem:manifolds}, we may assume that $\bigcap_{r \in J'} I_{\bm{a}}(\sigma_{i_r}, \sigma_{j_r})$ is a manifold for every $J' \subseteq \{1, \ldots, m\}$.  By \eqref{eq:lift2}, we may assume without loss of generality that there is a sequence $k_1 < \cdots < k_c$ such that $j_1 = k_1$ and $i_r = j_{r-1}$ for all $r$ and $j_{r + 1} = i_r$ for all $r$. In other words, we may assume that $\bigcap_{r \in J_\ell} I_{\bm{a}}(\sigma_{i_r}, \sigma_{j_r})$ is of the form 
\begin{equation*}
    I_{\bm{a}}(\sigma_{k_c}, \sigma_{k_{c-1}} )\cap \cdots \cap I_{\bm{a}}(\sigma_{k_3}, \sigma_{k_2}) \cap I_{\bm{a}}(\sigma_{k_2}, \sigma_{k_1})
\end{equation*}
for all $\bm{a}$. The idea is that because the intersection lifts to an intersection of the corresponding manifolds (see \eqref{eq:lift2}), we can pair up the indices however we like.

Define the function $g_i : B \to \mathbb{R}$ by 
\begin{equation*}
    g_i(\bp) := f(\sigma_{k_i}, \bp) - f(\sigma_{k_{i-1}}, \bp)\,.
\end{equation*}
For almost every $\bm{a} \in A$, the quantity $a_{k_i} - a_{k_{i-1}}$ is a regular value of $g_i$ for all $i$, and the set of regular values is open. By the same argument as in the proof of Lemma \ref{lem:manifolds}, we have
\begin{equation*}
    \Big(I_{\bm{a}}(\sigma_{k_c}, \sigma_{k_{c-1}}) \cap U\Big)\pitchfork \Big( \bigcap_{i=2}^{c - 1} I_{\bm{a}}(\sigma_{k_i}, \sigma_{k_{i-1}}) \cap U\Big)
\end{equation*}
for almost every $\bm{a} \in A$. Therefore, $(a_{k_c} - a_{k_{c-1}})$ is a regular value of $g_{k_c} \vert_{\bigcap_{i=2}^{c - 1} I_{\bm{a}}(\sigma_{k_i}, \sigma_{k_{i-1}}) \cap U}$ by \Cref{lem:RV_restriction}. Additionally, for $\bm{\epsilon} \in \mathbb{R}^N$ such that $\epsilon_{k_c}$ and $\epsilon_{k_{c-1}}$ are sufficiently small, there are no critical values between $(a_{k_c} - a_{k_{c-1}})$ and $(a_{k_c} - a_{k_{c-1}} + \epsilon_{k_c} - \epsilon_{k_{c-1}})$. Because there are no critical values, the set $\bigcap_{r \in J_\ell} I_{\bm{a}}(\sigma_{i_r}, \sigma_{j_r}) \cap U$ (which is the $(a_{k_c} - a_{k_{c-1}})$-level set of $g_k \vert_{\bigcap_{i=2}^{c - 1} I_{\bm{a}}(\sigma_{k_i}, \sigma_{k_{i-1}}) \cap U}$) is a submanifold of $B$ that is diffeomorphic to $I_{\bm{a} + \bm{\epsilon}}(\sigma_{k_c}, \sigma_{k_{c-1}}) \cap \Big( \bigcap_{i=2}^{c - 1} I_{\bm{a}}(\sigma_{k_i}, \sigma_{k_{i-1}})\Big) \cap U$ (which is the $(a_{k_c} - a_{k_{c-1}} + \epsilon_{k_c} - \epsilon_{k_{c-1}})$-level set of $g_k \vert_{\bigcap_{i=2}^{c - 1} I_{\bm{a}}(\sigma_{k_i}, \sigma_{k_{i-1}}) \cap U}$), and these submanifolds are smoothly parameterized by $\epsilon_{k_c}, \epsilon_{k_{c-1}}$.

Now consider any $i_* \in \{2, \ldots, c - 1\}$. By induction on $i_*$, we will show that there is a set $A' \subseteq A$ such that $A \setminus A'$ has measure zero and such that for all $\bm{a} \in A'$, we have that 
\begin{enumerate}
    \item the set $\bigcap_{r \in J_\ell} I_{\bm{a} + \bm{\epsilon}}(\sigma_{i_r}, \sigma_{j_r}) \cap U$ is a submanifold of $B$ that is diffeomorphic to
    \begin{equation*}
        \bigcap_{r \in J_\ell} I_{\bm{a}}(\sigma_{i_r}, \sigma_{j_r}) \cap U
    \end{equation*}
    for sufficiently small $\bm{\epsilon} \in \mathbb{R}^N$, and
    \item these submanifolds are smoothly parameterized by $\bm{\epsilon}$.
\end{enumerate} 

Because $(a_{k_{i_*}} - a_{k_{i_* - 1}})$ is a regular value of $g_{i_*}$ and the set of regular values is open, there are no critical values between $(a_{k_{i_*}} - a_{k_{i_* - 1}})$ and $(a_{k_{i_*}} - a_{k_{i_* - 1}} - \epsilon_{k_{i_* - 1}})$ for sufficiently small $\epsilon_{k_{i_* - 1}}$. Therefore, for sufficiently small $\epsilon_{k_{i_* - 1}}$, the $(a_{k_{i_*}} - a_{k_{i_* - 1}} - \epsilon_{k_{i_* - 1}})$-level set of $g_{i_*}$ is a submanifold of $B$ that is diffeomorphic to the $(a_{k_{i_*}} - a_{k_{i_* - 1}})$-level set, and these submanifolds are smoothly parameterized by (sufficiently small) $\epsilon_{k_{i_* - 1}}$. Because transverse intersections are generic,
\begin{equation*}
    \Big(\bigcap_{i = i_* + 1}^c I_{\bm{a + \epsilon}}(\sigma_{k_i}, \sigma_{k_{i-1}})\Big) \cap \Big( \bigcap_{i=2}^{i_* - 1}I_{\bm{a + \epsilon}}(\sigma_{k_i}, \sigma_{k_{i-1}})\Big) \cap U
\end{equation*}
is transverse to the $(a_{k_{i_*}} - a_{k_{i_* - 1}} - \epsilon_{k_{i_* - 1}})$-level set of $g_{i_*}$ for almost every (sufficiently small) $\epsilon_{k_{i_* - 1}}$. Additionally, if the intersection is transverse, it is transverse for an open neighborhood of $\epsilon_{k_{i_* - 1}}$. Therefore, we can assume without loss of generality that this intersection is transverse at $\epsilon_{k_{i_* - 1}} = 0$ (if not, we can perturb $a_{k_{i_* - 1}}$ so that it is) and for all sufficiently small $\epsilon_{k_{i_* - 1}}$. This implies that $(a_{k_{i_*}} - a_{k_{i_* - 1}} - \epsilon_{k_{i_* - 1}})$ is also a regular value of $g_{i_*}$ restricted to $\Big(\bigcap_{i = i_* + 1}^c I_{\bm{a + \epsilon}}(\sigma_{k_i}, \sigma_{k_{i-1}})\Big) \cap \Big( \bigcap_{i=2}^{i_* - 1}I_{\bm{a + \epsilon}}(\sigma_{k_i}, \sigma_{k_{i-1}})\Big) \cap U$, by \Cref{lem:RV_restriction}. For sufficiently small $\epsilon_{k_{i_*}}$, there are no critical values between $(a_{k_{i_*}} - a_{k_{i_* - 1}} - \epsilon_{k_{i_* - 1}})$ and  $(a_{k_{i_*}} - a_{k_{i_* - 1}} + \epsilon_{k_{i_*}}- \epsilon_{k_{i_* - 1}})$. Therefore, for sufficiently small $\epsilon_{k_{i_*}}$, $\epsilon_{k_{i_* - 1}}$, we have that 
\begin{equation*}
    \Big(\bigcap_{i = i_* + 1}^c I_{\bm{a + \epsilon}}(\sigma_{k_i}, \sigma_{k_{i-1}})\Big) \cap \Big( \bigcap_{i=2}^{i_*}I_{\bm{a + \epsilon}}(\sigma_{k_i}, \sigma_{k_{i-1}})\Big) \cap U\,,
\end{equation*}
which is the $(a_{k_{i_*}} - a_{k_{i_* - 1}})$-level set of $g_{i_*}$ restricted to 
\begin{equation*}
    \Big(\bigcap_{i = i_* + 1}^c I_{\bm{a + \epsilon}}(\sigma_{k_i}, \sigma_{k_{i-1}})\Big) \cap \Big( \bigcap_{i=2}^{i_* - 1}I_{\bm{a + \epsilon}}(\sigma_{k_i}, \sigma_{k_{i-1}})\Big) \cap U\,,
\end{equation*}
is a submanifold of $\base$ that is diffeomorphic to 
\begin{equation*}
      \Big(\bigcap_{i = i_*}^c I_{\bm{a + \epsilon}}(\sigma_{k_i}, \sigma_{k_{i-1}})\Big) \cap \Big( \bigcap_{i=2}^{i_* - 1}I_{\bm{a + \epsilon}}(\sigma_{k_i}, \sigma_{k_{i-1}})\Big) \cap U\,,
\end{equation*}
which is the $(a_{k_{i_*}} - a_{k_{i_* - 1}} + \epsilon_{k_{i_*}}- \epsilon_{k_{i_* - 1}})$-level set of $g_{i_*}$ restricted to
\begin{equation*}
    \Big(\bigcap_{i = i_* + 1}^c I_{\bm{a + \epsilon}}(\sigma_{k_i}, \sigma_{k_{i-1}})\Big) \cap \Big( \bigcap_{i=2}^{i_* - 1}I_{\bm{a + \epsilon}}(\sigma_{k_i}, \sigma_{k_{i-1}})\Big) \cap U\,.
\end{equation*}
These submanifolds are smoothly parameterized by $\epsilon_{k_{i_*}}$ and $\epsilon_{k_{i_* + 1}}$. This concludes the inductive step.

Let $\bm{a} \in A'$, where $A'$ is defined as it was earlier in the proof. We showed above that for sufficiently small $\bm{\epsilon} \in \mathbb{R}^N$, the set of manifolds $\bigcap_{r \in \ell} I_{\bm{a + \epsilon}}(\sigma_{i_r}, \sigma_{j_r}) \cap U$ (parameterized by $\bm{\epsilon}$) is a smoothly parameterized family of embeddings of $\bigcap_{r \in \ell} I_{\bm{a}}(\sigma_{i_r}, \sigma_{j_r}) \cap U$ into $U \subseteq B$. Varying $\{\epsilon_r \mid r \in J_{\ell}\}$ (while holding $\epsilon_r$ constant for $r \not\in J_\ell$) produces a smoothly parameterized family of embeddings of $\bigcap_{r \in \ell} I_{\bm{a}}(\sigma_{i_r}, \sigma_{j_r}) \cap U$ while holding $\Big( \bigcap_{\ell' < \ell} \bigcap_{r \in J_\ell} I_{\bm{a}}(\sigma_{i_r}, \sigma_{j_r}) \cap U\Big)$ constant. Therefore, because transverse intersections are generic,
\begin{equation*}
    \big( \bigcap_{r \in J_{\ell}} I_{\bm{a}}(\sigma_{i_r}, \sigma_{j_r}) \cap U\Big) \pitchfork \Big( \bigcap_{\ell' < \ell} \bigcap_{r \in J_\ell} I_{\bm{a}}(\sigma_{i_r}, \sigma_{j_r}) \cap U\Big)
\end{equation*}
for all $\ell$ for almost every $\bm{\epsilon}$ in a neighborhood of $\bm{0} \in \mathbb{R}^N$. This proves \eqref{eq:inter_component2}, which completes the proof.
\end{proof}

\begin{lemma}\label{lem:intra_component}
    Let $\{ (i_r, j_r) \}_{r = 1}^m$ be a set of index pairs such that $\{i_r, j_r\} \neq \{i_s, j_s\}$ if $r \neq s$. For almost every $\bm{a} \in A$, we have that
    if $U$ is an open set in $B$ such that
    \begin{equation}\label{eq:lift}
        \pi \Big(\bigcap_{r = 1}^m (M_{\bm{a}, i_r} \cap M_{\bm{a}, j_r})\Big) \cap U = \bigcap_{r =1}^m I_{\bm{a}}(\sigma_{i_r}, \sigma_{j_r}) \cap U\,,
    \end{equation}
    where $\pi$ is the projection $\pi: \base \times \mathbb{R} \to \base$,
    then $\bigcap_{r \in J'} I_{\bm{a}}(\sigma_{i_r}, \sigma_{j_r})$ is a manifold for every $J' \subseteq \{1, \ldots, m\}$ and
\begin{equation}\label{eq:eq:intra_component}
        T_y\Big( I_{\bm{a}}(\sigma_{i_1}, \sigma_{j_1}) \cap \cdots I_{\bm{a}}(\sigma_{i_m}, \sigma_{j_m})\Big) = \bigcap_{r = 1}^m T_y \Big(I_{\bm{a}}(\sigma_{i_r}, \sigma_{j_r}) \Big)
    \end{equation}
    for all $y \in \bigcap_{r = 1}^m I_{\bm{a}}(\sigma_{i_r}, \sigma_{j_r}) \cap U$.
\end{lemma}

\begin{proof}
     By Lemmas \ref{lem:reduced_form} and \ref{lem:manifolds}, $\bigcap_{r \in J'} I_{\bm{a}}(\sigma_{i_r}, \sigma_{j_r})$ is a manifold for every $J' \subseteq \{1, \ldots, m\}$ for almost every $\bm{a} \in A$. We have 
    \begin{equation*}
        T_y\Big( I_{\bm{a}}(\sigma_{i_1}, \sigma_{j_1}) \cap \cdots \cap I_{\bm{a}}(\sigma_{i_m}, \sigma_{j_m})\Big) \subseteq \bigcap_{r = 1}^m T_y \Big(I_{\bm{a}}(\sigma_{i_r}, \sigma_{j_r}) \Big)
    \end{equation*}
    because $\bigcap_{r = 1}^m I_{\bm{a}}(\sigma_{i_r}, \sigma_{j_r}) \subseteq I_{\bm{a}}(\sigma_{i_s}, \sigma_{j_s})$ for all $s$.
    
    Let $\bm{v} \in \bigcap_{r = 1}^m T_y \Big(I_{\bm{a}}(\sigma_{i_r}, \sigma_{j_r}) \Big)$. Define $\pi_{[m]} := \pi \vert_{\bigcap_{r = 1}^m (M_{\bm{a}, i_r} \cap M_{\bm{a}, j_r}) \cap \pi^{-1}(U)}$, and define $\pi_r := \pi \vert_{M_{\bm{a}, i_r} \cap M_{\bm{a}, j_r} \cap \pi^{-1}(U)}$ for each $r$. Each $\pi_r$ is a diffeomorphism from $M_{\bm{a}, i_r} \cap M_{\bm{a}, j_r} \cap \pi^{-1}(U)$ to $I_{\bm{a}}(\sigma_{i_r}, \sigma_{j_r}) \cap U$, and $\pi_{[m]}$ is a diffeomorphism from $\bigcap_{r = 1}^m (M_{\bm{a}, i_r} \cap M_{\bm{a}, j_r}) \cap \pi^{-1}(U)$ to $\bigcap_{r = 1}^m I_{\bm{a}}(\sigma_{i_r}, \sigma_{j_r}) \cap U$. Let $\tilde{y} := \pi^{-1}_{[m]}$ (which exists because $\pi^{-1}_{[m]}$ is a diffeomorphism), and let $\tilde{v} := d\pi^{-1}_{[m]}(v)$ (which exists because $d\pi_{[m]}$ is an isomorphism). For all $r$, we have $\tilde{y} = \pi^{-1}_r(y)$ and $\tilde{v} := d\pi^{-1}_r(v)$. Therefore, 
    \begin{equation*}
        \tilde{v} \in \bigcap_{r = 1}^m T_{\tilde{y}}\Big( M_{\bm{a}, i_r} \cap M_{\bm{a}, j_r} \Big)\,.
    \end{equation*}
    By Lemma \ref{lem:manifold_transverse}, we have $T_{\tilde{y}}\Big( \bigcap_{r=1}^m (M_{\bm{a}, {i_r}} \cap M_{\bm{a}, j_r})\Big) = \bigcap_{r = 1}^m T_{\tilde{y}}(M_{\bm{a}, {i_r}} \cap M_{\bm{a}, j_r})$ for all $\tilde{y} \in \bigcap_{r=1}^m (M_{\bm{a}, {i_r}} \cap M_{\bm{a}, j_r})$ for almost every $\bm{a} \in A$, so
    \begin{equation*}
        \tilde{v} \in T_{\tilde{y}}\Big( \bigcap_{r = 1}^m (M_{\bm{a}, i_r} \cap M_{\bm{a}, j_r})\Big)
    \end{equation*}
    for almost every $\bm{a} \in A$. Therefore, $v = d\pi_{[m]}(\tilde{v})$ is in $T_y \Big( \pi_{[m]} \Big( \bigcap_{r = 1}^m (M_{\bm{a}, i_r} \cap M_{\bm{a}, j_r}) \Big)\Big) = T_y\Big( \bigcap_{r = 1}^m I_{\bm{a}}(\sigma_{i_r}, \sigma_{j_r})\Big)$, which implies that 
    \begin{equation*}
        T_y\Big( \bigcap_{r = 1}^m I_{\bm{a}}(\sigma_{i_r}, \sigma_{j_r})\Big) \supseteq \bigcap_{r = 1}^m T_y \Big(I_{\bm{a}}(\sigma_{i_r}, \sigma_{j_r}) \Big)\,.
    \end{equation*}
\end{proof}

The following series of lemmas is used to prove Lemma \ref{lem:locallyfinite}, which shows that $\mathcal{Y}_{\bm{a}}$ is locally finite for almost every $\bm{a} \in A$, and Lemma \ref{lem:frontier}, which shows that $\mathcal{Y}_{\bm{a}}$ satisfies the Axiom of the Frontier for almost every $\bm{a} \in A$. (Recall that $\mathcal{Y}_{\bm{a}}$ is the set of subsets of $\base$ that is defined by \eqref{eq:Ya_def}.)

\begin{lemma}\label{lem:Z_decomp}
      If $\bm{a} \in A$ is such that each $Y \in \mathcal{Y}_{\bm{a}}$ is a manifold, then for any strict partial order $\prec$ on $\K$, there is a unique subset $\mathcal{Y}^{\prec}_{\bm{a}} \subseteq \mathcal{Y}_{\bm{a}}$ such that $Z^{\prec}_{\bm{a}} = \bigcup_{Y \in \mathcal{Y}^{\prec}_{\bm{a}}} Y$, where $Z^{\prec}_{\bm{a}}$ is defined as in \eqref{eq:Zdef}.
\end{lemma}

\begin{proof}
Let $Y \in \mathcal{Y}_{\bm{a}}$ and suppose that $Y \cap Z_{\bm{a}}^{\prec} \neq \emptyset$. This implies that there is a point $\bp \in Y$ such that $\prec_{f_{\bm{a}}(\cdot, \bp)}$ is the same as $\prec$. By Lemma \ref{lem:constantorder_strata}, the simplex order induced by $f$ is constant in $Y$, so $Y \subseteq Z_{\bm{a}}^{\prec}$.
\end{proof}

\begin{lemma}\label{lem:uniqueY}
Let $\prec$ be a strict partial order on the simplices in $\K$. Let $\bm{a} \in A$ be such that
\begin{enumerate}
\item every $Y \in \mathcal{Y}_{\bm{a}}$ is a manifold, where $\mathcal{Y}_{\bm{a}}$ is defined as in \eqref{eq:Ya_def}, 
\item $M_{\bm{a}, i} \pitchfork M_{\bm{a}, j}$ for all $i \neq j$, where $M_{\bm{a}, i}$ is defined as in \eqref{eq:Ma_def},
\item $\bigcap_{r=1}^m I_{\bm{a}}(\sigma_{i_r}, \sigma_{j_r})$ is a manifold for all sets $\{(i_r, j_r\}_{r =1}^m$ of index pairs, and
\item we have 
\begin{equation}\label{eq:tangent_condition}
T_{\bp}\Big(\bigcap_{r=1}^m I_{\bm{a}}(\sigma_{i_r}, \sigma_{j_r})\Big) = \bigcap_{r=1}^m T_{\bp} \Big(I_{\bm{a}}(\sigma_{i_r}, \sigma_{j_r})\Big)
\end{equation}
for all sets $\{(i_r, j_r\}_{r =1}^m$ of index pairs and all $\bp \in \bigcap_{r=1}^m I_{\bm{a}}(\sigma_{i_r}, \sigma_{j_r})$.
\end{enumerate}
Let $\mathcal{Y}^{\prec}_{\bm{a}}$ be the unique subset of $\mathcal{Y}_{\bm{a}}$ such that $Z_{\bm{a}}^{\prec} = \bigcup_{Y \in \mathcal{Y}^{\prec}_{\bm{a}}} Y$, which exists by Lemma \ref{lem:Z_decomp}. Then every $\bp \in \base$ has a neighborhood that intersects at most one set $Y \in \mathcal{Y}^{\prec}_{\bm{a}}$.
\end{lemma}
\begin{proof}
    Let $S(\bp) = \{(\sigma_i, \sigma_j) \mid \bp \in I_{\bm{a}}(\sigma_i, \sigma_j)\}$. There is a neighborhood $U_0$ of $\bp$ such that $I(\sigma_i, \sigma_j) \cap U_0 \neq \emptyset$ if and only if $(\sigma_i, \sigma_j) \in S(\bp)$. In a neighborhood of $\bp$, each $I(\sigma_i, \sigma_j)$ is locally diffeomorphic (via the exponential map, for example) to $T_{\bp}(I_{\bm{a}}(\sigma_i, \sigma_j))$, which is an $(n-1)$-dimensional hyperplane. By \eqref{eq:tangent_condition}, these local diffeomorphisms are compatible with each other, so there is a neighborhood $U$ of $\bp$, a set $\{H(\sigma_i, \sigma_j)\}_{(\sigma_i, \sigma_j) \in S(\bp)}$ of hyperplanes in $\mathbb{R}^n$, and a homeomorphism $\phi: U \to \mathcal{B}$, where $\mathcal{B}$ is the open unit $n$-ball, such that
\begin{equation*}
    \phi\Big(\bigcap_{(\sigma_i, \sigma_j) \in S'(\bp)} I_{\bm{a}}(\sigma_i, \sigma_j)\cap U \Big) = \bigcap_{(\sigma_i, \sigma_j) \in S'(\bp)} H(\sigma_i, \sigma_j) \cap \mathcal{B}
\end{equation*}
for all $S'(\bp) \subseteq S(\bp)$. See Figure \ref{fig:U} for intuition, where we illustrate the neighborhood $U$ for a few points $p \in \base$.
\begin{figure}
    \centering
    \includegraphics[width = .7\textwidth]{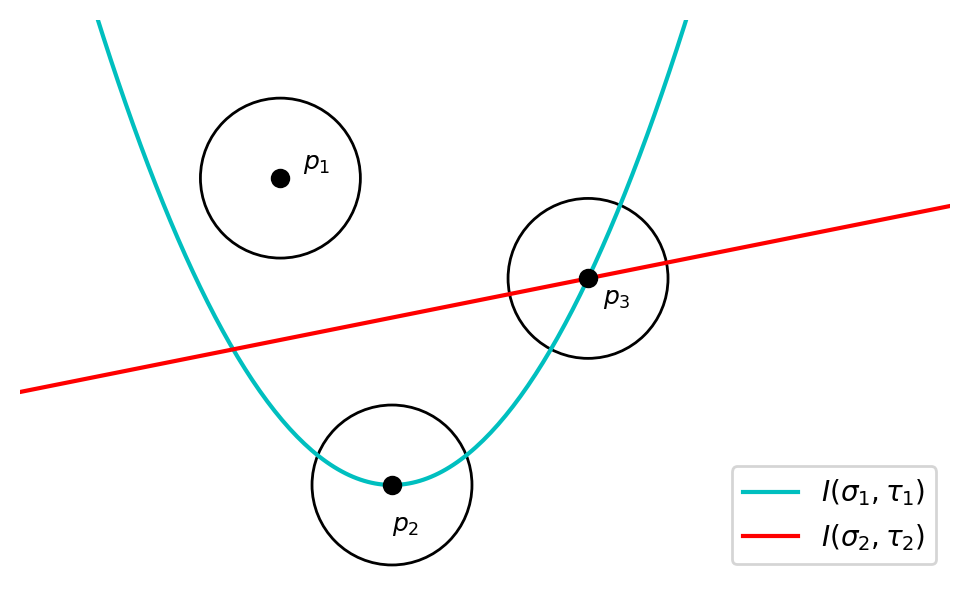}
    \caption{For each point $p_i$, we illustrate the idea behind the homeomorphism $\phi: U_i \to \mathcal{B}$ in the proof of Lemma \ref{lem:uniqueY}, where $U_i$ is a neighborhood of $p_i$ and $\mathcal{B}$ is an open unit ball in $\mathbb{R}^2$. The base space is $\base = \mathbb{R}^2$; we show only the sets $I(\sigma_1, \tau_1)$ and $I(\sigma_2, \tau_2)$, which are curves in the plane, for some simplices $\sigma_1, \sigma_2, \tau_1, \tau_2 \in \K$ and some fibered filtration function $f: \K \times \base \to \mathbb{R}$.}
    \label{fig:U}
\end{figure}
The hyperplanes induce a stratification of $\mathcal{B}$, with a set $\mathcal{Y}'$ of strata, such that for all $Y' \in \mathcal{Y}'$, we have $Y' = \phi(Y \cap U)$ for some $Y \in \mathcal{Y}_{\bm{a}}$. Because $M_{\bm{a}, i} \pitchfork M_{\bm{a}, j}$ for all $i \neq j$, we have that $\mathcal{B} \setminus H(\sigma_i, \sigma_j)$ is the disjoint union of open sets $W_1(\sigma_i, \sigma_j)$ and $W_2(\sigma_i, \sigma_j)$ such that 
\begin{align*}
    f_{\bm{a}}(\sigma_i, \bp') &< f_{\bm{a}}(\sigma_j, \bp') \qquad \text{for all } \bp' \in \phi^{-1}(W_1(\sigma_i, \sigma_j))\,,\\
    f_{\bm{a}}(\sigma_j, \bp') &< f_{\bm{a}}(\sigma_i, \bp') \qquad \text{for all } \bp' \in \phi^{-1}(W_2(\sigma_i, \sigma_j))
\end{align*}
for all $(\sigma_i, \sigma_j) \in S(\bp)$. Suppose that $u_1$ and $u_2$ are points in $U$ such that $\prec_{f_{\bm{a}}(\cdot, u_1)}$ and $\prec_{f_{\bm{a}}(\cdot, u_2)}$ are both the same as $\prec$. For each $(\sigma_i, \sigma_j) \in S(\bp)$, define the set
\begin{equation}\label{eq:Vsigmatau_def}
    V(\sigma_i, \sigma_j) := \begin{cases}
        H(\sigma_i, \sigma_j)\,, &  \sigma_i \not\prec \sigma_j \text{ and } \sigma_j \not\prec \sigma_i \\
        W_1(\sigma_i, \sigma_j)\,, & \sigma_i \prec \sigma_j \\
        W_2(\sigma_i, \sigma_j)\,, & \sigma_j \prec \sigma_i\,.
    \end{cases}
\end{equation}
We define
\begin{equation}\label{eq:V_def}
    V := \bigcap_{(\sigma_i, \sigma_j) \in S(\bp)} V(\sigma_i, \sigma_j) \neq \emptyset\,,
\end{equation}
which is a stratum in $\mathcal{Y}'$. Therefore, there is a $Y \in \mathcal{Y}_{\bm{a}}$ such that $Y \cap U = \phi^{-1}(V)$, with $u_1, u_2 \in Y \cap U$. Therefore, $Y$ is the only element of  $\mathcal{Y}^{\prec}_{\bm{a}}$ that $U$ intersects.
\end{proof}

\begin{lemma}\label{lem:Zboundary}
Let $\prec_0$ be a strict partial order on the simplices in $\K$, and define $\mathcal{O}$ to be the set of strict partial orders $\prec$ such that
\begin{enumerate}
    \item if $\sigma \prec_0 \tau$ and $\sigma \neq \tau$, then either we have $\sigma \prec \tau$ or we have $\sigma \not\prec \tau$ and $\tau \not\prec\sigma$,
    \item if $\sigma \not\prec_0 \tau$ and $\tau \not\prec_0 \sigma$, then $\sigma \not\prec \tau$ and $\tau \not\prec \sigma$, and
    \item the strict partial order $\prec$ is not the same as $\prec_0$.
\end{enumerate}
If $\bm{a} \in A$ is such that
\begin{enumerate}
\item every $S \in \overline{E_{\bm{a}}^{n - \ell}}$ is an $\ell$-dimensional smooth submanifold for every $\ell \in \{1, \ldots, n\}$, where $n$ is the dimension of $\base$,
\item $M_{\bm{a}, i} \pitchfork M_{\bm{a}, j}$ for all $i \neq j$,
\item the set $\bigcap_{r = 1}^m I_{\bm{a}}(\sigma_{i_r}, \sigma_{j_r})$ is a manifold for all sets $\{(i_r, j_r)\}_{r = 1}^m$ of index pairs, and
\item $T_y(\bigcap_{r = 1}^m I_{\bm{a}}(\sigma_{i_r}, \sigma_{j_r}) = \bigcap_{r = 1}^m T_y (I_{\bm{a}}(\sigma_{i_r}, \sigma_{j_r}))$ for all sets $\{(i_r, j_r)\}_{r = 1}^m$ of index pairs,
\end{enumerate}
then 
\begin{equation*}
    \partial Z^{\prec_0}_{\bm{a}} = \bigcup_{\prec \text{ in } \mathcal{O}} Z^{\prec}_{\bm{a}}\,.
\end{equation*}
\end{lemma}
\begin{proof}
By Lemma \ref{lem:stratum}, every $Y \in \mathcal{Y}_{\bm{a}}$ is a manifold. By Lemmas \ref{lem:Z_decomp} and \ref{lem:uniqueY}, the sets $Z^{\prec_0}_{\bm{a}}$ and $Z^{\prec}_{\bm{a}}$ (for all $\prec$ in $\mathcal{O}$) are submanifolds of $\base$. 

\vspace{3mm}

\noindent {\bf Case 1:} If $\dim(Z^{\prec_0}_{\bm{a}}) = 0$, then we must have
\begin{equation*}
    Z^{\prec_0}_{\bm{a}} = I_{\bm{a}}(\sigma_{i_1}, \sigma_{j_1}) \cap \cdots \cap I_{\bm{a}}(\sigma_{i_n}, \sigma_{j_n})\,.
\end{equation*}
for some $I_{\bm{a}}(\sigma_{i_1}, \sigma_{j_1}) \cap \cdots \cap I_{\bm{a}}(\sigma_{i_n}, \sigma_{j_n}) \in \overline{E^n_{\bm{a}}}$. If $\prec$ is in $\mathcal{O}$, then there is another pair $(\sigma_{i_{n+1}}, \sigma_{j_{n+1}})$ of distinct simplices such that $\sigma_{i_{n+1}} \not \prec \sigma_{j_{n+1}}$ and $\sigma_{j_{n+1}} \not \prec \sigma_{i_{n+1}}$. Therefore,
\begin{equation*}
    Z^{\prec}_{\bm{a}} \subseteq I_{\bm{a}}(\sigma_{i_1}, \sigma_{j_1}) \cap \cdots \cap I_{\bm{a}}(\sigma_{i_n}, \sigma_{j_n}) \cap I_{\bm{a}}(\sigma_{i_{n+1}}, \sigma_{j_{n+1}})\,,
\end{equation*}
which is an element of $\overline{E^{n+1}_{\bm{a}}}$. By choice of $\bm{a}$, every $S \in \overline{E^{n+1}_{\bm{a}}}$ is empty, so $Z^{\prec}_{\bm{a}} = \emptyset$.

\vspace{3mm}

\noindent {\bf Case 2:} If $\dim(Z^{\prec_0}_{\bm{a}}) \geq 1$, let $\prec$ be any strict partial order in $\mathcal{O}$. Let $\bp \in Z^{\prec}_{\bm{a}}$. 
Let $S(\bp) = \{(\sigma_i, \sigma_j) \mid \bp \in I_{\bm{a}}(\sigma_i, \sigma_j)\}$. By the same argument as in the proof of Lemma \ref{lem:uniqueY}, there is a neighborhood $U$ of $\bp$, a set $\{H(\sigma_i, \sigma_j)\}_{(\sigma_i, \sigma_j) \in S(\bp)}$ of hyperplanes in $\mathbb{R}^n$, and a homeomorphism $\phi: U \to \mathcal{B}$, where $\mathcal{B}$ is the open unit $n$-ball, such that
\begin{equation*}
    \phi\Big(\bigcap_{(\sigma_i, \sigma_j) \in S'(\bp)} I_{\bm{a}}(\sigma_i, \sigma_j)\cap U \Big) = \bigcap_{(\sigma_i, \sigma_j) \in S'(\bp)} H(\sigma_i, \sigma_j) \cap \mathcal{B}
\end{equation*}
for all $S'(\bp) \subseteq S(\bp)$. See Figure \ref{fig:U}.

Because $M_{\bm{a}, i} \pitchfork M_{\bm{a}, j}$ for all $i \neq j$, we have that $\mathcal{B} \setminus H(\sigma_i, \sigma_j)$ is the disjoint union of open sets $W_1(\sigma_i, \sigma_j)$ and $W_2(\sigma_i, \sigma_j)$ such that 
\begin{align*}
    f_{\bm{a}}(\sigma_i, \bp') &< f_{\bm{a}}(\sigma_j, \bp') \qquad \text{for all } \bp' \in \phi^{-1}(W_1(\sigma_i, \sigma_j))\,,\\
    f_{\bm{a}}(\sigma_j, \bp') &< f_{\bm{a}}(\sigma_i, \bp') \qquad \text{for all } \bp' \in \phi^{-1}(W_2(\sigma_i, \sigma_j))
\end{align*}
for all $(\sigma_i, \sigma_j) \in S(\bp)$. For each $(\sigma_i, \sigma_j) \in S(\bp)$, define the set $V(\sigma_i, \sigma_j)$ as in \eqref{eq:Vsigmatau_def}, and define the set $V$ as in \eqref{eq:V_def}. The set $\phi^{-1}(V)$ is a nonempty subset of $U \cap Z^{\prec_0}_{\bm{a}}$. This implies that $\bp$ is a limit point of $Z^{\prec_0}_{\bm{a}}$, so $Z_{\bm{a}}^{\prec_0} \subseteq \overline{Z^{\prec_0}_{\bm{a}}}$. Because $\prec$ is not the same as $\prec_0$, we have that $Z^{\prec}_{\bm{a}} \cap Z^{\prec_0}_{\bm{a}} = \emptyset$. Therefore, $Z_{\bm{a}}^{\prec_0} \subseteq \partial Z^{\prec_0}_{\bm{a}}$ and
\begin{equation*}
    \bigcup_{\prec \text{ in } \mathcal{O}}Z^{\prec}_{\bm{a}} \subseteq \partial Z^{\prec_0}_{\bm{a}}\,.
\end{equation*}
Now suppose that $\bp$ is in the complement of $Z^{\prec_0}_{\bm{a}} \cup \Big(\bigcup_{\prec \text{ in } \mathcal{O}}Z^{\prec}_{\bm{a}}\Big)$. Because $\prec_{f_{\bm{a}}(\cdot, \bp)}$ is not the same as $\prec_0$ or any $\prec$ in $\mathcal{O}$, there is a pair $(\sigma_i, \sigma_j)$ of simplices such that $f(\sigma_i, \bp) < f(\sigma_j, \bp)$ and either we have $\sigma_j \prec_0 \sigma_i$ or we have $\sigma_j \not\prec \sigma_i$ and $\sigma_i \not\prec \sigma_j$. By continuity of $f$, there is a neighborhood $U_{\sigma_i, \sigma_j}$ of $\bp$ such that $f_{\bm{a}}(\sigma_i, \bp') < f_{\bm{a}}(\sigma_j, \bp')$ for all $\bp' \in U_{\sigma_i, \sigma_j}$. Therefore, $U_{\sigma_i, \sigma_j}$ is in the complement of $Z^{\prec_0}_{\bm{a}}$, so $p$ is not in $\overline{Z^{\prec_0}_{\bm{a}}}$. This implies 
\begin{equation*}
    \partial Z^{\prec_0}_{\bm{a}} \subseteq  \bigcup_{\prec \text{ in } \mathcal{O}}Z^{\prec}_{\bm{a}}\,,
\end{equation*}
which completes the proof.
\end{proof}


\end{document}